\documentclass[11pt,a4paper]{amsart}

\usepackage{amsfonts,amssymb,amsmath,amsthm}
\usepackage{mathtools,esint}
\usepackage{lmodern}
\usepackage[noadjust]{cite}
\usepackage{enumitem}
\usepackage{xcolor}
\usepackage[plainpages=false,pdfpagelabels,backref=page]{hyperref}
\hypersetup{
 colorlinks=true,
 linkcolor={cyan!90!black},
 citecolor={magenta},
 urlcolor={green!40!black},
 pdftitle={Weighted Kolmogorov Equations: Transfer Estimates, Local Boundedness, and Harnack Inequalities},
 pdfauthor={Kaj Nystr\"om},
 pdfsubject={Weighted Kolmogorov equations, A2 weights, kinetic regularity, and Harnack inequalities}
}

\newtheorem{thm}{Theorem}[section]
\newtheorem{lem}[thm]{Lemma}
\newtheorem{prop}[thm]{Proposition}
\newtheorem{cor}[thm]{Corollary}

\theoremstyle{definition}
\newtheorem{defn}[thm]{Definition}
\newtheorem{rem}[thm]{Remark}
\numberwithin{equation}{section}

\usepackage{tikz}
\usetikzlibrary{arrows.meta,calc,patterns,decorations.pathreplacing}
\definecolor{KPBlue}{RGB}{35,82,111}
\definecolor{KPGreen}{RGB}{45,100,83}
\definecolor{KPGold}{RGB}{163,115,43}
\tikzset{
  kp/past/.style={draw=KPBlue,densely dashed,line width=.8pt,fill=KPBlue!7},
  kp/future/.style={draw=KPGreen,line width=.85pt,fill=KPGreen!6},
  kp/outer/.style={draw=black!55,dash pattern=on 4pt off 2pt,line width=.65pt},
  kp/arrow/.style={-{Stealth[length=2mm]},draw=black!65,line width=.7pt},
  kp/step/.style={-{Stealth[length=2.2mm]},draw=KPBlue,line width=1pt},
  kp/vertex/.style={draw=black!70,rounded corners=1.5pt,fill=white,
    minimum width=2.05cm,minimum height=1.02cm,inner sep=3pt,
    align=center,font=\small},
  kp/terminal/.style={kp/vertex,double,double distance=1.2pt},
  kp/small/.style={font=\footnotesize,align=center},
  kp/heading/.style={font=\small\bfseries,anchor=west},
  kp/leader/.style={draw=black!65,line width=.5pt},
  kp/axis/.style={-{Stealth[length=1.7mm]},line width=.55pt,draw=black!75}
}

\newcommand{\R}{\mathbb R}

\usepackage{mathrsfs}
\renewcommand{\d}{\,\mathrm{d}}
\newcommand{\K}{\mathcal K}
\newcommand{\Lw}{\mathcal L_w}
\newcommand{\muW}{\mu_w}
\newcommand{\divw}{\operatorname{div}_{w,x}}
\newcommand{\esssup}{\mathop{\mathrm{ess\,sup}}}
\newcommand{\essinf}{\mathop{\mathrm{ess\,inf}}}
\newcommand{\supp}{\operatorname{supp}}
\newcommand{\avgint}{\fint}
\allowdisplaybreaks
\begin{document}

\title[Weighted Kolmogorov Equations]{Weighted Kolmogorov Equations: Transfer Estimates,\\ Local Boundedness, and Harnack Inequalities}
\author{Kaj Nystr\"om}
\address{Department of Mathematics, Uppsala University, Box 480\\
SE-751 06 Uppsala, Sweden}
\email{kaj.nystrom@math.uu.se}
\renewcommand{\subjclassname}{\textup{2020} Mathematics Subject Classification}

\begin{abstract}
Let $m\geq1$, $k\geq0$, $n=m+k$, and write $x=(v,z)\in\R^m\times\R^k$,
where $v$ is the active velocity variable and $z$ is passive with respect
to $Y=v\cdot\nabla_y+\partial_t$.  We consider
\[
 \operatorname{div}_x(A\nabla_xu)-Y(wu)=0,
\]
where $w=w(x)\in A_2(\R^n)$ and the measurable matrix $A$ has ellipticity
and size controlled by $w$.  For arbitrary such weights, including weights
depending on the active variables, we prove a weighted hypoelliptic
transfer estimate, a kinetic Sobolev inequality, local boundedness,
the weak Harnack and Harnack inequalities, local H\"older continuity,
and the strong minimum principle.

The Harnack argument uses a full-variable weighted measure-data
compactness theorem.  An $\mathrm L^p$ velocity-averaging estimate gives
compactness of smooth moments of $wu$, and weighted Poincar\'e
inequalities in all diffusive variables reconstruct $u$.  At a
two-phase limit, testing with non-negative profiles produces transport
directions given by weighted mean velocities. A $\mathrm{BV}$ argument applied to the limiting subsolution
inequality yields the no-jump principle needed for expansion
of positivity and a weighted kinetic ink-spots lemma. The theory includes the weight $|\lambda|^{1-2s}$ in the extension of
Garofalo and Tralli for fractional powers of the Kolmogorov operator.
\end{abstract}

\maketitle


\section{Introduction and statement of main results}\label{Section_0}

Let $m\geq1$ and $k\geq0$, put $n=m+k$, and split the diffusive variables as $x=(v,z)\in\R^m\times\R^k$. With \(y\in\R^m\) denoting the transported variable and \(t\in\R\) denoting
time, we study equations of the form
\begin{equation}
 \operatorname{div}_x(A\nabla_xu)-Y(wu)=0,
 \qquad
 Y:=v\cdot\nabla_y+\partial_t,
 \label{eq:original}
\end{equation}
in $(v,z,y,t)\in\R^m\times\R^k\times\R^m\times\R$.  In view of the form of
\(Y\), we call \(v\) the \emph{active diffusive variable} as it is both
diffusive and coupled to the transport in \(y\).  We
call \(z\) the \emph{passive diffusive variable} as it is affected by the
diffusion but does not occur in the transport field.  The matrix
$A=A(x,y,t)$ is real and measurable, and its ellipticity and size are measured
by $w=w(v,z)\in A_2(\R^n)$, as specified in
\eqref{eq:ellipticity2} below.  Symmetry of $A$ is not assumed, and no continuity of the coefficients or of the weight is assumed.
We use $Y(wu)$ in \eqref{eq:original} in the ordinary distributional
sense.  Since $w$ is independent of $(y,t)$, this is the weak meaning
of the customary notation $wYu$.  Thus the same weight occurs in the
diffusion matrix and in the transport term.  We retain divergence form
throughout, without expanding derivatives of $w$.

The purpose of the paper is to develop an interior regularity theory for
\eqref{eq:original} under these rough weighted ellipticity assumptions.  Our
main objectives are local boundedness, the weak Harnack and Harnack
inequalities, local H\"older continuity, and the strong minimum principle.
Two complementary mechanisms enter the proof.  Weighted Caccioppoli
estimates and a kinetic Sobolev inequality give local boundedness by Moser
iteration.  The weak Harnack inequality is then obtained by a De Giorgi
scheme based on a weighted intermediate-value principle, expansion of
positivity, and a kinetic ink-spots argument.  The principal hypoelliptic
difficulty is that the energy directly controls derivatives only in the
diffusive variables \(x\).  One must transfer part of this regularity to the
transported variables \((y,t)\), and then propagate measure information
from an earlier cylinder to a later one in a manner compatible with the
weighted measure.

All results in the paper allow arbitrary dependence of $w$ on the full
diffusive variable $x=(v,z)$.  The distinction between active and passive
variables records the geometry of the transport field. It imposes no
restriction on the weight.  In particular, when $k=0$ we obtain the
weighted theory with $w=w(v)\in A_2(\R^m)$.  The structural invariance
used throughout is $Yw=0$, which follows from independence of $(y,t)$.

\subsection{Structural assumptions}

Throughout the paper, $P=(x,y,t)=(v,z,y,t)
 \in\R^m\times\R^k\times\R^m\times\R$, $n=m+k$.  A weight $w$ is
positive almost everywhere, with $w,w^{-1}\in\mathrm L^1_{\rm loc}(\R^n)$.
We assume
\begin{equation}
 w\in A_2(\R^n),
 \qquad
 [w]_{A_2}:=\sup_{\mathcal B}
 \left(\frac1{|\mathcal B|}\int_{\mathcal B}w\d x\right)
 \left(\frac1{|\mathcal B|}\int_{\mathcal B}w^{-1}\d x\right)\leq M,
 \label{eq:A2}
\end{equation}
where the supremum is over Euclidean balls $\mathcal B\subset\R^n$.
The matrix $A$ is real, measurable, and $n\times n$.  We assume that
$0<\lambda\leq\Lambda<\infty$ and
\begin{equation}
 \lambda w(x)|\xi|^2
 \leq \xi\cdot A(P)\xi,
 \qquad
 |A(P)\xi\cdot\zeta|
 \leq\Lambda w(x)|\xi||\zeta|
 \label{eq:ellipticity2}
\end{equation}
for almost every $P$ and all $\xi,\zeta\in\R^n$.  Put
$B(P)=A(P)/w(x)$.  Then
\begin{equation}
 \lambda|\xi|^2\leq\xi\cdot B(P)\xi,
 \qquad
 |B(P)\xi\cdot\zeta|\leq\Lambda|\xi||\zeta|.
 \label{eq:Bbounds}
\end{equation}
Symmetry of $B$ is not assumed.  Assumptions
\eqref{eq:A2}-\eqref{eq:ellipticity2} are in force throughout the paper.

With $A=wB$, equation \eqref{eq:original} takes the form
\begin{equation}
 \operatorname{div}_x\bigl(w(x)B(P)\nabla_xu\bigr)-Y(wu)=0.
 \label{eq:weighted-original}
\end{equation}
In weighted duality this is written as
\[
 \divw(B\nabla_xu)-Yu=0,
 \qquad
 \divw F:=w^{-1}\operatorname{div}_x(wF).
\]
The expression $w^{-1}\operatorname{div}_x(wF)$ is understood through
weighted duality, as specified in Section \ref{Section_3}, rather than
as multiplication of an arbitrary distribution by $w^{-1}$.
Expanding $\divw$ would introduce $\nabla_x\log w$, which is not an
admissible lower-order coefficient for a general $A_2$ weight.  We use
the weighted divergence formulation and the measure
\begin{equation}
 \d\muW(P):=w(x)\d x\d y\d t
 =w(v,z)\d v\d z\d y\d t.
 \label{measure}
\end{equation}
Since $Yw=0$ and the vector field associated with $Y$ has zero
Lebesgue divergence, $Y$ is formally skew-adjoint with respect to $\d\muW$.
Diffusion and transport therefore fit in one energy formulation without
differentiating the weight.

\subsection{Main results}

Let
\[
 P_0=(x_0,y_0,t_0)=(v_0,z_0,y_0,t_0),
 \qquad x_0=(v_0,z_0),
\]
and let \(r>0\).  We define the backward intrinsic cylinder
\begin{equation}
 Q_r(P_0)
 :=
 \bigl\{(x,y,t):
 |x-x_0|<r,\
 |y-y_0-(t-t_0)v_0|<r^3,\
 t_0-r^2<t<t_0
 \bigr\}.
 \label{eq:cylinder}
\end{equation}
When \(P_0=0\), we write \(Q_r:=Q_r(0)\).  The fixed past and future
reference boxes are
\begin{align}
 Q^-&=
 B_{1/4}^{v}\times B_{1/4}^{z}
 \times B_{1/64}^{y}\times(-3/4,-1/2),
 \label{eq:Qminus}\\
 Q^+&=
 B_{1/4}^{v}\times B_{1/4}^{z}
 \times B_{1/64}^{y}\times(-1/4,-1/8).
 \label{eq:Qplus}
\end{align}
When \(k=0\), the factor \(B_{1/4}^{z}\) is omitted.  Their images under
the intrinsic translation and dilation \(T_{P_0,r}\) defined in
\eqref{eq:normalization} are denoted by \(Q_r^-(P_0)\) and
\(Q_r^+(P_0)\), respectively.  Here and below,
\(B_r^x\), \(B_r^v\), \(B_r^z\), and \(B_r^y\) denote the open Euclidean
balls of radius \(r\), centered at the origin, in
\(\R_x^n\), \(\R_v^m\), \(\R_z^k\), and \(\R_y^m\), respectively.

We fix a structural enlargement factor \(R_*>2\), sufficiently large that
the reference boxes and the finitely many normalized enlargements used in
the intermediate-value and measure-to-point arguments are compactly
contained in \(Q_{R_*}\).  In the covering and stacking arguments, the
admissible radii are chosen so that the corresponding enlarged cylinders
and forward stacks remain in \(Q_{R_*}\).  The precise value of \(R_*\) is
irrelevant. It depends only on the fixed normalized geometry and is
preserved under intrinsic translations and dilations.

For the local regularity statement, we also fix the symmetric normalized
neighborhood
\[
 \mathcal W:=B_1^x\times B_1^y\times(-1,1),
 \qquad
 \mathcal W_r(P_0):=T_{P_0,r}(\mathcal W).
\]

Energy solutions and subsolutions are understood in the sense of
Definition \ref{def:energysolution}, and $\avgint_E$ denotes integration
divided by $\muW(E)$.  Our first main result is local boundedness under
the general weighted ellipticity assumptions.

\begin{prop}
\label{prop:localboundedness}
Assume \eqref{eq:A2} and \eqref{eq:ellipticity2}.  Let \(u\geq0\) be an
energy subsolution of \eqref{eq:original} in \(Q_R(P_0)\).  If
\(0<\rho<R\) and \(p>0\), then
\begin{equation}
 \esssup_{Q_\rho(P_0)}u
 \leq
 C_{p,\rho/R}
 \left(
 \avgint_{Q_R(P_0)}u^p\d\muW
 \right)^{1/p},
 \label{eq:localboundedness}
\end{equation}
where \(C_{p,\rho/R}\) depends only on
\(m,k,M,\lambda,\Lambda,p\), and \(\rho/R\).
\end{prop}

The same structural assumptions suffice for the Harnack inequality.

\begin{thm}
\label{thm:main}
Assume \eqref{eq:A2} and \eqref{eq:ellipticity2}.  There exists
\(C=C(m,k,M,\lambda,\Lambda)<\infty\) such that the following holds.  If
\(u\geq0\) is an energy solution of \eqref{eq:original} in an open set
containing \(\overline{Q_{R_*r}(P_0)}\), then
\begin{equation}
 \esssup_{Q_r^-(P_0)}u
 \leq
 C\essinf_{Q_r^+(P_0)}u.
 \label{eq:harnack}
\end{equation}
\end{thm}

\begin{cor}
\label{cor:holder}
Assume \eqref{eq:A2} and \eqref{eq:ellipticity2}.  Every energy solution of
\eqref{eq:original} is locally H\"older continuous in the intrinsic
Kolmogorov metric.  More precisely, there exist
\(\alpha\in(0,1)\) and \(C<\infty\), depending only on
\(m,k,M,\lambda,\Lambda\), such that
\begin{equation}
 \operatorname*{ess\,osc}_{\mathcal W_\rho(P_0)}u
 \leq
 C\left(\frac{\rho}{r}\right)^\alpha
 \operatorname*{ess\,osc}_{\mathcal W_r(P_0)}u,
 \qquad 0<\rho\leq r,
 \label{eq:holder-oscillation}
\end{equation}
whenever \(\overline{\mathcal W_r(P_0)}\) is contained in the solution
domain.
\end{cor}

\begin{rem}
For the locally H\"older continuous representative supplied by Corollary
\ref{cor:holder}, the essential extrema in \eqref{eq:harnack} may be
replaced by pointwise extrema.
\end{rem}

\begin{cor}
\label{cor:minimum}
Assume \eqref{eq:A2} and \eqref{eq:ellipticity2}.  Let \(u\geq0\) be an energy solution of
\eqref{eq:original} in a connected open set \(\Omega\).  If
\(u(P_0)=0\) at an interior point \(P_0\in\Omega\), then \(u\) vanishes in
the closure, relative to \(\Omega\), of the backward attainable set of
\(P_0\).  This attainable set consists of the points reached from \(P_0\)
by absolutely continuous curves contained in \(\Omega\) and satisfying
\[
 \dot t=-\alpha,
 \qquad
 \dot y=-\alpha v,
 \qquad
 \dot v=\omega_v,
 \qquad
 \dot z=\omega_z,
 \qquad
 \alpha\geq0
\]
almost everywhere, where \(\alpha,\omega_v,\omega_z\) are piecewise bounded
controls.
\end{cor}

\subsection{Previous work and outline of the proof}

For $w\equiv1$, local boundedness, H\"older regularity, and Harnack
inequalities for rough Kolmogorov equations have been developed by
potential-theoretic, Moser, De Giorgi, and logarithmic methods, see
\cite{PolidoroRagusa2001,PascucciPolidoro2004,
GolseImbertMouhotVasseur2019,GuerandMouhot2022,Guerand2023}.  In the
uniformly parabolic setting, Moser's Harnack theory goes back to
\cite{Moser1964}.  In the
elliptic and parabolic settings, the theory of equations whose degeneracy is
controlled by an $A_2$ weight originates in
\cite{FabesKenigSerapioni1982,ChiarenzaSerapioni1984a,
ChiarenzaSerapioni1984b}. For  recent treatments of fundamental solutions
and Gaussian bounds in the weighted parabolic setting, see
\cite{AtaeiNystrom2025,Baadi2026}.  The present problem lies at the intersection of
these two theories, but they cannot simply be superimposed.  For recent
accounts of quantitative De Giorgi methods in the elliptic, parabolic, and
kinetic settings, see \cite{BrigatiMouhot2025,Imbert2026}.

Muckenhoupt weights have also appeared in mixed-norm estimates for kinetic
equations.  In that literature the weight belongs to the function space,
whereas the leading diffusion matrix remains uniformly non-degenerate in the
velocity variables, see, for example, \cite{DongYastrzhembskiy2022}.  In our paper the
weight instead measures the ellipticity of the coefficients themselves.
This distinction is essential: it is what forces the use of weighted
divergence and prevents the equation from being reduced to a uniformly
elliptic kinetic equation by division and expansion.

The first new analytic ingredient is the weighted transfer estimate in
Proposition \ref{prop:transfer}.  For an equation
$Yf=\divw F+g$, weighted $\mathrm L^2$ control of
$f,\nabla_xf,F$, and $g$ gives a positive fractional Sobolev gain in
$(y,t)$.  After Fourier transformation in these variables, the transport
multiplier is $v\cdot\eta+\tau$.  Its resonant region is a thin Euclidean
slab in the full diffusive space, even though the phase is independent
of $z$.  Quantitative $A_\infty$ absolute continuity controls the weighted
mass of this slab, while a reciprocal phase multiplier controls its
complement.  This realizes the regularity-transfer mechanism of
\cite{Hormander1967,Bouchut2002} without differentiating the weight.

Combining this gain with the weighted Sobolev inequality of
Fabes-Kenig-Serapioni gives the kinetic Sobolev inequality in
Theorem \ref{thm:kineticsobolev}, with an integrability exponent
$2\kappa>2$.  The equations for subsolutions may also contain a
non-negative Radon measure.  Section \ref{Section_5} uses an auxiliary
weighted comparison problem to remove this measure before applying the
Sobolev estimate.  Caccioppoli estimates and Moser iteration then give
Proposition \ref{prop:localboundedness}.

The transfer exponent used here is not claimed to be sharp.  The
critical-trajectory methods of
\cite{DietertMouhotNiebelZacher2025,Niebel2026} recover sharp gains in
the unweighted theory. The present iteration needs only $\kappa>1$.
The resonant-slab argument supplies this strict gain under arbitrary
$A_2$ degeneracy in all diffusive variables.

The second principal ingredient is the full-variable weighted
measure-data compactness theorem, Proposition
\ref{prop:weighted-measure-data-compactness}.  Direct translations in
the active variables cannot be controlled by the $A_2$ characteristic
alone. Indeed, if $T_hf(v,z)=f(v+h,z)$, then
\begin{equation}
 \|T_hf\|_{\mathrm L^2(w\d x)}^2
 =\int_{\R^n}|f(v,z)|^2w(v-h,z)\d v\d z,
 \label{eq:active-translation}
\end{equation}
and the $A_2$ condition gives no pointwise comparison between the two
weights.  Section \ref{Section_6} instead adapts the compactness strategy
of \cite[Section~4]{GolseImbertMouhotVasseur2019} through averages in all
diffusive variables.  Both the weights and the coefficients may vary
along the sequence, which is needed for constants uniform over
$[w]_{A_2}\leq M$.

The main distinction from the unweighted argument is that the transport
equation acts on the density $wu$, whereas the diffusive energy controls
$\nabla_xu$.  For a rough $A_2$ weight, this energy bound gives no
corresponding diffusive Sobolev bound for $wu$.  When $w=1$, the density
and the solution coincide.  The weighted argument therefore combines
averaging of the density with reconstruction of the solution through
weighted Poincar\'e estimates, uniformly for varying weights.

For normalized weights $w_j$ and functions $0\leq u_j\leq1$ with uniformly
bounded weighted energy, reverse H\"older puts the densities $w_ju_j$
and their diffusion fluxes in an unweighted $\mathrm L^p$ space for
some $p>1$, which may be close to one.  The general $\mathrm L^p$
framework of DiPerna-Lions-Meyer \cite{DiPernaLionsMeyer1991} becomes crucial here as it provides a
positive regularity gain even with this limited integrability.
Indeed, we derive the required averaging statement with Radon
measure data from \cite[Theorem~2]{DiPernaLionsMeyer1991}, introducing an
auxiliary velocity variable to include time in the stationary transport
framework.  This gives local compactness in $(y,t)$ of
$\int\theta(x)w_j(x)u_j(x,y,t)\d x$ for smooth compactly supported
$\theta$.  Finitely many such weighted averages reconstruct $u_j$ as weighted Poincar\'e inequalities on small balls in the full diffusive
space control the approximation error.  The inverse-weight bounds and
reverse H\"older then give strong local convergence both for Lebesgue
measure and for the varying weighted measures.

This compactness supplies the weighted intermediate-value principle.
If that principle failed, we would replace the normalized supersolutions
$f_j$ by the bounded subsolutions $u_j=1-\min\{f_j,1\}$.
This sequence would retain a lower phase in a past region and an upper
phase in a future region while the intermediate phase disappeared.  Strong convergence produces a
two-valued limit and preserves both endpoint phases.  Its diffusive
Sobolev regularity makes it independent of $x$, so it has the form
$u_\infty(x,y,t)=\chi_E(y,t)$ for a measurable set $E$ in the
transported variables.

Testing the limiting equation with a non-negative diffusive profile
produces a transport direction equal to the profile's weighted mean
velocity.  Concentrating the active support approximates any prescribed
interior velocity.  A finite set of independent directions first gives
$\chi_E\in\mathrm{BV}_{\rm loc}$.  The derivative of this characteristic
function is singular, whereas the tested diffusion flux contributes an
absolutely continuous measure.  Separating these parts and using the
non-negativity of the measure in the subsolution equation shows that
$\chi_E$ is non-increasing along every admissible forward characteristic.
The geometry of the past and future regions contradicts the surviving
endpoint phases.

Section \ref{Section_7} combines this principle with local boundedness
to obtain expansion of positivity and proves a weighted kinetic
ink-spots lemma.  Its slicing argument holds $x$ fixed and uses
free-transport coordinates in $(y,t)$, so it allows full dependence of
the weight on $x$.  Section \ref{Section_8} derives weak Harnack and
combines it with local boundedness to prove Harnack.

\subsection{Fractional powers of the Kolmogorov operator}

An application of the weighted theory concerns fractional powers
of the Kolmogorov operator
\begin{equation}
 \K=\Delta_v-Y
 =\Delta_v-v\cdot\nabla_y-\partial_t.
 \label{eq:model-K}
\end{equation}
Let $0<s<1$, set $a=1-2s\in(-1,1)$, and let
$\{\mathcal P_\tau^{\K}\}_{\tau>0}$ denote the evolution semigroup generated
by $\K$.  Initially for Schwartz data, and subsequently on the
appropriate operator domains, Garofalo and Tralli \cite{GarofaloTralli2021}
use the formula
\begin{equation}
 (-\K)^s f
 =-\frac{s}{\Gamma(1-s)}
 \int_0^\infty
 \frac{\mathcal P_\tau^{\K}f-f}{\tau^{1+s}}\d\tau.
 \label{eq:fractional-K}
\end{equation}
If
\begin{equation}
 U(P,\lambda)
 =\frac{\lambda^{2s}}{4^s\Gamma(s)}
 \int_0^\infty
 e^{-\lambda^2/(4\tau)}
 \mathcal P_\tau^{\K}f(P)
 \frac{\d\tau}{\tau^{1+s}},
 \qquad \lambda>0,
 \label{eq:GT-extension}
\end{equation}
then $U(P,0)=f(P)$ and
\begin{equation}
 \partial_\lambda(\lambda^a\partial_\lambda U)
 +\lambda^a\K U=0.
 \label{eq:GT-extension-equation}
\end{equation}
Moreover,
\begin{equation}
 -d_s\lim_{\lambda\to0^+}
 \lambda^{1-2s}\partial_\lambda U(P,\lambda)
 =(-\K)^sf(P),
 \qquad
 d_s=2^{2s-1}\frac{\Gamma(s)}{\Gamma(1-s)}.
 \label{eq:GT-DN}
\end{equation}
See \cite{CaffarelliSilvestre2007,GarofaloTralli2021}.

Under the one-sided energy and weak conormal assumptions specified in
Corollary \ref{cor:fractional-harnack} below, the equation $(-\K)^sf=0$ permits
an even reflection of \(U\) across
\(\{\lambda=0\}\).  The reflected extension satisfies
\begin{equation}
 \operatorname{div}_{v,\lambda}
 \bigl(|\lambda|^a\nabla_{v,\lambda}U\bigr)
 -Y(|\lambda|^aU)=0
 \label{eq:GT-weighted-form}
\end{equation}
in the weak sense across \(\{\lambda=0\}\).  This is the special case
$z=\lambda$ and $w(v,z)=|z|^a\in A_2(\R^{m+1})$ of
\eqref{eq:original}.  Here the degeneracy is confined to the extension
variable, although the main theorem permits dependence on every
diffusive variable.

We stress that \((-\K)^s\) denotes the fractional power of the full
Kolmogorov operator.  It is different from the velocity-fractional kinetic
operator obtained by replacing \(-\Delta_v\) with \((-\Delta_v)^s\). The two
operators have different extension structures and Harnack theories.
For the globally generated extension semigroup, Garofalo and Tralli
proved sharp pointwise Harnack estimates for $a\geq0$, with a boundary
version valid for $a>-1$, see
\cite[Theorem~5.4]{GarofaloTralli2020}.  Corollary
\ref{cor:fractional-harnack} below concerns a different, local statement: it
applies the weighted variable-coefficient theory developed here to the even
extension of a function satisfying \((-\K)^sf=0\) in an intrinsic cylinder,
under the stated one-sided energy and weak conormal assumptions.

We state the local consequence under explicit extension hypotheses.  Let $\Pi(v,\lambda,y,t):=(v,y,t)$ denote the projection that removes the extension variable.

\begin{cor}
\label{cor:fractional-harnack}
Let $0<s<1$, let $f\in\mathrm{Dom}((-\K)^s)\subset
\mathrm L^2(\R^{2m+1})$ be non-negative on the whole space, and put $a:=1-2s$ and
$\widetilde P_0:=(v_0,0,y_0,t_0)$. Suppose that
\[
 (-\K)^sf=0
 \qquad\hbox{in an open neighborhood of }
 \Pi\bigl(\overline{Q_{R_*r}(\widetilde P_0)}\bigr).
\]
Here the fractional power is the closed $\mathrm L^2$ realization of
\eqref{eq:fractional-K}.  Let $\mathcal O$ be a reflection-invariant
open neighborhood of $\overline{Q_{R_*r}(\widetilde P_0)}$ whose projection
lies in that harmonicity region, and set
$\mathcal O_+=\mathcal O\cap\{\lambda>0\}$.  Assume that the extension
$U$ in \eqref{eq:GT-extension} has trace $f$ in local $\mathrm L^2$
and one-sided local energy
up to $\lambda=0$:
\[
 \int_{K\cap\mathcal O_+}
 (|U|^2+|\nabla_vU|^2+|\partial_\lambda U|^2)
 \lambda^a\d v\d\lambda\d y\d t<\infty
 \qquad(K\Subset\mathcal O).
\]
Assume also that \eqref{eq:GT-DN} holds as a weak conormal identity, namely
\begin{equation}
 \int_{\mathcal O_+}
 \bigl(\nabla_{v,\lambda}U\cdot\nabla_{v,\lambda}\psi-UY\psi\bigr)
 \lambda^a\d v\d\lambda\d y\d t
 =d_s^{-1}\langle(-\K)^sf,\psi|_{\lambda=0}\rangle
 \label{eq:extension-weak-conormal}
\end{equation}
for smooth tests supported away from the other boundary faces.
Then
\begin{equation}
 \sup_{Q_r^-(\widetilde P_0)\cap\{\lambda=0\}}f
 \leq
 C
 \inf_{Q_r^+(\widetilde P_0)\cap\{\lambda=0\}}f,
 \label{eq:fractional-harnack}
\end{equation}
and \(f\) has a representative locally H\"older continuous in the
intrinsic Kolmogorov metric on
$\Pi(\mathcal O\cap\{\lambda=0\})$.  The extrema in
\eqref{eq:fractional-harnack} refer to this representative.
The constant depends only on \(m\), \(s\), and the structural constants of
the extension equation.
\end{cor}

The corollary assumes the stated one-sided energy and weak conormal
conditions. Its proof by even reflection is given in
Section \ref{subsec:fractional-extension-proof}.

\subsection{Organization of the paper} The kinetic geometry, weighted energy spaces, and consequences of the
$A_2$ condition are introduced in Section \ref{Section_2}.  The transfer
and kinetic Sobolev estimates are proved in Section \ref{Section_3}.
Sections \ref{Section_4} and \ref{Section_5} establish the energy estimates
and local boundedness.

Section \ref{Section_6} proves the measure-data compactness theorem, the
two-phase no-jump lemma, and the weighted intermediate-value principle.
Section \ref{Section_7} begins with the measure-to-point estimate, then
introduces the auxiliary kinetic geometry and develops expansion on
forward stacks and the weighted kinetic ink-spots lemma.
Weak Harnack and Harnack inequalities are proved in Section
\ref{Section_8}, and Section \ref{Section_9} gives H\"older continuity,
the minimum principle, and the fractional application.  Section \ref{Section_10} discusses the scope of the
method and further questions on sharp transfer, lower-order terms, and
weights depending on the transported variables.  Every section uses
the general assumption $w=w(x)\in A_2(\R^n)$.

Appendix \ref{app:positivity-figures} contains the proofs of the
geometric lemmas used in Section \ref{Section_7}, together with
schematic illustrations of the rooted trees of admissible ordered
pairs and the constructions used to propagate positivity.

\section{Kinetic geometry, weak solutions, and weights}\label{Section_2}
We use the measure $\d\muW$ introduced in \eqref{measure}.  For every
measurable set $E$ satisfying $0<\muW(E)<\infty$, we write
\[
 \avgint_E h\d\muW
 :=\frac{1}{\muW(E)}\int_E h\d\muW.
\]
For a ball $B\subset\R^n$ in the full diffusive space, we similarly use the
conventions
\[
 \avgint_B h\d x
 :=\frac{1}{|B|}\int_B h\d x,
 \qquad
 \avgint_B h w\d x
 :=\frac{1}{w(B)}\int_B h w\d x.
\]

\subsection{Translations and dilations}

We use the Kolmogorov group structure associated with
$Y=v\cdot\nabla_y+\partial_t$.  If
$P=(v,z,y,t)$ and
$\widetilde P=(\widetilde v,\widetilde z,\widetilde y,\widetilde t)$, then
\begin{equation}
 \widetilde P\circ P
 =\bigl(\widetilde v+v,\widetilde z+z,
 \widetilde y+y+t\widetilde v,\widetilde t+t\bigr),
 \qquad
 \delta_r(v,z,y,t)=(rv,rz,r^3y,r^2t).
 \label{eq:group-dilations}
\end{equation}
The transformations relevant to \eqref{eq:original} are therefore
\begin{align}
 T_{P_0,r}(\xi_v,\xi_z,\eta,\tau)
 &=P_0\circ\delta_r(\xi_v,\xi_z,\eta,\tau)=\bigl(v_0+r\xi_v,z_0+r\xi_z,
 y_0+r^2\tau v_0+r^3\eta,
 t_0+r^2\tau\bigr).
 \label{eq:normalization}
\end{align}
If $U=u\circ T_{P_0,r}$, then
\begin{equation}
 Y u\circ T_{P_0,r}
 =r^{-2}(\xi_v\cdot\nabla_\eta U+\partial_\tau U),
 \qquad
 \nabla_xu\circ T_{P_0,r}
 =r^{-1}\nabla_{\xi_v,\xi_z}U.
 \label{eq:scalingfields}
\end{equation}
The rescaled weight and coefficient are
\[
 w_{P_0,r}(\xi_v,\xi_z)
 =w(v_0+r\xi_v,z_0+r\xi_z),
 \qquad
 B_{P_0,r}=B\circ T_{P_0,r}.
\]
They satisfy
\[
 [w_{P_0,r}]_{A_2}=[w]_{A_2}
\]
and the same bounds \eqref{eq:Bbounds}.  All estimates may therefore be
proved in fixed unit cylinders. The flow of $Y$ is
\begin{equation}
 \Phi_s(v,z,y,t)=(v,z,y+sv,t+s).
 \label{eq:Yflow}
\end{equation}
It preserves Lebesgue measure and, since $w$ depends only on $x$, it also
preserves $\d\muW$.

\subsection{Weighted energy spaces}

Let \(D\subset\R^n\) be a bounded open set.  We define
\[
 \mathrm H^1_w(D)
 :=
 \left\{
 f\in \mathrm L^2(D,w\d x):
 \nabla_x f\in \mathrm L^2(D,w\d x;\R^n)
 \right\},
\]
equipped with the norm
\[
 \|f\|_{\mathrm H^1_w(D)}^2
 :=
 \int_D\bigl(|f|^2+|\nabla_xf|^2\bigr)w\d x.
\]
The space \(\mathrm H^1_{0,w}(D)\) is the closure of
\(C^\infty_0(D)\) in \(\mathrm H^1_w(D)\), and we let its dual be denoted
\[\mathrm H^{-1}_w(D)
 :=
 \bigl(\mathrm H^1_{0,w}(D)\bigr)^*.
\]

Let \(D_x\subset\R^n\) and \(D_y\subset\R^m\) be bounded open sets, and let
\(I\subset\R\) be a bounded open interval.  On the product cylinder
\(D_x\times D_y\times I\), the kinetic energy class for equations with divergence data consists of
functions \(u\) satisfying
\[
 u\in
 \mathrm L^2\bigl(D_y\times I;\mathrm H^1_w(D_x)\bigr),
 \qquad
 Yu\in
 \mathrm L^2\bigl(D_y\times I;\mathrm H^{-1}_w(D_x)\bigr).
\]
Here $Yu$ is understood in weighted duality:
$\langle Yu,\varphi\rangle_w=-\int uY\varphi\d\muW$.
The corresponding ordinary distribution is $Y(wu)$.
This is the weighted counterpart of the kinetic energy framework used for
rough Kolmogorov equations, compare \cite{AuscherImbertNiebel2025}.
For subsolutions and supersolutions we impose only the diffusive energy
condition and the one-sided distributional formulation below. A Radon measure need not belong to the displayed dual space.

We also fix the meaning of the fractional spaces used below.  For
$0<s<1$ and $G\subset\R^{m+1}$ open, $\mathrm H^s(G)$ denotes the restriction
space of the Bessel-potential space $\mathrm H^s(\R^{m+1})$, equipped with
the quotient norm
\begin{equation}
 \|v\|_{\mathrm H^s(G)}
 :=\inf\bigl\{\|V\|_{\mathrm H^s(\R^{m+1})}:V|_G=v\bigr\},
 \label{eq:Hs-restriction}
\end{equation}
where
\begin{equation}
 \|V\|_{\mathrm H^s(\R^{m+1})}^2
 =\int_{\R^m\times\R}
 (1+|\eta|^2+|\tau|^2)^s
 |\widehat V(\eta,\tau)|^2\d\eta\d\tau.
 \label{eq:Hs-fourier}
\end{equation}
Thus
\[
 \|f\|_{\mathrm L^2(D_x;\mathrm H^s(G);w\d x)}^2
 =\int_{D_x}\|f(x,\cdot,\cdot)\|_{\mathrm H^s(G)}^2w(x)\d x.
\]
On bounded Lipschitz sets this norm is equivalent to the usual fractional
Sobolev norm defined by the Gagliardo seminorm.

\begin{defn}
\label{def:energysolution}
Let $\Omega\subset\R^{n+m+1}$ be open.  A function $u$ is an energy solution of
\eqref{eq:original} in $\Omega$ if $u,\nabla_xu\in \mathrm L^2_{\mathrm{loc}}(\Omega,\d\muW)$ and
\begin{equation}
 \int_\Omega B\nabla_xu\cdot\nabla_x\varphi\d\muW
 -\int_\Omega u\,Y\varphi\d\muW=0
 \label{eq:weakform}
\end{equation}
for every $\varphi\in C^\infty_0(\Omega)$.
A subsolution satisfies ``$\leq0$'' on the left of
\eqref{eq:weakform}, and a supersolution satisfies ``$\geq0$'', for every
non-negative test function.
When $\Omega$ is a bounded cylinder, we require in addition
that $u,\nabla_xu\in\mathrm L^2(\Omega,\d\muW)$. This is the
meaning of an energy solution, subsolution, or supersolution
in such a cylinder below.
\end{defn}

The sign convention is chosen so that
\[
 \Lw u:=\divw(B\nabla_xu)-Y u=0.
\]
Thus $\Lw u\geq0$ means that $u$ is a subsolution.  Since
\begin{equation}
 \int \varphi\,Y\psi\d\muW
 =-\int \psi\,Y\varphi\d\muW,
 \label{eq:Yskew}
\end{equation}
this agrees with \eqref{eq:weakform}.

\subsection{Regularization and nonlinear tests}

The distinction between an equation and a one-sided inequality is relevant
here.  Energy solutions have time traces, whereas a subsolution may have a
downward time jump.  We use the following local form of renormalization.

\begin{lem}
\label{lem:weighted-renormalization}
Suppose that $u,\nabla_xu,F,g\in\mathrm L^2_{\rm loc}(\Omega,\d\muW)$ and
\begin{equation}
 Y(wu)=\operatorname{div}_x(wF)+wg-\nu
 \quad\hbox{in }\mathcal D'(\Omega),
 \label{eq:renormalization-data}
\end{equation}
where $\nu$ is a non-negative locally finite Radon measure.  If $\beta\in C^2(\R)$
has bounded second derivative and $\beta'\geq0$, then
\begin{equation}
 -\int\beta(u)Y\varphi\d\muW
 \leq -\int\beta'(u)F\cdot\nabla_x\varphi\d\muW
       -\int\beta''(u)F\cdot\nabla_xu\,\varphi\d\muW
       +\int\beta'(u)g\varphi\d\muW
 \label{eq:weighted-renormalization}
\end{equation}
for every non-negative $\varphi\in C_0^\infty(\Omega)$.
If $\nu=0$, equality holds and the restrictions on the signs of
$\beta'$ and $\varphi$ are unnecessary.  Smooth truncation gives the corresponding formulas for
powers whenever all terms have the required local
integrability.  If $F=B\nabla_xu$, convex approximation
also gives the positive-part subsolution inequalities.
\end{lem}

\begin{proof}
By localization, it suffices to consider a non-negative
$\varphi\in C^\infty_0(Q)$, where $Q=D_x\times G$,
$G=D_y\times I$, is a product cylinder compactly contained
in a slightly larger product cylinder in $\Omega$.
We convolve in $(y,t)$ with a non-negative smooth kernel of
total mass one, and denote the regularized quantities by
$u_\varepsilon$, $F_\varepsilon$, $g_\varepsilon$, and
$\nu_\varepsilon$.  For sufficiently small $\varepsilon$, the regularized equation reads
\[
 Y(wu_\varepsilon)
 =\operatorname{div}_x(wF_\varepsilon)
  +wg_\varepsilon-\nu_\varepsilon
 \quad\hbox{in }\mathcal D'(Q),
 \qquad \nu_\varepsilon\geq0.
\]
Convolution in $(y,t)$ commutes with
$Y$, because $v$ is independent
of $(y,t)$.  It also commutes with
$\operatorname{div}_x(w\,\cdot)$, since $w=w(x)$ and
differentiation in $x$ commutes with convolution in $(y,t)$. The convolution kernel is non-negative, so $\nu_\varepsilon$
remains a non-negative Radon measure. Moreover, differentiating the convolution places the derivative
on the smooth kernel.  Young's inequality in $(y,t)$, applied for
each fixed $x$ and then integrated against $w(x)\d x$, shows that
the $(y,t)$-derivatives of $u_\varepsilon$ belong to
$\mathrm L^2(Q,\d\muW)$, with bounds that may depend on
$\varepsilon$. Since $v$ is bounded on $D_x$,
this implies $Yu_\varepsilon\in\mathrm L^2(Q,\d\muW)$,
with a bound that may depend on $\varepsilon$.

For every non-negative $\psi\in C^\infty_0(Q)$, the
regularized equation and the sign of $\nu_\varepsilon$ give
\[
 \int_Q (Yu_\varepsilon)\psi\d\muW
 \leq
 -\int_Q F_\varepsilon\cdot\nabla_x\psi\d\muW
 +\int_Q g_\varepsilon\psi\d\muW.
\]
All three integrals define continuous linear functionals
on $\mathrm L^2(G;\mathrm H^1_{0,w}(D_x))$.
Consequently, this inequality extends to every
non-negative, compactly supported test in that space
which can be approximated by non-negative smooth tests.

We apply this observation to
$\psi_\varepsilon=\beta'(u_\varepsilon)\varphi$.
The boundedness of $\beta''$ implies that $\beta'$ is
Lipschitz and has at most linear growth.  The weighted
Sobolev chain rule therefore gives
\[
 \nabla_x\psi_\varepsilon
 =
 \beta''(u_\varepsilon)\nabla_xu_\varepsilon\,\varphi
 +\beta'(u_\varepsilon)\nabla_x\varphi.
\]
Thus $\psi_\varepsilon$ belongs to
$\mathrm L^2(G;\mathrm H^1_{0,w}(D_x))$, has compact
support in $Q$, and is non-negative because
$\beta'\geq0$ and $\varphi\geq0$.
To justify its admissibility, extend it by zero and
convolve with non-negative smooth product kernels.
Mollification in $x$ converges in the weighted Sobolev
norm for $w\in A_2$. This follows from the weighted
$\mathrm L^2$ boundedness of the Hardy-Littlewood
maximal operator, which controls the mollifiers.
Mollification in $(y,t)$ converges in the same norm
because the weight is independent of these variables.
The resulting approximants are non-negative and, for
sufficiently small mollification parameters, compactly
supported in $Q$.  Passing to the limit in the preceding
inequality makes $\psi_\varepsilon$ an admissible test.
In particular, no evaluation of $\nu_\varepsilon$ on
this nonsmooth test is required.

The chain rule in $(y,t)$ gives
$Y\beta(u_\varepsilon)
=\beta'(u_\varepsilon)Yu_\varepsilon$.
These quantities are locally integrable, since
$\beta'$ has at most linear growth.  Integration by
parts in $(y,t)$, together with the displayed formula
for $\nabla_x\psi_\varepsilon$, now yields
\eqref{eq:weighted-renormalization} with
$u,F,g$ replaced by
$u_\varepsilon,F_\varepsilon,g_\varepsilon$.

We next let $\varepsilon\downarrow0$.
Convolution converges strongly in the local weighted
$\mathrm L^2$ spaces for $u$, $\nabla_xu$, $F$, and $g$.
Since $\beta'$ is Lipschitz,
$\beta'(u_\varepsilon)\to\beta'(u)$ strongly in
$\mathrm L^2(Q,\d\muW)$.  Moreover,
\[
 |\beta(s)-\beta(r)|
 \leq C(1+|s|+|r|)|s-r|,
\]
so $\beta(u_\varepsilon)\to\beta(u)$ strongly in
$\mathrm L^1(Q,\d\muW)$ by Cauchy-Schwarz.
The same inequality shows that the products
$\beta'(u_\varepsilon)F_\varepsilon$ and
$\beta'(u_\varepsilon)g_\varepsilon$ converge to their
respective limits in $\mathrm L^1(Q,\d\muW)$.

For the remaining term, put
$H_\varepsilon=F_\varepsilon\cdot\nabla_xu_\varepsilon$
and $H=F\cdot\nabla_xu$.  Strong $\mathrm L^2$
convergence gives $H_\varepsilon\to H$ in
$\mathrm L^1(Q,\d\muW)$.  After passing to a subsequence,
$u_\varepsilon\to u$ almost everywhere, and
\[
 \beta''(u_\varepsilon)H_\varepsilon-\beta''(u)H
 =
 \beta''(u_\varepsilon)(H_\varepsilon-H)
 +\bigl(\beta''(u_\varepsilon)-\beta''(u)\bigr)H.
\]
The first term tends to zero in $\mathrm L^1$ because
$\beta''$ is bounded, and the second does so by dominated
convergence.  This proves
\eqref{eq:weighted-renormalization}.
If $\nu=0$, the regularized equation can be tested with
functions of either sign, and the same calculation
gives equality without sign restrictions on $\beta'$
or $\varphi$.

For power tests, one first uses smooth truncations of
$\beta'$ and then removes the truncation, provided the
terms in the resulting formula have the required local
integrability.  In the applications with
$F=B\nabla_xu$, convex approximations also give the
positive-part inequalities: the term
$\beta''(u)B\nabla_xu\cdot\nabla_xu$ is non-negative
and can be discarded before passing to the limit.
\end{proof}

For an energy subsolution, take $F=B\nabla_xu$ and $g=0$ in Lemma \ref{lem:weighted-renormalization}.
Convex non-decreasing functions preserve the subsolution inequality.
Concave non-decreasing functions preserve the supersolution inequality.
These statements include $(u-\ell)_+$ and $\min\{u,\ell\}$ in the
respective cases.  Nonlinear tests with a terminal time face are obtained
by first integrating the regularized inequality up to an interior time
$\tau$, and then passing to the limit for almost every $\tau$.  The
non-negative terminal energy can be discarded.  This gives the one-sided
energy inequalities used below and requires no continuous time trace for
a measure-data subsolution.

On a bounded diffusive ball $\mathcal B$, the same argument permits
cutoffs reaching $\partial\mathcal B$, provided the weighted energy and
data bounds hold on the whole diffusive ball, $u$ has zero lateral
$\mathrm H^1_w$ trace, and $\beta'(0)=0$.  Indeed, partial convolution in
$(y,t)$ preserves $\mathrm H^1_{0,w}(\mathcal B)$, and the Sobolev chain
rule gives $\beta'(u_\varepsilon)\varphi$ in that space.  After the
non-negative measure has been discarded on smooth tests, the regularized
inequality extends by non-negative $\mathrm H^1_{0,w}$ approximation to
this test.  No evaluation of the measure at the lateral boundary is
involved.  The preceding time-cutoff argument then gives the same
one-sided energy inequality.

For clarity, we also record the trace argument needed for equations.  On
$\mathcal B\times\R^m_y\times I$, suppose that $u$ has zero lateral
$\mathrm H^1_w$ trace and satisfies
$Yu=\divw F+g$ with all energy quantities in $\mathrm L^2$.
Convolution in $y$ alone gives
\[
 \partial_tu_\varepsilon
 =\divw F_\varepsilon+g_\varepsilon-v\cdot\nabla_yu_\varepsilon
 \in\mathrm L^2(I;\mathcal V^*),
 \qquad
 \mathcal V=\mathrm L^2(\R^m_y;\mathrm H^1_{0,w}(\mathcal B)).
\]
The Hilbert-space energy identity therefore applies to $u_\varepsilon$.
The transport term integrates to zero, since $v$ is bounded on
$\mathcal B$. A cutoff in $y$ removes any issue at infinity.
Fix a compact interval $J\Subset I$ and apply this identity to
$q=u_\varepsilon-u_\delta$.  Between any fixed $s\in J$ and $t\in J$,
the absolute change of $\tfrac12\|q(t)\|_{\mathrm L^2_w}^2$ is bounded by
$\|F_\varepsilon-F_\delta\|_2\|\nabla_xq\|_2+
\|g_\varepsilon-g_\delta\|_2\|q\|_2$, where these norms are taken over
$\mathcal B\times\R^m_y\times J$ with the weighted measure.
This bound tends to zero uniformly in $t$ by strong convolution
convergence.  Choose $s$ among the time slices for which
$u(s)\in\mathrm L^2_w$. Convolution also converges in that slice.
It follows that $u_\varepsilon$ is Cauchy in
$C(J;\mathrm L^2_w)$.
Thus $u$ has a continuous $\mathrm L^2_w$ representative and
\begin{equation}
 \frac12\|u(t)\|_{\mathrm L^2_w}^2
 -\frac12\|u(s)\|_{\mathrm L^2_w}^2
 =-\int_s^t\!\int F\cdot\nabla_xu\d\muW
   +\int_s^t\!\int gu\d\muW.
 \label{eq:weighted-green-identity}
\end{equation}
If a zero initial datum is prescribed in the Cauchy weak formulation,
the $y$-regularizations have that same datum, so the argument includes the
initial endpoint.  The local version follows by multiplying $u$ by a
spatial cutoff.  This proves the weighted trace statement used here
directly, compare the unweighted kinetic embedding and Cauchy theory in
\cite{AuscherImbertNiebel2025}.

\subsection{Consequences of the \texorpdfstring{$A_2$}{A2} condition}

We collect the properties of $w$ used in the proof.  Their constants will
be called admissible if they depend only on $n=m+k$ and $M$.

\begin{prop}
\label{prop:Ainfty}
There are admissible constants $D,C_0,C_1>0$ and exponents
$\delta_0,\delta_1>0$ such that
\begin{align}
 w(2B)&\leq D w(B),
 \label{eq:doubling}\\
 C_0^{-1}\left(\frac{|E|}{|B|}\right)^{\delta_0}
 &\leq\frac{w(E)}{w(B)}
 \leq C_1\left(\frac{|E|}{|B|}\right)^{\delta_1}
 \label{eq:Ainftyquant}
\end{align}
whenever $B$ is a ball and $E\subset B$ is measurable.
Consequently, $\muW$ is doubling on intrinsic cylinders and
\begin{equation}
 \muW(Q_r(P_0))\simeq r^{3m+2}w(B_r(x_0)).
 \label{eq:cylvolume}
\end{equation}
\end{prop}

\begin{prop}
\label{prop:FKS}
There exist admissible constants $C$ and $\chi_x>1$ such that
for every Euclidean ball \(B=B_r\subset\R^n\) and every
\(f\in\mathrm H^1_w(2B)\),
\begin{align}
 \avgint_B|f-f_{B,w}|^2w\d x
 &\leq Cr^2\avgint_B|\nabla_xf|^2w\d x,
 \label{eq:weightedpoincare}\\
 \left(\avgint_B|f-f_{B,w}|^{2\chi_x}w\d x\right)^{1/(2\chi_x)}
 &\leq Cr
 \left(\avgint_{2B}|\nabla_xf|^2w\d x\right)^{1/2}.
 \label{eq:weightedsobolev}
\end{align}
Here $f_{B,w}=w(B)^{-1}\int_Bfw\d x$.
\end{prop}

The preceding consequences of the \(A_2\) condition are standard, see
\cite{FabesKenigSerapioni1982,GarciaCuervaRubio1985}.

\begin{prop}
\label{prop:reverse-holder}
Let \(d\geq1\), and suppose that
\[
 w\in A_2(\R^d),
 \qquad
 [w]_{A_2}\leq M.
\]
There exist \(\epsilon=\epsilon(d,M)>0\) and
\(C=C(d,M)<\infty\) such that, for every Euclidean ball \(B\subset\R^d\),
\begin{align}
 \left(\avgint_B w^{1+\epsilon}\d x\right)^{1/(1+\epsilon)}
 &\leq
 C\avgint_B w\d x,
 \label{eq:reverse-holder-w}\\
 \left(\avgint_B w^{-(1+\epsilon)}\d x\right)^{1/(1+\epsilon)}
 &\leq
 C\avgint_B w^{-1}\d x.
 \label{eq:reverse-holder-winverse}
\end{align}
In particular,
\[
 w,\ w^{-1}\in\mathrm{RH}_{1+\epsilon,\mathrm{loc}}(\R^d),
\]
with constants depending only on \(d\) and \(M\).
\end{prop}

\begin{proof}
The first estimate is the reverse H\"older self-improvement for
Muckenhoupt weights.  Since
\[
 w^{-1}\in A_2(\R^d),
 \qquad
 [w^{-1}]_{A_2}=[w]_{A_2},
\]
the same result applied to \(w^{-1}\) gives
\eqref{eq:reverse-holder-winverse}.  See
\cite[Chapter~IV]{GarciaCuervaRubio1985}.
\end{proof}

\section{A weighted kinetic transfer estimate}\label{Section_3}

This section contains the main functional estimate.  We prove it by separating
Fourier space into a thin resonant slab and its complement.  Quantitative
$A_\infty$ absolute continuity controls the weighted mass of the slab, while
the equation controls the complement.  In particular, the proof neither
differentiates $w$ nor invokes a translation-invariant regularizer in the
$x$ variables.

For a vector field $F=(F_1,\dots,F_n)$, define $\divw F$ as a weighted
distribution by
\[
 \langle\divw F,\varphi\rangle_w
 :=-\int F\cdot\nabla_x\varphi\d\muW.
\]
Thus
\[
 Yf=\divw F+g
\]
means
\[
 -\int fY\varphi\d\muW
 =-\int F\cdot\nabla_x\varphi\d\muW
 +\int g\varphi\d\muW
\]
for every test function \(\varphi\).  Equivalently,
\[
 Y(wf)=\operatorname{div}_x(wF)+wg
\]
in the ordinary distributional sense.

\begin{prop}
\label{prop:transfer}
Let $I\Subset I'$ be bounded intervals, let $D_x\Subset D_x'$ be bounded
balls in $\R^n$, and let $D_y\Subset D_y'$ be bounded Lipschitz sets in
$\R^m$.  Suppose
\[
 f,\nabla_xf,F,g\in \mathrm L^2(D_x'\times D_y'\times I',\d\muW)
\]
and
\begin{equation}
 Y f=\divw F+g
 \label{eq:transportdiv}
\end{equation}
in the weighted distributional sense.  Then there are
$s_*=s_*(m,k,M)>0$ and $C<\infty$ such that, for
every $0<s\leq s_*$,
\begin{align}
 \|f\|_{\mathrm L^2(D_x;\mathrm H^s(D_y\times I);w\d x)}
 \leq C\big(&\|f\|_{\mathrm L^2(\d\muW)}
 +\|\nabla_xf\|_{\mathrm L^2(\d\muW)}+\|F\|_{\mathrm L^2(\d\muW)}
 +\|g\|_{\mathrm L^2(\d\muW)}\big),
 \label{eq:transferestimate}
\end{align}
where the norms on the right are over $D_x'\times D_y'\times I'$.
The constant also depends on the fixed nested domains, in particular on
the cutoff bounds and on an upper bound for \(|v|\) in \(D_x'\).  Equivalently,
it depends on their full geometry, not only on the three separation
distances.
\end{prop}

\begin{proof}
Choose a product cutoff
\[
 \chi=\chi_x\chi_y\chi_t
 \in C^\infty_0(D_x'\times D_y'\times I')
\]
which is equal to one on a neighborhood of
\(D_x\times D_y\times I\), and define
\[
 \widetilde f:=\chi f,
 \qquad
 \widetilde F:=\chi F,
 \qquad
 \widetilde g
 :=\chi g+fY\chi-F\cdot\nabla_x\chi.
\]
The product rules for \(Y\) and \(\divw\), used in the weighted
distributional sense, give
\begin{equation}
 Y\widetilde f
 =
 \divw\widetilde F+\widetilde g.
 \label{eq:localized-transfer-equation}
\end{equation}
Indeed,
\[
 Y(\chi f)=\chi Yf+fY\chi,
 \qquad
 \divw(\chi F)=\chi\divw F+F\cdot\nabla_x\chi,
\]
which yields the stated formula from \(Yf=\divw F+g\).  Because \(\chi\)
is compactly supported in
\(D_x'\times D_y'\times I'\), the functions
\(\widetilde f,\widetilde F,\widetilde g\) may be extended by zero to the
whole space.  Their zero extensions still satisfy
\eqref{eq:localized-transfer-equation} in the weighted distributional sense
on \(\R^{n+m+1}\), because they vanish in a neighborhood of the boundary of
the localization cylinder.  Choose a ball
$B_{R/2}(0)\subset\R^n$ containing $\operatorname{supp}\chi_x$, and let
\(B_R\) be the concentric ball with twice the radius.  The zero-extended
localized functions are then supported in \(B_{R/2}\) in the \(x\) variable,
while \(B_R\) is available for the weighted Sobolev estimate. Relabeling \((\widetilde f,\widetilde F,\widetilde g)\) as \((f,F,g)\), we are reduced to the whole-space equation
\begin{equation}
 Yf=\divw F+g
 \qquad\hbox{in the weighted distributional sense},
 \label{eq:globaltransportdiv}
\end{equation}
with \(f,F,g\) compactly supported in \(B_{R/2}\) in the \(x\) variable.
The cutoff terms \(fY\chi\) and \(F\cdot\nabla_x\chi\) are controlled by the
corresponding local \(\mathrm L^2\) norms on
\(D_x'\times D_y'\times I'\), with constants depending on the cutoff
derivatives and hence on the distances between the nested domains.  The
Fourier argument below is valid directly for these energy data as, after
partial Fourier transformation, density in
\(\mathrm H^1_{0,w}(B_R)\) justifies the multiplier tests.

We now establish the weighted thin-slab estimate that replaces the
Lebesgue-measure estimate for resonant sets in classical Fourier proofs of
kinetic transfer of regularity, compare Bouchut's multiplier method
\cite{Bouchut2002}.  Let \(h\in\mathrm H^1_w(B_R)\) satisfy
\(\operatorname{supp}h\subset B_{R/2}\).  In the Fourier argument below,
\(h\) will be the function  $h(x)=\widehat f(x,\eta,\tau)$ for fixed \((\eta,\tau)\).

If $e\in\mathbb S^{m-1}$,
$a\in\R$, and $0<\delta\leq R$, then
\begin{equation}
 S_{e,a,\delta}
 :=\{(v,z)\in B_R^x:|v\cdot e-a|<\delta\}
 \quad\hbox{satisfies}\quad
 \frac{w(S_{e,a,\delta})}{w(B_R)}
 \leq C\left(\frac{\delta}{R}\right)^{\delta_1}.
 \label{eq:weightedslabmass}
\end{equation}
Indeed, slicing the $n$-dimensional ball in the active direction $e$ gives
$|S_{e,a,\delta}|/|B_R^x|\leq C\delta/R$, so
\eqref{eq:weightedslabmass} follows from \eqref{eq:Ainftyquant}.  Weighted
H\"older and Proposition \ref{prop:FKS}, with the mean absorbed into the
$\mathrm L^2$ term, therefore give
\begin{equation}
 \int_{S_{e,a,\delta}}|h|^2w\d x
 \leq C\left(\frac{\delta}{R}\right)^\theta
 \int_{B_R}\bigl(|h|^2+R^2|\nabla_xh|^2\bigr)w\d x,
 \qquad
 \theta:=\delta_1\bigl(1-1/{\chi_x}\bigr)>0.
 \label{eq:weightedslabfunction}
\end{equation}
Indeed, let \(S=S_{e,a,\delta}\), \(B=B_R\), and
\(\chi=\chi_x\).  Weighted H\"older's inequality gives
\begin{align}
 \int_S |h|^2w\d x
 &\leq
 w(S)^{1-1/\chi}
 \left(\int_B|h|^{2\chi}w\d x\right)^{1/\chi}=
 w(B)
 \left(\frac{w(S)}{w(B)}\right)^{1-1/\chi}
 \left(\avgint_B|h|^{2\chi}w\d x\right)^{1/\chi}.
 \label{eq:weightedslabholder}
\end{align}
To estimate the last factor, write
\(h=(h-h_{B,w})+h_{B,w}\). By Proposition
\ref{prop:FKS} and Cauchy-Schwarz we obtain
\begin{align}
 \left(\avgint_B|h|^{2\chi}w\d x\right)^{1/(2\chi)}
 &\leq
 \left(\avgint_B|h-h_{B,w}|^{2\chi}w\d x\right)^{1/(2\chi)}
 +|h_{B,w}|
 \notag\\
 &\leq
 CR\left(\avgint_{2B}|\nabla_xh|^2w\d x\right)^{1/2}
 +
 \left(\avgint_B|h|^2w\d x\right)^{1/2}.
 \label{eq:weightedslabsobolev}
\end{align}
Thus the contribution of the weighted mean is absorbed into the
\(\mathrm L^2\) term.  Squaring \eqref{eq:weightedslabsobolev}, inserting
the result into \eqref{eq:weightedslabholder}, and using
\eqref{eq:weightedslabmass}, we find \eqref{eq:weightedslabfunction}.
Here a fixed enlargement of \(B_R\), arising from the occurrence of
\(2B\) in Proposition \ref{prop:FKS}, is harmless as after the cutoff
reduction, \(h\) is supported away from the boundary of \(B_R\), and
the radius may be enlarged and relabelled.

Take the Fourier transform of \eqref{eq:globaltransportdiv} in $(y,t)$ and
write $(\eta,\tau)$ for the dual variables.  For almost every
$(\eta,\tau)$, with
\[
 h=\widehat f(\cdot,\eta,\tau),\qquad
 H=\widehat F(\cdot,\eta,\tau),\qquad
 k=\widehat g(\cdot,\eta,\tau),
\]
we have
\begin{equation}
 iT h=\divw H+k,
 \qquad T(x)=T(v,z):=v\cdot\eta+\tau.
 \label{eq:fouriertransport}
\end{equation}
Set
\begin{equation}
 \mathcal E(\eta,\tau)
 :=
 \|h\|_{\mathrm L^2(B_R,w\d x)}^2
 +\|\nabla_xh\|_{\mathrm L^2(B_R,w\d x)}^2
 +\|H\|_{\mathrm L^2(B_R,w\d x)}^2
 +\|k\|_{\mathrm L^2(B_R,w\d x)}^2.
 \label{eq:frequencyenergy}
\end{equation}

Suppose first that $\kappa:=|\eta|\geq1$, put $e=\eta/\kappa$, and fix
$0<\delta\leq\min\{1,R/2\}$.  Let $\beta\in C^\infty(\R)$ satisfy
$0\leq\beta\leq1$, $\beta=0$ on $[-1,1]$, and $\beta=1$ outside
$[-2,2]$.  Define
\[
 \psi(x)=\beta\left(\frac{T(x)}{\kappa\delta}\right),
 \qquad
 m(x)=\frac{\psi(x)^2}{T(x)},
\]
where $m$ is set equal to zero on the region on which $\psi=0$.  Then
\begin{equation}
 |m|\leq\frac{C}{\kappa\delta},
 \qquad
 |\nabla_xm|\leq\frac{C}{\kappa\delta^2}.
 \label{eq:multiplierbounds}
\end{equation}
Use $-im\overline h$ as a test function in \eqref{eq:fouriertransport} and
take the real part of the resulting complex identity.  The definition of
$\divw$, not an expansion of it, yields
\begin{align}
 \int_{B_R}\psi^2|h|^2w\d x
 &\leq \int_{B_R}|H|\,|\nabla_x(m\overline h)|w\d x
       +\int_{B_R}|k|\,|m h|w\d x \notag\\
 &\leq C\left(\frac1{\kappa\delta}
              +\frac1{\kappa\delta^2}\right)\mathcal E(\eta,\tau)
 \leq \frac{C}{\kappa\delta^2}\mathcal E(\eta,\tau).
 \label{eq:nonresonantestimate}
\end{align}
Complex-valued testing is justified by applying the weak formulation to real
and imaginary parts.  Crucially, \eqref{eq:multiplierbounds} differentiates
only $m$, never $w$.

Since $\psi=1$ when $|T|\geq2\kappa\delta$, we have \[
 \operatorname{supp}(1-\psi^2)
 \subset
 \left\{
 (v,z)\in B_R:
 \left|v\cdot e+{\tau}/{\kappa}\right|<2\delta
 \right\}
 =
 S_{e,-\tau/\kappa,2\delta},
\]
and hence the remaining part of $h$ is
supported in the slab $|v\cdot e+{\tau}/{\kappa}|<2\delta$. Consequently,
\begin{align}
 \|h\|_{\mathrm L^2(B_R,w\d x)}^2
 &=
 \int_{B_R}\psi^2|h|^2w\d x
 +
 \int_{B_R}(1-\psi^2)|h|^2w\d x
 \notag\\
 &\leq
 \int_{B_R}\psi^2|h|^2w\d x
 +
 \int_{S_{e,-\tau/\kappa,2\delta}}|h|^2w\d x.
 \label{eq:frequencybalance-}
\end{align}
Combining \eqref{eq:weightedslabfunction} and
\eqref{eq:nonresonantestimate} gives
\begin{equation}
 \|h\|_{\mathrm L^2(B_R,w\d x)}^2\leq
 C\left[
 \left(\frac{2\delta}{R}\right)^\theta
 +\frac{1}{\kappa\delta^2}
 \right]\mathcal E(\eta,\tau)\leq
 C\left(
 \delta^\theta+\frac{1}{\kappa\delta^2}
 \right)
 \mathcal E(\eta,\tau).
 \label{eq:frequencybalance}
\end{equation}
For all sufficiently large $\kappa$, choose
$\delta=\kappa^{-1/(\theta+2)}$.  Bounded frequencies are absorbed by the
\(\|h\|_{\mathrm L^2(B_R,w\d x)}^2\) term in
\eqref{eq:frequencyenergy}.  Consequently,
\begin{equation}
 (1+|\eta|)^\alpha
 \|h\|_{\mathrm L^2(B_R,w\d x)}^2
 \leq C\mathcal E(\eta,\tau),
 \qquad
 \alpha:=\frac{\theta}{\theta+2}>0.
 \label{eq:etagain}
\end{equation}

It remains to control the time frequency.  If \(|\tau|\leq2R|\eta|\), then
\[
 1+|\eta|^2+|\tau|^2
 \leq
 1+(1+4R^2)|\eta|^2
 \leq
 C_R(1+|\eta|)^2.
\]
Set \(s_*:=\alpha/2\).  Then
\[
 (1+|\eta|^2+|\tau|^2)^{s_*}
 \leq
 C_R(1+|\eta|)^\alpha,
\]
and hence the required estimate follows from \eqref{eq:etagain}.  If
\(|\tau|>2R|\eta|\), then
\[
 |T(v,z)|=|v\cdot\eta+\tau|
 \geq |\tau|-R|\eta|
 >\frac{|\tau|}{2}
 \qquad\hbox{on }B_R.
\]
Testing \eqref{eq:fouriertransport} with
\(-i\overline h/T\) and taking the real part, we obtain an estimate analogous
to \eqref{eq:nonresonantestimate}.  Indeed, on \(B_R\),
\[
 \left|\frac{1}{T}\right|
 \leq\frac{2}{|\tau|},
 \qquad
 \left|\nabla_x\!\left(\frac{1}{T}\right)\right|
 =
 \frac{|\eta|}{|T|^2}
 \leq\frac{C_R}{|\tau|},
\]
where we used \(|\tau|>2R|\eta|\).  The definition of \(\divw\) therefore
gives
\[
 \int_{B_R}|h|^2w\d x
 \leq
 \frac{C}{|\tau|}\mathcal E(\eta,\tau)
\]
whenever \(|\tau|\) is sufficiently large.  Enlarging \(C\) to include the
bounded range of \(|\tau|\), which is controlled by the
\(\|h\|_{\mathrm L^2(B_R,w\d x)}^2\) term in
\(\mathcal E(\eta,\tau)\), we conclude that, when \(|\tau|>2R|\eta|\),
\begin{equation}
 \|h\|_{\mathrm L^2(B_R,w\d x)}^2
 \leq
 \frac{C}{1+|\tau|}
 \mathcal E(\eta,\tau).
 \label{eq:taugain}
\end{equation}
In the complementary sector \(|\tau|\leq2R|\eta|\),
\eqref{eq:etagain} controls both frequency variables.
Since \(0<\alpha<1\), \eqref{eq:etagain} and \eqref{eq:taugain} imply
\begin{equation}
 (1+|\eta|^2+|\tau|^2)^{s_*}
 \|h\|_{\mathrm L^2(B_R,w\d x)}^2
 \leq C\mathcal E(\eta,\tau),
 \qquad
 s_*=\frac{\theta}{2(\theta+2)}.
 \label{eq:frequencyestimate}
\end{equation}
The case $\eta=0$ is included in the nonresonant $\tau$ estimate.  Integrating
\eqref{eq:frequencyestimate} in $(\eta,\tau)$ and applying Plancherel's
theorem proves the whole-space estimate.  Restriction to
$D_x\times D_y\times I$ and the initial cutoff reduction prove
\eqref{eq:transferestimate}.  The multiplier tests were justified directly
by weighted Sobolev density, so no regularization of $w$ is needed.
\end{proof}

\begin{rem}
The exponent $s_*$ is not claimed to be sharp.  When $w\equiv1$, Bouchut's
multiplier method gives the sharp gain of one third of a derivative in the
transported variable \cite{Bouchut2002}.  The same sharp gain has recently
been recovered in physical space by mollification along critical kinetic
trajectories \cite{DietertMouhotNiebelZacher2025,Niebel2026}.  The
resonant-slab proof above deliberately trades this sharpness for robustness
under arbitrary $A_2$ degeneracy.  The subsequent Moser iteration uses only
the strict positivity of $s_*$, so no sharp transfer exponent is needed for
local boundedness.  The weak Harnack argument developed in Sections
\ref{Section_6}-\ref{Section_8} uses instead the compactness consequences of
the diffusive energy.
\end{rem}

\begin{rem}
In the unweighted theory, the gain of integrability needed for
iteration can also be obtained from the explicit fundamental
solution of the Kolmogorov operator and estimates for its
derivatives, see \cite[Section~2.2]{GuerandMouhot2022}.
For general $A_2$ weights, the corresponding model operator is
$\operatorname{div}_x(w\nabla_x\,\cdot)-Y(w\,\cdot)$.
Adapting this approach would require kernel estimates that
control the divergence-form source and yield a weighted gain
of integrability, with constants depending only on the
structural parameters.  The existence of a fundamental
solution alone would not suffice, and we are not aware of
the required estimates under the present assumptions. Developing such a fundamental-solution theory is an interesting
open problem.  It would extend the weighted parabolic results
of \cite{AtaeiNystrom2025,Baadi2026} to a setting in which
the degenerate diffusion interacts with the kinetic transport
term $v\cdot\nabla_y$.  The argument above establishes
regularity transfer directly from the weighted energy bounds,
without requiring fundamental-solution estimates.
\end{rem}

Combining Proposition \ref{prop:transfer} with the weighted Sobolev inequality
in $x$ gives a genuine gain of integrability in all variables.

\begin{thm}
\label{thm:kineticsobolev}
Under the assumptions of Proposition \ref{prop:transfer}, there are
$\kappa>1$ and $C<\infty$, depending only on the indicated sets, $m,k$, and
$M$, such that, with
\(\mathcal V:=D_x\times D_y\times I\) and
\(\mathcal V':=D_x'\times D_y'\times I'\),
\begin{align}
 \left(
 \avgint_{\mathcal V}|f|^{2\kappa}\d\muW
 \right)^{1/(2\kappa)}
 \leq C\bigg[&
 \left(
 \avgint_{\mathcal V'}|f|^2\d\muW
 \right)^{1/2}
 +\left(
 \avgint_{\mathcal V'}|\nabla_xf|^2\d\muW
 \right)^{1/2}
 \notag\\
 &+\left(
 \avgint_{\mathcal V'}|F|^2\d\muW
 \right)^{1/2}
 +\left(
 \avgint_{\mathcal V'}|g|^2\d\muW
 \right)^{1/2}\bigg].
 \label{eq:kineticsobolev}
\end{align}
\end{thm}

\begin{proof}
Since the asserted estimate is invariant under multiplication of the weight
by a positive constant, we may normalize the weighted measure of the outer
diffusive ball.  We also normalize Lebesgue measure on the fixed outer
\((y,t)\)-set.  Quantitative \(A_\infty\), together with the fixed inclusions
of the inner sets in the outer sets, compares the corresponding inner and
outer normalizations.  We may therefore work with normalized measures in the
two mixed-norm endpoint estimates below.

Proposition \ref{prop:FKS}, applied for almost every $(y,t)$ on a finite
cover of \(D_x\) by balls \(B_i\) satisfying \(2B_i\Subset D_x'\), and then
summed, gives a bound in the mixed space
$\mathrm L^2_{y,t}\bigl(\mathrm L^{2\chi_x}_{w,x}\bigr)$.
Since $2\chi_x\geq2$, Minkowski's integral inequality also gives
\begin{equation}
 \|f\|_{\mathrm L^{2\chi_x}_{w,x}(\mathrm L^2_{y,t})}
 \leq
 \|f\|_{\mathrm L^2_{y,t}(\mathrm L^{2\chi_x}_{w,x})}.
 \label{eq:minkowski-mixed}
\end{equation}
Indeed, writing $p=2\chi_x\geq2$ and applying Minkowski's integral
inequality with exponent $p/2$, we obtain
\begin{align*}
 \|f\|_{\mathrm L^p_{w,x}(\mathrm L^2_{y,t})}^2
 &=
 \left\|
   \int |f(\,\cdot\,,y,t)|^2\d y\d t
 \right\|_{\mathrm L^{p/2}_{w,x}}
 \leq
 \int
 \bigl\||f(\,\cdot\,,y,t)|^2\bigr\|_{\mathrm L^{p/2}_{w,x}}
 \d y\d t=
 \|f\|_{\mathrm L^2_{y,t}(\mathrm L^p_{w,x})}^2.
\end{align*}
Proposition \ref{prop:transfer} and fractional Sobolev embedding in the
$(m+1)$ variables $(y,t)$ give the other endpoint bound
$f\in\mathrm L^2_{w,x}(\mathrm L^q_{y,t})$ for some $q>2$. More precisely, when \(2s_*<m+1\), one may take
\[
 q=\frac{2(m+1)}{m+1-2s_*},
 \qquad\hbox{equivalently}\qquad
 \frac1q=\frac12-\frac{s_*}{m+1}.
\]
See \cite[Theorem~6.5]{DiNezzaPalatucciValdinoci2012}.  The embedding is
applied for almost every \(x\), after which its square is integrated against
\(w(x)\d x\).  Put
\[
 a=\frac12-\frac1{2\chi_x}>0,
 \qquad
 b=\frac12-\frac1q>0,
 \qquad
 \vartheta=\frac{a}{a+b}.
\]
Choose $r>2$ by
\[
 \frac1r
 =\frac{1-\vartheta}{2\chi_x}+\frac{\vartheta}{2}
 =\frac{1-\vartheta}{2}+\frac{\vartheta}{q}
 =\frac12-\frac{ab}{a+b}<\frac12.
\]
H\"older's inequality first in $(y,t)$ and then in $x$ gives directly
\[
 \|f\|_{\mathrm L^r(\d\muW)}
 \leq
 \|f\|_{\mathrm L^{2\chi_x}_{w,x}(\mathrm L^2_{y,t})}^{1-\vartheta}
 \|f\|_{\mathrm L^2_{w,x}(\mathrm L^q_{y,t})}^{\vartheta}.
\]
Both factors have already been bounded.  Taking $2\kappa=r$ proves the
normalized estimate.  Returning to the original measures proves
\eqref{eq:kineticsobolev}. The normalization factors cancel because all norms
in that estimate are normalized averages.
\end{proof}

We record the following localized, scale-invariant form of
Theorem \ref{thm:kineticsobolev}.  Let $\Phi_s$ be the flow of $Y$
defined in \eqref{eq:Yflow}.  Fix $P_0$, let
$R/2\leq\rho<R$, and set
\[
 P_{\rho,R}:=\Phi_{-(R^2-\rho^2)/2}(P_0).
\]
This shift gives $Q_\rho(P_{\rho,R})\Subset Q_R(P_0)$.  If
\(f,\nabla_xf,F,g\in\mathrm L^2(Q_R(P_0),\d\muW)\) and
\[
 Yf=\divw F+g
 \qquad\hbox{in }Q_R(P_0)
\]
in the weighted distributional sense, then
\begin{align}
 \left(
 \avgint_{Q_\rho(P_{\rho,R})}|f|^{2\kappa}\d\muW
 \right)^{1/(2\kappa)}
 &\leq
 C\left(\frac{R}{R-\rho}\right)^N
 \bigg[
 \left(
 \avgint_{Q_R(P_0)}|f|^2\d\muW
 \right)^{1/2}
 \notag\\
 &\quad+R\left(
 \avgint_{Q_R(P_0)}
 \bigl(|\nabla_xf|^2+|F|^2\bigr)\d\muW
 \right)^{1/2}
 +R^2\left(
 \avgint_{Q_R(P_0)}|g|^2\d\muW
 \right)^{1/2}
 \bigg],
 \label{eq:scaledkineticsobolev}
\end{align}
where \(\kappa>1\), \(N<\infty\), and \(C<\infty\) depend only on the
structural parameters.

To see this, normalize by $T_{P_0,R}$ and set
$f_R=f\circ T_{P_0,R}$, $F_R=RF\circ T_{P_0,R}$, and
$g_R=R^2g\circ T_{P_0,R}$.  The normalized equation has the same form
and the rescaled weight has the same $A_2$ characteristic.  With
$\theta=\rho/R$, the inner cylinder becomes
$Q_\theta(\Phi_{-(1-\theta^2)/2}(0))$.  Its spatial and time margins
in $Q_1$ are bounded below by fixed multiples of $1-\theta$.
Applying Theorem \ref{thm:kineticsobolev} with a cutoff supported in
$Q_1$ therefore gives $(1-\theta)^{-N}$.  The Jacobian cancels in
normalized averages, while gradients and the two data terms scale by
$R$, $R$, and $R^2$, respectively.  This proves
\eqref{eq:scaledkineticsobolev}.  The shift supplies an interior time
margin. The unshifted subsolution estimate needed below is proved
separately in Lemma \ref{lem:subsolutiongain}.

\section{Energy estimates}\label{Section_4}

Let \(Q_R(P_0)\) be an intrinsic cylinder.  We call \(\eta\) an
\emph{intrinsic Lipschitz cutoff} in \(Q_R(P_0)\) if
\(0\leq\eta\leq1\),
\[
 \eta,\ \nabla_x\eta,\ Y\eta\in\mathrm L^\infty(Q_R(P_0)),
\]
and \(\eta\) is compactly supported in the \(x\)- and \(y\)-sections of
\(Q_R(P_0)\) and vanishes near its initial time face.  Support inclusions
for such cutoffs are understood relative to the open time interval.
The cutoff need not vanish at the terminal time face.  In that case, one first tests below a time
\(T\), using an auxiliary one-sided time cutoff, and then lets \(T\) approach
the terminal face.  The resulting terminal energy term is non-negative and
may be discarded from the energy inequality.  In particular, if
\[
 Q_\rho(P_0)\Subset_{\mathrm{par}}Q_R(P_0),
\]
where the notation indicates nested cylinders with the same terminal time
face, then one may choose \(\eta=1\) on \(Q_\rho(P_0)\), with
\(\operatorname{supp}\eta\subset Q_R(P_0)\), and
\begin{equation}
 |\nabla_x\eta|
 \leq\frac{C}{R-\rho},
 \qquad
 |Y\eta|
 \leq\frac{C}{(R-\rho)^2}.
 \label{eq:cutoff}
\end{equation}
These bounds are invariant under the intrinsic translations and dilations
\eqref{eq:normalization}. This is the meaning of the adjective intrinsic.

\begin{lem}
\label{lem:caccioppoli}
Let \(h\geq0\) be an energy subsolution in \(Q_R(P_0)\), let \(p\geq2\),
and let \(\eta\) be a non-negative intrinsic Lipschitz cutoff in
\(Q_R(P_0)\).  Then
\begin{equation}
 \int_{Q_R(P_0)}
 \eta^2|\nabla_xh^{p/2}|^2\d\muW
 \leq
 Cp^2\int_{Q_R(P_0)}
 h^p\bigl(|\nabla_x\eta|^2+|\eta\,Y\eta|\bigr)\d\muW,
 \label{eq:caccioppoli}
\end{equation}
where \(C=C(\lambda,\Lambda)\).  The same estimate holds with
\(h=(u-\ell)_+\), where \(u\) is a subsolution and \(\ell\in\R\).
\end{lem}

\begin{proof}
Apply Lemma \ref{lem:weighted-renormalization} with
$F=B\nabla_xh$, $g=0$, and smooth truncations of
$\beta(h)=h^p/p$, using the cutoff $\eta^2$.  After discarding the
non-negative terminal energy, ellipticity gives
\begin{align*}
 \lambda(p-1)\int\eta^2h^{p-2}|\nabla_xh|^2\d\muW
 \leq{}&2\Lambda\int\eta h^{p-1}|\nabla_xh|\,|\nabla_x\eta|\d\muW
 +\frac2p\int h^p|\eta Y\eta|\d\muW.
\end{align*}
Young's inequality absorbs half of the first term on the left.  Since
\[
 |\nabla_xh^{p/2}|^2=\frac{p^2}{4}h^{p-2}|\nabla_xh|^2,
\]
we obtain \eqref{eq:caccioppoli}.  For an unbounded $h$, use
$\beta_T(q)=q^p/p$ on $0\leq q\leq T$ and continue it quadratically for
$q>T$, matching its first two derivatives at $T$.  Then $\beta_T''$ is
bounded, $\beta_T(q)\leq q^p/p$, and
$(\beta_T'(q))^2/\beta_T''(q)\leq C p q^p$ for $q>0$.
The same absorption estimate is uniform in $T$.  Let $T\to\infty$ and
use lower semicontinuity.  The assertion is substantive when the
right-hand side is finite.  Finally, $(u-\ell)_+$ is a subsolution by the same lemma.
\end{proof}

For arbitrary \(P_0\) and concentric intrinsic cylinders
\[
 Q_\rho(P_0)\Subset_{\mathrm{par}}Q_R(P_0),
\]
choose \(\eta\) as above.  Lemma \ref{lem:caccioppoli} and
\eqref{eq:cutoff} then give
\begin{equation}
 \int_{Q_\rho(P_0)}
 |\nabla_xh^{p/2}|^2\d\muW
 \leq
 \frac{Cp^2}{(R-\rho)^2}
 \int_{Q_R(P_0)}h^p\d\muW.
 \label{eq:caccioppoliballs}
\end{equation}
When \(P_0=0\), we write simply \(Q_\rho\) and \(Q_R\).

We also record explicitly the weighted Kato inequality used below.  If
\(h\) is smooth and \(\beta\in C^2(\R)\) is convex, then the formal chain
rule gives, in the sense of distributions,
\begin{equation}
 \Lw(\beta(h))
 =
 \beta'(h)\Lw h
 +\beta''(h)B\nabla_xh\cdot\nabla_xh
 \geq
 \beta'(h)\Lw h.
 \label{eq:weighted-kato}
\end{equation}
Here again only the symmetric part of \(B\) contributes to the
quadratic term.

For an energy subsolution, the expression
\(\beta'(h)\Lw h\) need not be separately defined as a distribution.  The
rigorous formulation is Lemma \ref{lem:weighted-renormalization},
applied to smooth convex approximations of $\beta$.  It yields
\[
 \Lw h\geq0,\qquad
 \beta'\geq0,\qquad
 \beta''\geq0
 \quad\Longrightarrow\quad
 \Lw(\beta(h))\geq0
\]
for convex Lipschitz functions \(\beta\), first with smooth bounded
derivatives and then by approximation.  Taking
\(\beta_j(s)\to(s-\ell)_+\) proves that \((h-\ell)_+\) is again a
subsolution.  This renormalized consequence of
\eqref{eq:weighted-kato} is what we call the weighted Kato inequality.  Its
proof uses only ellipticity, the weighted weak formulation, and the
formal skew-adjointness of \(Y\). In particular, it never differentiates \(w\).
The corresponding concave truncation statement holds for supersolutions.
In particular, if \(f\) is a supersolution, then \(\min\{f,c\}\) is a
supersolution and \((c-f)_+\) is a subsolution for every constant \(c\).

\section{Local boundedness}\label{Section_5}

We now combine the kinetic Sobolev estimate with the energy inequality.  We
first specify the convention for equations with measure data and establish
the comparison result used to remove the resulting Radon measure.

Let \(\mathcal O\subset\R^{n+m+1}\) be open, and let \(\nu\) be a locally
finite Radon measure on \(\mathcal O\).  We write
\begin{equation}
 \Lw f=\divw P+g+\nu
 \qquad\hbox{in }\mathcal O
 \label{eq:comparisonproblemapa}
\end{equation}
in the weighted distributional sense if
\begin{equation}
 -\int_{\mathcal O}B\nabla_xf\cdot\nabla_x\varphi\d\muW
 +\int_{\mathcal O}fY\varphi\d\muW
 =
 -\int_{\mathcal O}P\cdot\nabla_x\varphi\d\muW
 +\int_{\mathcal O}g\varphi\d\muW
 +\int_{\mathcal O}\varphi\d\nu
 \label{eq:weighted-measure-distribution}
\end{equation}
for every \(\varphi\in C^\infty_0(\mathcal O)\).  Equivalently,
\eqref{eq:weighted-measure-distribution} is the ordinary distributional
identity
\begin{equation}
 \operatorname{div}_x(wB\nabla_xf)-Y(wf)
 =
 \operatorname{div}_x(wP)+wg+\nu
 \qquad\hbox{in }\mathcal D'(\mathcal O).
 \label{eq:ordinary-measure-distribution}
\end{equation}
All terms in \eqref{eq:ordinary-measure-distribution} are defined directly as
distributions with respect to Lebesgue measure.  In particular, this
formulation does not require multiplication of a distribution by the
possibly rough function \(w\).

\begin{prop}
\label{prop:auxiliary-comparison}
Let \(\mathcal B\subset\R^n\) be a ball, let \(t_0<t_1\), and set
\[
 \mathcal C:=\mathcal B\times\R^m_y\times(t_0,t_1).
\]
Suppose that \(B\) satisfies \eqref{eq:Bbounds} in \(\mathcal C\) and that
\(P,g\in\mathrm L^2(\mathcal C,\d\muW)\) have compact support in the \(y\)
variable.  Then there exists a unique kinetic energy solution \(H\) of
\begin{equation}
 \Lw H=\divw P+g
 \quad\hbox{in }\mathcal C,
 \qquad
 H(t_0)=0,
 \qquad
 H|_{\partial\mathcal B}=0,
 \label{eq:comparisonproblem}
\end{equation}
and
\begin{equation}
 \sup_{t_0<t<t_1}\|H(t)\|_{\mathrm L^2_w}^2
 +\int_{\mathcal C}|\nabla_xH|^2\d\muW
 +\int_{\mathcal C}|H|^2\d\muW
 \leq
 C\int_{\mathcal C}(|P|^2+|g|^2)\d\muW.
 \label{eq:comparisonenergy}
\end{equation}
Here the time-slice norm is taken over
\(\mathcal B\times\R^m_y\), and \(C\) depends only on the structural
parameters, the weighted Poincar\'e constant of \(\mathcal B\), and
\(t_1-t_0\). Suppose, in addition, that
\[
 f,\ \nabla_xf\in\mathrm L^2(\mathcal C,\d\muW),
\]
that $f\in\mathrm L^2((t_0,t_1)\times\R^m_y;
\mathrm H^1_{0,w}(\mathcal B))$, and that $f$ satisfies
\eqref{eq:comparisonproblemapa} in the  sense of
\eqref{eq:weighted-measure-distribution}, for a non-negative
locally finite Radon measure $\nu$.
If its initial positive part vanishes in the sense
\begin{equation}
 \lim_{h\downarrow0}\frac1h
 \int_{t_0}^{t_0+h}\!\int_{\mathcal B\times\R^m_y}
 f_+^2\d\muW=0,
 \label{eq:comparison-initial-condition}
\end{equation}
then
\[
 f\leq H
 \qquad\hbox{almost everywhere in }\mathcal C.
\]
\end{prop}

\begin{proof}
The idea of the proof is to first approximate the problem by periodic problems in the $y$
variables.  This avoids prescribing lateral boundary data in $y$
and preserves the cancellation of the transport term in the energy
estimate.  We shall obtain estimates independent of the period and
then let it tend to infinity.

Let \(C_L:=(-L/2,L/2]^m\) and identify \(C_L\), modulo its opposite
faces, with the flat torus \(\mathbb T_L^m:=(\R/L\mathbb Z)^m\).
Choose \(L\) sufficiently large that the \(y\)-supports of \(P\)
and \(g\) are compactly contained in \(C_L\).  On
\(\mathcal B\times C_L\times(t_0,t_1)\), set
\[
 B_L(x,y,t):=B(x,y,t),\qquad
 P_L(x,y,t):=P(x,y,t),\qquad
 g_L(x,y,t):=g(x,y,t),
\]
and extend these functions \(L\)-periodically in \(y\).  Thus
\(B_L,P_L,\) and \(g_L\) are defined on
\(\mathcal B\times\mathbb T_L^m\times(t_0,t_1)\), and only the \(y\)
variable is periodized. Periodization preserves the coefficient bounds and the norms of the
compactly supported data on one period.  All subsequent estimates are
therefore uniform in $L$.

Set
\[
 \mathcal H_L
 :=
 \mathrm L^2(\mathcal B\times\mathbb T_L^m,w(x)\d x\d y),
 \qquad
 \mathcal V_L
 :=
 \mathrm L^2(\mathbb T_L^m;\mathrm H^1_{0,w}(\mathcal B)).
\]
Since $\mathrm H^1_{0,w}(\mathcal B)$ is separable and
$C^\infty_0(\mathcal B)$ is dense in this space, we may
choose a countable dense family of smooth compactly supported
functions.  Successively discarding functions that lie in
the span of the previously retained functions yields a
linearly independent family
$\{\phi_j\}_{j\geq1}\subset C^\infty_0(\mathcal B)$
whose linear span is dense in $\mathrm H^1_{0,w}(\mathcal B)$.
Let $\{e_\ell\}_{\ell\geq1}$ be a real trigonometric Fourier
basis of $\mathrm L^2(\mathbb T_L^m)$, with the constant
function first and both members of each sine-cosine pair
listed consecutively. For \(N\geq1\), set
\[
 \mathcal V_{L,N}
 :=
 \operatorname{span}
 \{\phi_j(x)e_\ell(y):1\leq j\leq N,\ 1\leq\ell\leq2N+1\}.
\]
The union of the spaces \(\mathcal V_{L,N}\) is dense in
\(\mathcal V_L\).

We seek \(H_{L,N}(t)\in\mathcal V_{L,N}\), with
\(H_{L,N}(t_0)=0\), such that, for every
\(\psi\in\mathcal V_{L,N}\),
\begin{align*}
 (\partial_tH_{L,N},\psi)_{\mathcal H_L}
 &+\int_{\mathcal B\times\mathbb T_L^m}
 (v\cdot\nabla_yH_{L,N})\psi\,w\d x\d y\\
 &+\int_{\mathcal B\times\mathbb T_L^m}
 B_L\nabla_xH_{L,N}\cdot\nabla_x\psi\,w\d x\d y\\
 &=
 \int_{\mathcal B\times\mathbb T_L^m}
 P_L\cdot\nabla_x\psi\,w\d x\d y
 -
 \int_{\mathcal B\times\mathbb T_L^m}
 g_L\psi\,w\d x\d y
\end{align*}
for almost every \(t\in(t_0,t_1)\).  Expanding $H_{L,N}$ in the chosen basis gives a finite system
$\mathsf M\dot a(t)+\mathsf K(t)a(t)=b(t)$ for its coefficient
vector $a$, with $a(t_0)=0$.  Here $\mathsf M$ is the
time-independent Gram matrix of the basis in $\mathcal H_L$.
It is positive definite because the basis functions are linearly
independent.  For fixed $L$ and $N$, the coefficient bounds imply
that $\mathsf K$ is bounded and measurable in time, while the
assumptions on $P_L$ and $g_L$ give
$b\in\mathrm L^2((t_0,t_1))\subset\mathrm L^1((t_0,t_1))$.
After multiplication by $\mathsf M^{-1}$, the associated integral
equation is solved by contraction on sufficiently short time
intervals.  Solving successively on these intervals gives a unique
absolutely continuous solution on $[t_0,t_1]$.

Testing the Galerkin equation with \(H_{L,N}\) gives
\begin{align}
 \frac12\|H_{L,N}(t)\|_{\mathcal H_L}^2
 &+\int_{t_0}^t\!\int
 B_L\nabla_xH_{L,N}\cdot\nabla_xH_{L,N}\d\muW
 \notag\\
 &=\int_{t_0}^t\!\int
 P_L\cdot\nabla_xH_{L,N}\d\muW -\int_{t_0}^t\!\int
 g_LH_{L,N}\d\muW.
 \label{eq:comparison-galerkin-energy}
\end{align}
Indeed, periodicity in \(y\) and the fact that \(w=w(x)\) imply
\[
 \int_{\mathcal B\times\mathbb T_L^m}
 H_{L,N}\,v\cdot\nabla_yH_{L,N}\,w\d x\d y
 =
 \frac12\int_{\mathcal B\times\mathbb T_L^m}
 v\cdot\nabla_y(H_{L,N}^2)\,w\d x\d y
 =0.
\]
Ellipticity, Young's inequality, and the weighted Poincar\'e inequality in
\(\mathcal B\) therefore give
\begin{equation}
 \esssup_{t_0<t<t_1}\|H_{L,N}(t)\|_{\mathcal H_L}^2
 +\int_{t_0}^{t_1}\!\int
 \bigl(|H_{L,N}|^2+|\nabla_xH_{L,N}|^2\bigr)\d\muW
 \leq
 C\int_{t_0}^{t_1}\!\int
 \bigl(|P_L|^2+|g_L|^2\bigr)\d\muW,
 \label{eq:comparison-galerkin-bound}
\end{equation}
where \(C\) is independent of \(L\) and \(N\).

For fixed \(L\), weak compactness gives, along a subsequence,
\[
 H_{L,N}\rightharpoonup H_L
 \quad\hbox{in }\mathrm L^2((t_0,t_1);\mathcal V_L),
 \qquad
 H_{L,N}\overset{*}{\rightharpoonup}H_L
 \quad\hbox{in }\mathrm L^\infty((t_0,t_1);\mathcal H_L).
\]
To pass to the equation while retaining the initial datum, let
\(\varphi\) be a smooth function on
\([t_0,t_1]\times\overline{\mathcal B}\times\mathbb T_L^m\) such that
\(\varphi(t_1)=0\) and
\(\varphi|_{\partial\mathcal B}=0\).  First truncate the Fourier series in $y$ to its first
$2N+1$ modes, and then project its coefficients onto $\operatorname{span}\{\phi_1,\ldots,\phi_N\}$ using
the $\mathrm H^1_w$ orthogonal projection.  These projections are
independent of time and commute with $\partial_t$ and $\nabla_y$.
Density in $\mathrm H^1_{0,w}$ and smoothness of $\varphi$ give,
by a diagonal choice, time-dependent
$\varphi_N\in\mathcal V_{L,N}$ for which
$\varphi_N,\nabla_x\varphi_N,\partial_t\varphi_N$, and
$\nabla_y\varphi_N$ converge strongly in the corresponding weighted
$\mathrm L^2$ spaces.  The zero value at $t_1$ is preserved.
Since $v$ is bounded on $\mathcal B$, $Y\varphi_N\to Y\varphi$
strongly as well.  Testing the Galerkin
equation with \(\varphi_N\), integrating by parts in time and \(y\), and
using \(H_{L,N}(t_0)=0\), we may pass to the limit and obtain
\begin{equation}
 \int_{t_0}^{t_1}\!\int
 \bigl(B_L\nabla_xH_L\cdot\nabla_x\varphi-H_LY\varphi\bigr)\d\muW
 =
 \int_{t_0}^{t_1}\!\int
 \bigl(P_L\cdot\nabla_x\varphi-g_L\varphi\bigr)\d\muW.
 \label{eq:comparison-cauchy-weak-form}
\end{equation}
Thus \(H_L\) solves the periodic problem and has zero initial datum in the
Cauchy weak formulation.  The equation now gives
\[
 YH_L
 =
 \divw(B_L\nabla_xH_L-P_L)-g_L
 \in\mathrm L^2((t_0,t_1);\mathcal V_L^*).
\]
The trace argument following Lemma
\ref{lem:weighted-renormalization} identifies this datum with the initial
$\mathcal H_L$ trace.  Weak lower semicontinuity in
\eqref{eq:comparison-galerkin-bound} gives the energy estimate for $H_L$.

We next let \(L\to\infty\).  Lift \(H_L\) periodically to
\(\mathcal B\times\R^m_y\times(t_0,t_1)\).  Given
\(K\Subset\R^m_y\), the lifted coefficients \(B_L,P_L,g_L\) agree with
\(B,P,g\) on \(\mathcal B\times K\times(t_0,t_1)\) for all sufficiently
large \(L\).  The preceding estimate gives uniform bounds on every such
compact cylinder.  The equation also gives a local bound for \(YH_L\) in
\[
 \mathrm L^2((t_0,t_1);
 \mathrm L^2(K;\mathrm H^{-1}_w(\mathcal B))).
\]
A diagonal weak compactness argument therefore gives a function \(H\) such
that, along a subsequence, \(H_L\rightharpoonup H\) and
\(\nabla_xH_L\rightharpoonup\nabla_xH\) locally in the corresponding
weighted \(\mathrm L^2\) spaces.  Since the periodized coefficients agree
with the original coefficients on each fixed compact \(y\)-set, passage to
the limit in \eqref{eq:comparison-cauchy-weak-form} shows that \(H\) solves
\eqref{eq:comparisonproblem} in \(\mathcal C\).

The zero initial datum is preserved by passing to the limit in the Cauchy
weak formulation with test functions that need not vanish at \(t=t_0\).
The same trace argument identifies this variational datum with the
initial \(\mathrm L^2_w\) trace.  The lateral zero trace is preserved because
\(\mathrm H^1_{0,w}(\mathcal B)\) is weakly closed.  Finally, weak lower
semicontinuity, first on
\(\mathcal B\times B_R^y\times(t_0,t_1)\) and then as \(R\to\infty\), gives
\eqref{eq:comparisonenergy}, including its
\(\mathrm L^\infty((t_0,t_1);\mathrm L^2_w)\) term.

To prove uniqueness, let $D$ be the difference of two solutions
with the same initial and lateral data.  Then
$YD=\divw(B\nabla_xD)$ and $D(t_0)=0$.
The global energy bounds and the zero lateral trace allow us
to apply \eqref{eq:weighted-green-identity} directly to $D$,
with $F=B\nabla_xD$ and $g=0$.  Using ellipticity, we obtain
\[
 \frac12\|D(t)\|_{\mathrm L^2_w}^2
 +\lambda\int_{t_0}^t\!\int_{\mathcal B\times\R^m_y}
 |\nabla_xD|^2\d\muW
 \leq0.
\]
Both terms are non-negative, so $D=0$, proving uniqueness.
The preceding construction is the bounded-$x$, weighted
analogue of the form construction in
\cite[Theorems~1.5 and~6.3]{AuscherImbertNiebel2025}.

For comparison, set $d:=f-H$ and choose
$\chi_R\in C^\infty_0(\R^m_y)$ such that
\[
 0\leq\chi_R\leq1,\qquad
 \chi_R=1\ \hbox{on }B_R^y,\qquad
 \supp\chi_R\subset B_{2R}^y,\qquad
 |\nabla_y\chi_R|\leq\frac{C}{R}.
\]
We may choose these cutoffs so that $\chi_R\uparrow1$ as
$R\to\infty$.  The function $d$ has zero lateral trace and satisfies
\[
 \operatorname{div}_x(wB\nabla_xd)-Y(wd)=\nu\geq0
 \qquad\hbox{in }\mathcal D'(\mathcal C).
\]
Moreover, $d_+\leq f_++|H|$.  Thus
\eqref{eq:comparison-initial-condition}, together with the
$\mathrm L^2_w$-continuity of $H$ at $t_0$ and $H(t_0)=0$,
gives the same averaged initial condition for $d_+$. Apply the zero-trace version of Lemma
\ref{lem:weighted-renormalization}, justified after its proof,
with the cutoff $\chi_R^2$ and smooth convex approximations to
$\beta(d)=\tfrac12d_+^2$ whose derivatives vanish for $d\leq0$.
The resulting nonlinear tests have zero lateral trace.
To include the initial endpoint, use a non-decreasing time
cutoff that vanishes near $t_0$, equals one for $t\geq t_0+h$,
and has derivative bounded by $C/h$.  First remove the
regularizations with $R$ and $h$ fixed.  The contribution of
the initial time cutoff then tends to zero as $h\downarrow0$
by the averaged initial condition.  Integrating to an interior
terminal time gives, for almost every $t$,
\begin{align*}
 \frac12\int_{\mathcal B\times\R^m_y}
 d_+^2(t)\chi_R^2w\d x\d y
 +\lambda\int_{t_0}^t\!\int_{\mathcal B\times\R^m_y}
 |\nabla_xd_+|^2\chi_R^2\d\muW
 \leq
 \frac{C}{R}\int_{t_0}^t\!\int_{\mathcal B\times
 (B_{2R}^y\setminus B_R^y)}
 d_+^2\d\muW.
\end{align*}
Here the right-hand side comes from the transport cutoff term,
using the bound on $\nabla_y\chi_R$ and the boundedness of $v$
on $\mathcal B$.  Letting $R\to\infty$, the global
$\mathrm L^2(\d\muW)$ bound and monotone convergence give
\[
 \frac12
 \|d_+(t)\|_{\mathrm L^2(\mathcal B\times\R^m_y,w\d x\d y)}^2
 +\lambda\int_{t_0}^t\!\int_{\mathcal B\times\R^m_y}
 |\nabla_xd_+|^2\d\muW
 \leq0.
\]
Consequently, $d_+=0$, and hence $f\leq H$ almost everywhere
in $\mathcal C$.  No additional integrability of $\nu$ is
required, since its contribution is discarded using
non-negativity.
\end{proof}

\begin{lem}
\label{lem:subsolutiongain}
There exist \(\kappa>1\), \(N<\infty\), and \(C<\infty\), depending only on
\(m,k,M,\lambda,\) and \(\Lambda\), such that the following holds.  Let
\(P_0\) be fixed, let \(R/2\leq\rho<R\), and let \(u\geq0\) be an energy
subsolution in \(Q_R(P_0)\).  Then
\begin{equation}
 \left(
 \avgint_{Q_\rho(P_0)}u^{2\kappa}\d\muW
 \right)^{1/(2\kappa)}
 \leq
 C\left(\frac{R}{R-\rho}\right)^N
 \left(
 \avgint_{Q_R(P_0)}u^2\d\muW
 \right)^{1/2}.
 \label{eq:subsolutiongain}
\end{equation}
The same estimate holds with \(u\) replaced by \((u-\ell)_+\), for every
\(\ell\in\R\).
\end{lem}

\begin{proof}
By the intrinsic translation and dilation
\eqref{eq:normalization}, it is enough to prove the estimate for
\(P_0=0\) and \(R=1\).  The rescaled weight has the same \(A_2\)
characteristic, and the normalized weighted averages are unchanged by the
geometric Jacobian.  We therefore assume that \(1/2\leq\rho<1\).

Set \(r:=(u-\ell)_+\).  The weighted Kato inequality
\eqref{eq:weighted-kato} shows that \(r\) is a subsolution.  More precisely,
the functional
\[
 \varphi\longmapsto
 -\int_{Q_1}B\nabla_xr\cdot\nabla_x\varphi\d\muW
 +\int_{Q_1}rY\varphi\d\muW
\]
is a positive distribution on \(Q_1\).  Every positive distribution is
represented by a unique non-negative locally finite Radon measure, see
\cite[Section~2.1]{Hormander1983}.  Hence there exists such a measure
\(\nu\) satisfying
\begin{equation}
 \Lw r=\nu
 \qquad\hbox{in }Q_1
 \label{eq:subsolutiondefect}
\end{equation}
in the weighted distributional sense of
\eqref{eq:weighted-measure-distribution}.  Equivalently,
\[
 \operatorname{div}_x(wB\nabla_xr)-Y(wr)=\nu
 \qquad\hbox{in }\mathcal D'(Q_1).
\]
All terms are defined by taking distributional derivatives
of the locally integrable functions $wr$ and $wB\nabla_xr$.
Thus this formulation does not require multiplication of a
distribution by the possibly rough weight $w$.

Choose radii $\rho<\rho_1<\rho_2<1$ so that all three gaps are comparable to \(1-\rho\).  Let \(\eta\) be an
intrinsic cutoff satisfying
\[
 \eta=1\quad\hbox{on }Q_{\rho_1},
 \qquad
 \operatorname{supp}\eta\subset Q_{\rho_2},
\]
and the corresponding bounds from \eqref{eq:cutoff}.  Set \(f:=\eta r\).
Using the product rule only in the weighted weak formulation gives
\begin{equation}
 \Lw f=\divw P+g+\eta\nu,
 \qquad
 P:=rB\nabla_x\eta,
 \qquad
 g:=B\nabla_xr\cdot\nabla_x\eta-rY\eta.
 \label{eq:localizedsubsolution}
\end{equation}
Indeed, the ordinary distributional form of this identity is
\[
 \operatorname{div}_x(wB\nabla_xf)-Y(wf)
 =
 \operatorname{div}_x(wP)+wg+\eta\nu.
\]
Thus no derivative of \(w\) is introduced.

Let \(t_*:=0\), choose a fixed ball \(\mathcal B\subset\R^n\) such that
\(\overline{B_1^x}\Subset\mathcal B\), and choose
\(t_0<-1<0<t_1\).  Set
\[
 \mathcal C
 :=
 \mathcal B\times\R^m_y\times(t_0,t_1),
 \qquad
 \mathcal C_-
 :=
 \mathcal B\times\R^m_y\times(t_0,t_*).
\]
We use the restriction of the original coefficient matrix $B$
to $\mathcal C$, which satisfies \eqref{eq:Bbounds}, and
extend $P$ and $g$ by zero to $\mathcal C$. Proposition \ref{prop:auxiliary-comparison} gives a solution \(H\) of
\[
 \Lw H=\divw P+g
 \qquad\hbox{in }\mathcal C
\]
with zero initial and lateral data.  On \(\mathcal C_-\), extend \(f\) and
\(\eta\nu\) by zero outside \(Q_1\) in the spatial variables and below the
initial time of \(Q_1\).  Since \(\eta\) is supported away from these parts
of \(\partial Q_1\), identity \eqref{eq:localizedsubsolution} remains valid
in $\mathcal C_-$.  The function $f$ has zero lateral trace and vanishes
on a whole interval following $t_0$, so it satisfies
\eqref{eq:comparison-initial-condition}.  Applying the comparison statement on
\(\mathcal C_-\) to \(f\) and the restriction of \(H\) gives
\[
 0\leq f\leq H
 \qquad\hbox{almost everywhere in }\mathcal C_-.
\]

We next apply the kinetic Sobolev estimate to an interior localization of
\(H\).  Choose a product cutoff
\(\zeta\in C_0^\infty(\mathcal C)\) satisfying
\(0\leq\zeta\leq1\) and \(\zeta=1\) on a neighborhood of
\(\overline{Q_1}\).  Since
\[
 YH=\divw(B\nabla_xH-P)-g,
\]
the localized function \(\widetilde H:=\zeta H\) satisfies
\[
 Y\widetilde H=\divw F_H+g_H,
\]
where
\[
 F_H:=\zeta(B\nabla_xH-P)
\]
and
\[
 g_H
 :=
 HY\zeta
 -(B\nabla_xH-P)\cdot\nabla_x\zeta
 -\zeta g.
\]
All derivatives fall on the fixed cutoff \(\zeta\), not on \(w\).
Theorem \ref{thm:kineticsobolev}, \eqref{eq:comparisonenergy}, and weighted
doubling therefore give
\begin{equation}
 \begin{aligned}
 \left(
 \avgint_{Q_1}|f|^{2\kappa}\d\muW
 \right)^{1/(2\kappa)}
 &\leq
 \left(
 \avgint_{Q_1}|H|^{2\kappa}\d\muW
 \right)^{1/(2\kappa)}\\
 &\leq
 C\muW(Q_1)^{-1/2}\bigl(
 \|P\|_{\mathrm L^2(\d\muW)}
 +\|g\|_{\mathrm L^2(\d\muW)}
 \bigr).
 \end{aligned}
 \label{eq:comparisongain}
\end{equation}
where the norms are over \(\mathcal C\) and the cutoff dependence is
controlled by a fixed power of \((1-\rho)^{-1}\).

It remains to estimate \(P\) and \(g\).  The cutoff bounds and
\eqref{eq:Bbounds} give
\[
 \|P\|_{\mathrm L^2(\d\muW)}
 \leq
 \frac{C}{1-\rho}
 \|r\|_{\mathrm L^2(Q_1,\d\muW)}
\]
and
\[
 \|g\|_{\mathrm L^2(\d\muW)}
 \leq
 \frac{C}{1-\rho}
 \|\nabla_xr\|_{\mathrm L^2(Q_{\rho_2},\d\muW)}
 +
 \frac{C}{(1-\rho)^2}
 \|r\|_{\mathrm L^2(Q_1,\d\muW)}.
\]
Choose a second intrinsic cutoff equal to one on \(Q_{\rho_2}\) and
supported in \(Q_1\).  Lemma \ref{lem:caccioppoli}, with \(p=2\), yields
\[
 \|\nabla_xr\|_{\mathrm L^2(Q_{\rho_2},\d\muW)}
 \leq
 \frac{C}{1-\rho}
 \|r\|_{\mathrm L^2(Q_1,\d\muW)}.
\]
Combining these estimates with \eqref{eq:comparisongain}, using
\(f=r\) on \(Q_\rho\), and applying weighted doubling to the normalized
averages gives
\[
 \left(
 \avgint_{Q_\rho}r^{2\kappa}\d\muW
 \right)^{1/(2\kappa)}
 \leq
 \frac{C}{(1-\rho)^N}
 \left(
 \avgint_{Q_1}r^2\d\muW
 \right)^{1/2}
\]
for a structural finite exponent \(N\).  Scaling back gives
\eqref{eq:subsolutiongain}.  Taking \(\ell=0\) proves the assertion for
\(u\), while arbitrary \(\ell\) gives the positive-part version.
\end{proof}

\begin{proof}[Proof of Proposition \ref{prop:localboundedness}]
Fix \(P_0\), \(0<\rho<R\), and \(p>0\).  We first consider \(p\geq2\).
If the right-hand side of \eqref{eq:localboundedness} is infinite, there is
nothing to prove, so assume that it is finite.  For $\gamma\geq1$ and $T>1$, define the convex truncation
\[
 b_{\gamma,T}(s):=
 \begin{cases}
 s^\gamma, & 0<s\leq T,\\
 T^\gamma+\gamma T^{\gamma-1}(s-T), & s>T,
 \end{cases}
\]
and zero elsewhere.  Thus $b_{\gamma,T}$ agrees with the power function up to
level $T$ and continues along its tangent line beyond $T$.
To obtain smooth approximations, fix a non-negative
$\varrho\in C^\infty_0((1,2))$ with $\int_1^2\varrho(a)\d a=1$
and set
\[
 \beta_{\gamma,T}(s)
 :=\int_1^2\varrho(a)b_{\gamma,T}(s-a/T)\d a,
 \qquad s\geq0.
\]
These functions are smooth, convex, and non-decreasing, and
vanish near zero.  Moreover,
\[
 0\leq\beta_{\gamma,T}(s)\leq s^\gamma,
 \qquad
 \beta_{\gamma,T}(s)\uparrow s^\gamma
 \quad\hbox{as }T\to\infty.
\]
Indeed, both $b_{\gamma,T}(q)$ for fixed $q$ and the
argument $s-a/T$ are non-decreasing in $T$.
For each fixed $T$, the bounds
$0\leq\beta'_{\gamma,T}\leq\gamma T^{\gamma-1}$ and
$\beta_{\gamma,T}(s)\leq\gamma T^{\gamma-1}s$ show that
$\beta_{\gamma,T}$ has bounded derivative and at most linear growth,
with constants depending on $\gamma$ and $T$.

The weighted Kato inequality shows that
\(\beta_{\gamma,T}(u+\varepsilon)\) is an energy subsolution.  Thus all
applications below are first made to these bounded-slope convex
approximations and then passed to the limit by Fatou's lemma and monotone
convergence.  This justifies the customary shorthand
\((u+\varepsilon)^\gamma\) for the iterated subsolutions.

Set
\[
 \rho_*:=\max\{\rho,R/2\},
 \qquad
 R_j:=\rho_*+(R-\rho_*)2^{-j},
 \qquad
 p_j:=p\kappa^j.
\]
Then \(R_0=R\), \(R_j\downarrow\rho_*\), and
\[
 R_j-R_{j+1}=(R-\rho_*)2^{-j-1}.
\]
Apply Lemma \ref{lem:subsolutiongain} first to
\[
 h_{j,T}:=\beta_{p_j/2,T}(u+\varepsilon)
\]
on \(Q_{R_j}(P_0)\), with inner cylinder \(Q_{R_{j+1}}(P_0)\), and then
let \(T\to\infty\).  Since \(p_j/2\geq1\), the preceding approximation is
admissible at every step.  We obtain
\begin{align*}
 &\left(
 \avgint_{Q_{R_{j+1}}(P_0)}
 (u+\varepsilon)^{p_{j+1}}\d\muW
 \right)^{1/p_{j+1}}\leq
 \left[
 C\left(
 \frac{R_j}{R_j-R_{j+1}}
 \right)^N
 \right]^{2/p_j}
 \left(
 \avgint_{Q_{R_j}(P_0)}
 (u+\varepsilon)^{p_j}\d\muW
 \right)^{1/p_j}.
\end{align*}
Because
\[
 \frac{R_j}{R_j-R_{j+1}}
 \leq
 C\,2^j\frac{R}{R-\rho_*}
\]
and
\[
 \sum_{j=0}^\infty\frac1{p_j}<\infty,
 \qquad
 \sum_{j=0}^\infty\frac{j}{p_j}<\infty,
\]
the product of the iteration constants converges.  Letting \(j\to\infty\)
gives
\[
 \esssup_{Q_{\rho_*}(P_0)}(u+\varepsilon)
 \leq
 C_{p,\rho/R}
 \left(
 \avgint_{Q_R(P_0)}(u+\varepsilon)^p\d\muW
 \right)^{1/p}.
\]
Since \(Q_\rho(P_0)\subset Q_{\rho_*}(P_0)\), letting
\(\varepsilon\downarrow0\) proves \eqref{eq:localboundedness} for
\(p\geq2\).

Now let \(0<p<2\).  For
\(\rho<\sigma<R\), the estimate already proved with exponent \(2\) gives
\[
 \esssup_{Q_\rho(P_0)}u
 \leq
 C_{\rho/\sigma}
 \left(
 \avgint_{Q_\sigma(P_0)}u^2\d\muW
 \right)^{1/2}.
\]
Writing
\[
 \left(
 \avgint_{Q_\sigma(P_0)}u^2\d\muW
 \right)^{1/2}
 \leq
 \bigl(\esssup_{Q_\sigma(P_0)}u\bigr)^{1-p/2}
 \left(
 \avgint_{Q_\sigma(P_0)}u^p\d\muW
 \right)^{1/2}
\]
and applying Young's inequality, we obtain, for every \(\delta>0\),
\[
 \esssup_{Q_\rho(P_0)}u
 \leq
 \delta\esssup_{Q_\sigma(P_0)}u
 +
 C_{\delta,p,\rho/\sigma}
 \left(
 \avgint_{Q_\sigma(P_0)}u^p\d\muW
 \right)^{1/p}.
\]
Choose radii
\[
 \rho_j:=\rho+(R-\rho)(1-2^{-j}),
\]
so that \(\rho_0=\rho\) and \(\rho_j\uparrow R\), and set
\[
 S_j:=\esssup_{Q_{\rho_j}(P_0)}u,
 \qquad
 A_p:=
 \left(
 \avgint_{Q_R(P_0)}u^p\d\muW
 \right)^{1/p}.
\]
Weighted doubling controls the normalized \(p\)-moments on
\(Q_{\rho_j}(P_0)\) by the corresponding moment on \(Q_R(P_0)\).  Keeping
track of the cutoff loss, the preceding inequality therefore gives
\[
 S_j
 \leq
 \delta S_{j+1}
 +C\delta^{-a}2^{bj}A_p
\]
for exponents \(a,b>0\) depending only on the structural parameters and \(p\).  Enlarge \(b\), if necessary, so that
it also dominates the cutoff-growth exponent in the already proved
exponent-two estimate.  Iterating the preceding inequality gives
\[
 S_0
 \leq
 \delta^J S_J
 +C\delta^{-a}A_p
 \sum_{j=0}^{J-1}(\delta2^b)^j.
\]
Choose \(\delta>0\) so that \(\delta2^b<1\).  The exponent-two estimate,
applied between \(Q_{\rho_J}(P_0)\) and \(Q_R(P_0)\), gives
\[
 S_J
 \leq
 C2^{bJ}
 \left(
 \avgint_{Q_R(P_0)}u^2\d\muW
 \right)^{1/2},
\]
and the last quantity is finite by the energy assumption.  Hence
\(\delta^JS_J\to0\).  Letting \(J\to\infty\) yields
\[
 \esssup_{Q_\rho(P_0)}u
 \leq
 C_{p,\rho/R}
 \left(
 \avgint_{Q_R(P_0)}u^p\d\muW
 \right)^{1/p}.
\]
This proves \eqref{eq:localboundedness} for every \(p>0\).
\end{proof}

\section{A weighted intermediate-value principle}
\label{Section_6}

We prove an intermediate-value principle for general $w=w(x)\in A_2(\R^n)$
by the compactness form of the kinetic De Giorgi method, compare
\cite[Section~4]{GolseImbertMouhotVasseur2019} and
\cite[Section~4.7]{Imbert2026}.  After normalizing the upper level to one,
the principle asserts that a non-negative supersolution which is at
least one on a set of positive weighted measure in the past and at most
$\theta$ on a set of positive weighted measure in the future must
occupy the intermediate range $(\theta,1)$ on a quantitatively positive
portion of a surrounding cylinder.  The constants depend only on the
structural parameters and the prescribed endpoint proportions.

The main input is compactness for bounded sequences with uniformly
controlled weighted diffusive energy and measure data.  We first prove
an unweighted averaging lemma, then apply it to full-variable moments
of the weighted densities.  Weighted Poincar\'e reconstructs the
functions in all diffusive variables.  Finally, weighted mean velocities
and the sign of the limiting measure rule out a two-phase jump.
For quantitative approaches based on kinetic Poincar\'e inequalities,
see \cite{GuerandMouhot2022,NiebelZacher2025,
AnceschiDietertGuerandLoherMouhotRebucci2024}.

\subsection{Averaging with a measure source}

The weighted energy bound controls derivatives in the
diffusive variables $x$, but alone gives no compactness
in the transported variables $(y,t)$.  Velocity averaging
supplies this missing ingredient: the transport equation
gives additional regularity in $(y,t)$ after integration
against a smooth velocity profile.  We will use this
principle to obtain compactness of weighted moments,
and then use weighted Poincar\'e inequalities to recover
compactness of the functions themselves. The equations satisfied by truncations contain Radon
measures, so the averaging statement must accommodate
measure sources.  We first establish the required
unweighted result, including regularity in time. For a Euclidean variable $X$, write
$\mathcal L_X=(1-\Delta_X)^{1/2}$ where $\Delta_X$ is the Laplacian in $X$. The following lemma is unweighted.

\begin{lem}
\label{lem:full-averaging}
Let $q_0>1$.  Suppose that $f_j,H_j,h_j$ are bounded in
$\mathrm L^{q_0}(\R^m_v\times\R^{m+1}_{y,t})$, have support in a
common compact set, and satisfy
\begin{equation}
 (\partial_t+v\cdot\nabla_y)f_j
 =\operatorname{div}_vH_j+h_j+\mu_j,
 \label{eq:full-unweighted-transport}
\end{equation}
where $\mu_j$ are signed Radon measures of uniformly bounded total
variation, supported in a common compact set.
For every $\psi\in C_0^\infty(\R^m)$, the sequence
\[
 A_j(y,t)=\int_{\R^m}\psi(v)f_j(v,y,t)\d v
\]
is relatively compact in $\mathrm L^p_{\rm loc}(\R^{m+1})$ whenever
\begin{equation}
 1<p<\min\left\{q_0,2,\frac{2d}{2d-1}\right\},\qquad d=m+1.
 \label{eq:full-averaging-p}
\end{equation}
In fact, $A_j$ is bounded in the inhomogeneous Besov space
$B^{1/(6p')}_{p,2}(\R^{m+1})$, where $p'=p/(p-1)$.
The bound depends only on $m,p,q_0$, the common support, $\psi$, and the
assumed norm and total-variation bounds.
\end{lem}

\begin{proof}
We use the stationary averaging theorem
\cite[Theorem~2, p.~282]{DiPernaLionsMeyer1991}\footnote{In the printed statement of Theorem~2, the second factor
should read $(I-\Delta_v)^{m/2}$, in the notation of that paper. See formulas (28)-(29) in its proof and the statement of
Theorem~3.}. The usefulness of this theorem here lies in its combination of
$\mathrm L^p$ integrability with negative-order regularity of the
source.  This allows us to accommodate the integrability supplied
by reverse H\"older together with the Radon measures arising from
truncation. In our notation, if $1<p\leq2$,
$F,L\in\mathrm L^p(\R^d_X\times\R^d_V)$, and
\[
 V\cdot\nabla_XF
 =\mathcal L_X^\alpha\mathcal L_V^\beta L
 \quad\hbox{in }\mathcal D'(\R^{2d}),
 \qquad 0\leq\alpha<1,\quad\beta>0,
\]
then, for every $\Psi\in C_0^\infty(\R^d)$,
\[
 \left\|\int_{\R^d}\Psi(V)F(\,\cdot\,,V)\d V
 \right\|_{B^s_{p,2}(\R^d)}
 \leq C\bigl(\|F\|_{\mathrm L^p}
             +\|L\|_{\mathrm L^p}\bigr),
 \qquad
 s=\frac{1-\alpha}{(\beta+1)p'},
 \quad p'=\frac{p}{p-1},
\]
where $C$ depends only on $d,p,\alpha,\beta$, and $\Psi$.
We use the modern notation $B^s_{p,2}$ for the Besov space in that
theorem.  Taking $\alpha=1/2$ and $\beta=2$ gives $s=1/(6p')$.

To apply this result we first lift \eqref{eq:full-unweighted-transport} to a stationary
equation.  Choose $0\leq\rho\in C_0^\infty((1,2))$ with
$\int_1^2\rho(a)a^m\d a=1$.  Set $V=(\xi,a)$, $X=(y,t)$, and
define, with extension by zero outside $1<a<2$,
\begin{align*}
 F_j(V,X)&=\rho(a)f_j(\xi/a,X),\\
\mathcal H_j(V,X)&=\bigl(a^2\rho(a)H_j(\xi/a,X),0\bigr),\\
 \mathfrak h_j(V,X)&=a\rho(a)h_j(\xi/a,X).
\end{align*}
Let $T(v,a,X)=(av,a,X)$ and define the measure
\[
 \mathfrak m_j
 :=T_\#\bigl(a^{m+1}\rho(a)\d a\otimes\mu_j\bigr).
\]
Here the tensor product has the natural variables $(v,a,X)$.
The change of variables $\xi=av$ gives, in distributions,
\begin{equation}
 V\cdot\nabla_XF_j
 =\operatorname{div}_V\mathcal H_j+\mathfrak h_j+\mathfrak m_j.
 \label{eq:full-lifted}
\end{equation}
Indeed, $V\cdot\nabla_XF_j=a\rho(a)Yf_j(\xi/a,X)$ and
$\operatorname{div}_V\mathcal H_j
=a\rho(a)\operatorname{div}_vH_j(\xi/a,X)$.
The factor $a^{m+1}$ in the lifted measure includes the Jacobian
$\d\xi=a^m\d v$.  The lifted functions are uniformly bounded in
$\mathrm L^p$, since their supports are fixed and $p<q_0$, and
$|\mathfrak m_j|(\R^{2d})\leq C|\mu_j|(\R^{2m+1})$.

We next put the right-hand side of \eqref{eq:full-lifted} in the form
required by the stationary theorem.  Let $K_\gamma^{(d)}$ denote the
kernel of $(1-\Delta)^{-\gamma/2}$ on $\R^d$, where $\gamma>0$.
The heat-kernel
representation gives
\[
 K_\gamma^{(d)}\in\mathrm L^p(\R^d)
 \quad\hbox{if }\gamma>d(1-1/p),
 \qquad K_\gamma^{(d)}\in\mathrm L^1(\R^d),
 \qquad \nabla K_2^{(d)}\in\mathrm L^1(\R^d).
\]
For example, the first assertion follows by integrating
$e^{-s}s^{\gamma/2-1}\|e^{s\Delta}\|_{\mathrm L^1\to\mathrm L^p}$
over $s>0$. Integrability at zero is exactly the displayed condition.
Condition \eqref{eq:full-averaging-p} implies $d/p'<1/2$.  Consequently,
\[
 K_{1/2}^{(d)}(X)K_2^{(d)}(V)\in\mathrm L^p(\R^{2d}).
\]
With our convention,
$\mathcal L_X^{-1/2}=(1-\Delta_X)^{-1/4}$ and
$\mathcal L_V^{-2}=(1-\Delta_V)^{-1}$, whose kernels are
$K_{1/2}^{(d)}$ and $K_2^{(d)}$, respectively. Define
\[
 L_j:=\mathcal L_X^{-1/2}\mathcal L_V^{-2}
 \bigl(\operatorname{div}_V\mathcal H_j
              +\mathfrak h_j+\mathfrak m_j\bigr).
\]
Young's convolution inequality, including convolution against a finite
measure, yields
\begin{equation}
 \|L_j\|_{\mathrm L^p}
 \leq C\bigl(\|\mathcal H_j\|_{\mathrm L^p}
             +\|\mathfrak h_j\|_{\mathrm L^p}
             +|\mathfrak m_j|(\R^{2d})\bigr).
 \label{eq:full-bessel-bound}
\end{equation}
For the divergence term one uses the $\mathrm L^1$ kernel
$K_{1/2}^{(d)}\otimes\nabla K_2^{(d)}$. For the scalar term one
uses $K_{1/2}^{(d)}\otimes K_2^{(d)}$ in $\mathrm L^1$. For the
measure term one uses the same kernel in $\mathrm L^p$.
Thus
\[
 V\cdot\nabla_XF_j=\mathcal L_X^{1/2}\mathcal L_V^2L_j
\]
with $F_j,L_j$ uniformly bounded in $\mathrm L^p$.

Choose $\kappa\in C_0^\infty((1,2))$ equal to one on
$\operatorname{supp}\rho$ and set
$\Psi(\xi,a)=\kappa(a)\psi(\xi/a)$, with smooth zero extension.
Then
\[
 \int_{\R^d}\Psi(V)F_j(V,X)\d V
 =\left(\int_1^2\rho(a)a^m\d a\right)
   \int_{\R^m}\psi(v)f_j(v,X)\d v=A_j(X).
\]
The stationary theorem gives a uniform bound in
$B^s_{p,2}(\R^d)$, where $s=1/(6p')\in(0,1)$.
Here $B^s_{p,2}$ denotes the inhomogeneous Besov space.
For $0<s<1$, the embedding
$B^s_{p,2}\hookrightarrow B^s_{p,\infty}$ and the
difference-quotient characterization of the latter space give
\[
 \|A_j(\,\cdot+h)-A_j\|_{\mathrm L^p(\R^d)}
 \leq C|h|^s\|A_j\|_{B^s_{p,\infty}(\R^d)}
 \leq C|h|^s\|A_j\|_{B^s_{p,2}(\R^d)}
 \leq C|h|^s,
 \qquad |h|\leq1,
\]
uniformly in $j$.  Thus translations are uniformly small
in $\mathrm L^p$ as $h\to0$.  Moreover, the functions $A_j$
are uniformly bounded in $\mathrm L^p(\R^d)$ and supported
in a common compact set, the projection onto the
$X$-variables of the common support of the $f_j$.
The Kolmogorov-Riesz compactness theorem therefore implies
that $\{A_j\}$ is relatively compact in
$\mathrm L^p(\R^d)$, and in particular locally. This proves the lemma.
\end{proof}

\subsection{Weighted measure-data compactness} We now combine the preceding averaging lemma with weighted
Poincar\'e inequalities in the full diffusive space to prove
a compactness theorem for bounded sequences with uniformly
controlled weighted energy and measure data.  The averaging
lemma gives compactness of smooth weighted moments, while
the Poincar\'e inequalities allow us to recover compactness
of the functions themselves.

In the following we allow both the weights $w_j$ and the coefficient matrices
$B_j$ to vary along the sequence.  This is needed to obtain
constants uniform over the class $[w]_{A_2}\leq M$.
The conclusion concerns strong compactness of  functions
$u_j$ and their weighted moments.  The densities $w_ju_j$
may retain oscillations of the weights, and no convergence
of the diffusion matrices is assumed.

Let $D_v\subset\R^m$, $D_z\subset\R^k$, and $D_y\subset\R^m$ be
bounded balls and let $I$ be a bounded open interval.  Put
\[
 D=D_v\times D_z,\qquad G=D_y\times I,\qquad \mathcal Q=D\times G.
\]
When $k=0$, the $z$-variable and the factor $D_z$ are omitted.  In
particular, when $k=0$ the weights below may still depend on $v$.
All equations are understood in ordinary distributions.

\begin{prop}
\label{prop:weighted-measure-data-compactness}
Suppose that $w_j\in A_2(\R^n)$ depend only on $x$, and that
\begin{equation}
 [w_j]_{A_2}\leq M,\qquad \int_D w_j\d x=1.
 \label{eq:compactness-weight-normalization}
\end{equation}
Set $\d\mu_j:=w_j(x)\d x\d y\d t$.
Let $B_j$ be measurable matrices with $|B_j\xi|\leq\Lambda|\xi|$.
Suppose that $0\leq u_j\leq1$, that $u_j$ have weak $x$-gradients,
and that
\begin{equation}
 \operatorname{div}_x(w_jB_j\nabla_xu_j)-Y(w_ju_j)
 =w_js_j+\nu_j\qquad\hbox{in }\mathcal D'(\mathcal Q),
 \label{eq:weighted-measure-data-sequence}
\end{equation}
where $\nu_j$ are finite signed Radon measures and
\begin{equation}
 \sup_j\left\{
 \int_{\mathcal Q}w_j\bigl(|\nabla_xu_j|^2+|s_j|^2\bigr)
     \d x\d y\d t
 +|\nu_j|(\mathcal Q)\right\}<\infty.
 \label{eq:weighted-measure-data-bound}
\end{equation}
Then there exist $\epsilon=\epsilon(n,M)>0$, a subsequence, and
functions $w_\infty$ and $u_\infty$ such that
\begin{align}
 w_j&\rightharpoonup w_\infty
       &&\hbox{weakly in }\mathrm L^{1+\epsilon}(D),
       \label{eq:full-weight-limit}\\
 u_j&\longrightarrow u_\infty
       &&\hbox{strongly in }\mathrm L^2_{\rm loc}(\mathcal Q),
       \label{eq:full-unweighted-limit}\\
 \int_K|u_j-u_\infty|^2w_j\d x\d y\d t&\longrightarrow0
       &&\hbox{for every }K\Subset\mathcal Q.
       \label{eq:full-weighted-limit}
\end{align}
Moreover, $0\leq u_\infty\leq1$,
\[
 \int_Dw_\infty\d x=1,\qquad
 w_\infty>0\ \hbox{almost everywhere},\qquad
 w_\infty,w_\infty^{-1}\in\mathrm L^{1+\epsilon}(D),
\]
and the $A_2$ inequality with constant $M$ holds for $w_\infty$ on
every ball compactly contained in $D$.
For every $\theta\in C_0^\infty(D)$ and $G'\Subset G$, the moments
$\int_D\theta w_ju_j\d x$ converge strongly to
$\int_D\theta w_\infty u_\infty\d x$ in $\mathrm L^q(G')$ for every
$1\leq q<\infty$.

\noindent
For some $p>1$ there are also local weak $\mathrm L^p$ limits
\[
 J_j:=w_jB_j\nabla_xu_j\rightharpoonup J_\infty,
 \qquad h_j:=-w_js_j\rightharpoonup h_\infty,
 \qquad \nabla_xu_j\rightharpoonup\nabla_xu_\infty.
\]
After extraction, $\nu_j\stackrel{*}{\rightharpoonup}\nu_\infty$ locally
as measures, and
\begin{equation}
 Y(w_\infty u_\infty)
 =\operatorname{div}_xJ_\infty+h_\infty-\nu_\infty
 \qquad\hbox{in }\mathcal D'(\mathcal Q).
 \label{eq:full-limit-equation}
\end{equation}
If $\nu_j\geq0$, then $\nu_\infty\geq0$.
\end{prop}

\begin{proof} We divide the proof into a number of steps.

\smallskip
\noindent
\emph{Step 1. Uniform integrability and weak limits of the weights.}
Choose a fixed ball $B_0\supset D$.  By Cauchy-Schwarz on $D$ and
the $A_2$ inequality on $B_0$,
\[
 1\leq\int_{B_0}w_j\d x
       \leq M\frac{|B_0|^2}{|D|^2},
 \qquad
 \int_{B_0}w_j^{-1}\d x\leq M|B_0|^2.
\]
Reverse H\"older for $w_j$ and $w_j^{-1}$ therefore gives, for
$r=1+\epsilon>1$ depending only on $n$ and $M$,
\begin{equation}
 \sup_j\bigl(\|w_j\|_{\mathrm L^r(D)}
                  +\|w_j^{-1}\|_{\mathrm L^r(D)}\bigr)<\infty.
 \label{eq:full-rh}
\end{equation}
After extraction, let their weak $\mathrm L^r(D)$ limits be
$w_\infty$ and $\widetilde w_\infty$.  These are non-negative, and
$\int_Dw_\infty\d x=1$. For each fixed rational number $a\geq0$, testing the
inequality $2a-a^2w_j\leq w_j^{-1}$ against non-negative
bounded functions and passing to the weak limit gives
\[
 2a-a^2w_\infty\leq\widetilde w_\infty
 \quad\hbox{almost everywhere}.
\]
Since the non-negative rational numbers form a countable
set, we can remove a single set of measure zero so that
this inequality holds for every such $a$ at each remaining
point. We first show that $w_\infty>0$ almost everywhere.
We already know that $w_\infty\geq0$.  At any point where
$w_\infty=0$, the preceding inequality would give
$2a\leq\widetilde w_\infty$ for every non-negative rational
$a$.  This is impossible because $\widetilde w_\infty$
is finite almost everywhere. Now fix a point where $w_\infty>0$ and all the preceding
inequalities hold.  Choose non-negative rational numbers
$a_\ell$ tending to $1/w_\infty$ at this point.  Then
$2a_\ell-a_\ell^2w_\infty$ tends to $1/w_\infty$.
Passing to the limit therefore gives
\[
 w_\infty^{-1}\leq\widetilde w_\infty<\infty
 \quad\hbox{almost everywhere}.
\]
Passing to the weak
limit on any ball $B\Subset D$ gives
\[
 \left(\frac1{|B|}\int_Bw_\infty\d x\right)
 \left(\frac1{|B|}\int_B\widetilde w_\infty\d x\right)\leq M,
\]
which proves the stated local $A_2$ bound.

\smallskip
\noindent
\emph{Step 2. Unweighted bounds for the transport equation.}
Set $b_j=w_ju_j$, $J_j=w_jB_j\nabla_xu_j$, and $h_j=-w_js_j$.
Then
\begin{equation}
 Yb_j=\operatorname{div}_xJ_j+h_j-\nu_j.
 \label{eq:full-b-equation}
\end{equation}
Let $q_0=2r/(r+1)>1$ where $r=1+\epsilon$.  For any vector or scalar function $F$,
H\"older's inequality gives
\begin{equation}
 \|w_jF\|_{\mathrm L^{q_0}(\mathcal Q)}
 \leq\left(\int_{\mathcal Q}w_j|F|^2\right)^{1/2}
       \left(\int_{\mathcal Q}w_j^r\right)^{1/(2r)}.
 \label{eq:full-flux-integrability}
\end{equation}
Here and below, unspecified integrals over $\mathcal Q$ use
$\d x\d y\d t$.  Thus $b_j,J_j,h_j$ are bounded in
$\mathrm L^{q_0}(\mathcal Q)$. For $b_j$ use $0\leq b_j\leq w_j$.
Similarly,
\begin{equation}
 \|\nabla_xu_j\|_{\mathrm L^{q_0}(\mathcal Q)}
 \leq\left(\int_{\mathcal Q}w_j|\nabla_xu_j|^2\right)^{1/2}
       \left(\int_{\mathcal Q}w_j^{-r}\right)^{1/(2r)}\leq C.
 \label{eq:full-gradient-integrability}
\end{equation}

\smallskip
\noindent
\emph{Step 3. Compactness of full-variable moments.} To apply Lemma \ref{lem:full-averaging}, we first localize
the equation to obtain quantities with common compact support
and integrate out the passive variables $z$.  This produces
an equation in $(v,y,t)$ of the form required by the lemma. Write $J_j=((J_j)_v,(J_j)_z)$ according to the splitting
$x=(v,z)$, so that $(J_j)_v$ consists of the first $m$
components of the diffusion flux.  Fix
$\theta\in C_0^\infty(D)$ and $\eta\in C_0^\infty(G)$.
Multiplying \eqref{eq:full-b-equation} by
$\theta(x)\eta(y,t)$ and integrating in $z$, with the
localized quantities extended by zero, gives
\[
 Yf_j=\operatorname{div}_vH_j+q_j-\widetilde\nu_j,
\]
where
\begin{align*}
 f_j(v,y,t)
 &=\eta(y,t)\int_{\R^k}\theta(v,z)b_j(v,z,y,t)\d z,\\
 H_j(v,y,t)
 &=\eta(y,t)\int_{\R^k}
       \theta(v,z)(J_j)_v(v,z,y,t)\d z,\\
 q_j(v,y,t)
 &=\int_{\R^k}
    \bigl(\eta\theta h_j+\theta b_jY\eta
                   -\eta J_j\cdot\nabla_x\theta\bigr)\d z,
\end{align*}
and $\widetilde\nu_j$ is the projection of
$\eta\theta\nu_j$ onto $(v,y,t)$.
The full gradient $\nabla_x\theta$ in $q_j$ includes both
the active and passive contributions.  The divergence in $z$
integrates to zero because the localized flux has compact
support in these variables. H\"older's inequality on the bounded $z$-support gives
uniform $\mathrm L^{q_0}$ bounds for $f_j$, $H_j$, and $q_j$.
The total variation of $\widetilde\nu_j$ is at most
$C|\nu_j|(\mathcal Q)$.  All these quantities have support
in a common compact set determined by $\theta$ and $\eta$.
When $k=0$, the same formulas apply without the $z$-integrals.

Let $G'\Subset G$, and choose $\eta\in C_0^\infty(G)$
equal to one on a neighbourhood of $\overline{G'}$.
Choose also $\psi\in C_0^\infty(\R^m)$ equal to one
at every velocity $v$ for which $(v,z)$ belongs to
$\operatorname{supp}\theta$ for some $z$.
Such a choice is possible because $\theta$ has compact
support. We apply Lemma \ref{lem:full-averaging} to the equation
for $f_j$, with scalar term $q_j$ and source measure
$-\widetilde\nu_j$.  The hypotheses of the lemma follow
from the support and uniform bounds established above.
Consequently, for every fixed $p$ satisfying
\eqref{eq:full-averaging-p}, the averages
$\int_{\R^m}\psi(v)f_j(v,y,t)\d v$ are relatively
compact in $\mathrm L^p(G')$.

To identify these averages, use the definition of $f_j$
and the fact that $\psi(v)\theta(v,z)=\theta(v,z)$.
Fubini's theorem gives
\begin{align*}
 \int_{\R^m}\psi(v)f_j(v,y,t)\d v
 &=
 \eta(y,t)\int_{\R^m}\int_{\R^k}
       \theta(v,z)b_j(v,z,y,t)\d z\d v\\
 &=
 \eta(y,t)\int_D\theta(x)w_j(x)u_j(x,y,t)\d x.
\end{align*}
Thus the integration in $z$ used to define $f_j$,
followed by the velocity average in $v$, gives an
integral over the entire diffusive variable $x=(v,z)$.
Since $\eta=1$ on $G'$, we have proved that
\begin{equation}
 M_{j,\theta}(y,t):=
 \int_D\theta(x)w_j(x)u_j(x,y,t)\d x
 \quad\hbox{is relatively compact in }\mathrm L^p(G').
 \label{eq:full-moment-compactness}
\end{equation}

\smallskip
\noindent
\emph{Step 4. Finite reconstruction in all diffusive variables.}
Fix $D'\Subset D''\Subset D$ and $G'\Subset G$.  For sufficiently
small $\delta>0$, cover $D'$ by finitely many balls $B_\ell$ of
radius $\delta$, with $4B_\ell\Subset D''$ and uniformly bounded
overlap of the doubled balls.  Choose a non-negative partition of
unity $\zeta_\ell$ on $D'$ subordinate to these balls, and choose
$\theta_\ell\in C_0^\infty(2B_\ell)$ with
$0\leq\theta_\ell\leq1$ and $\theta_\ell=1$ on $B_\ell$.  Define
\[
 c_{j,\ell}:=\int_D\theta_\ell w_j\d x,
 \qquad a_{j,\ell}:=c_{j,\ell}^{-1}M_{j,\theta_\ell},
 \qquad P_\delta u_j:=\sum_\ell\zeta_\ell a_{j,\ell}.
\]
For each fixed $\delta$, the denominators converge to
$\int_D\theta_\ell w_\infty\d x>0$.  They are also bounded below
directly by
\[
 c_{j,\ell}\geq w_j(B_\ell)
 \geq\frac{|B_\ell|^2}{\int_Dw_j^{-1}\d x}>c_\delta>0.
\]
Moment compactness and convergence of the positive denominators
give, after extraction, $a_{j,\ell}\to a_\ell$ in
$\mathrm L^p(G')$.  Since $0\leq a_{j,\ell},a_\ell\leq1$
and $p<2$, we have
\[
 \|a_{j,\ell}-a_\ell\|_{\mathrm L^2(G')}^2
 \leq
 \|a_{j,\ell}-a_\ell\|_{\mathrm L^p(G')}^p
 \longrightarrow0.
\]
For each fixed $\delta$, there are only finitely many
coefficients.  Extracting a common subsequence therefore
proves relative compactness of $P_\delta u_j$ in
$\mathrm L^2(D'\times G')$ for each fixed $\delta$.

We now prove a uniform approximation estimate.  At each $(y,t)$ let
$\bar u_{j,\ell}$ denote the $w_j$-mean of $u_j$ on $2B_\ell$.
Jensen's inequality gives
\[
 |a_{j,\ell}-\bar u_{j,\ell}|^2
 \leq\frac1{c_{j,\ell}}
       \int_{2B_\ell}\theta_\ell
             |u_j-\bar u_{j,\ell}|^2w_j\d x.
\]
Since $w_j(B_\ell)\leq c_{j,\ell}$, weighted Poincar\'e on the
full $n$-dimensional ball $2B_\ell$ implies
\begin{align*}
 \int_{B_\ell}|u_j-a_{j,\ell}|^2w_j\d x
 &\leq4\int_{2B_\ell}|u_j-\bar u_{j,\ell}|^2w_j\d x\leq C\delta^2\int_{2B_\ell}|\nabla_xu_j|^2w_j\d x.
\end{align*}
The constant depends only on $n$ and $M$.  Here we used Proposition \ref{prop:FKS}.
Jensen's inequality for the partition of unity, followed by bounded
overlap and \eqref{eq:weighted-measure-data-bound}, now gives
\begin{equation}
 \int_{D'\times G'}|u_j-P_\delta u_j|^2w_j
 \leq C\delta^2.
 \label{eq:full-reconstruction}
\end{equation}
Note that no Poincar\'e inequality on a conditional $v$- or $z$-slice is used.

Put $e_{j,\delta}:=u_j-P_\delta u_j$.  Since both $u_j$ and
$P_\delta u_j$ take values in $[0,1]$, one has
$|e_{j,\delta}|\leq1$ on $D'\times G'$.  Cauchy-Schwarz,
\eqref{eq:full-reconstruction}, and \eqref{eq:full-rh} therefore give
\begin{align}
 \int_{D'\times G'}|e_{j,\delta}|^2
 &\leq \int_{D'\times G'}|e_{j,\delta}| \leq
 \left(\int_{D'\times G'}|e_{j,\delta}|^2w_j\right)^{1/2}
 \left(\int_{D'\times G'}w_j^{-1}\right)^{1/2}
 \leq C\delta.
 \label{eq:full-unweighted-reconstruction}
\end{align}
The constant is independent of $j$ and $\delta$.

Choose a sequence $\delta_\ell\downarrow0$.  By successive
subsequence extraction and a diagonal selection, we obtain
a subsequence of $u_j$, with the corresponding weights taken
along the same indices, such that $P_{\delta_\ell}u_j$
converges in $\mathrm L^2(D'\times G')$ for every fixed $\ell$.
We relabel this subsequence as $u_j$.

To prove convergence of $u_j$ itself, the triangle inequality
and \eqref{eq:full-unweighted-reconstruction} give
\[
 \|u_j-u_i\|_{\mathrm L^2(D'\times G')}
 \leq 2C\delta_\ell^{1/2}
 +\|P_{\delta_\ell}u_j-P_{\delta_\ell}u_i\|
       _{\mathrm L^2(D'\times G')}.
\]
Given $\varepsilon>0$, first choose $\ell$ so large that
$2C\delta_\ell^{1/2}<\varepsilon/2$.  At this fixed scale,
the reconstructed sequence converges, so the last term is
less than $\varepsilon/2$ for all sufficiently large $i,j$.
Thus $u_j$ is Cauchy in $\mathrm L^2(D'\times G')$ and
converges strongly there.  A further diagonal selection
over an interior exhaustion of $\mathcal Q$ proves
\eqref{eq:full-unweighted-limit}.
After passing to a further subsequence, the strong local
$\mathrm L^2$ convergence gives $u_j\to u_\infty$ almost
everywhere.  Passing to the limit in $0\leq u_j\leq1$
therefore yields $0\leq u_\infty\leq1$ almost everywhere.

\smallskip
\noindent
\emph{Step 5. Strong convergence for the varying measures.}
Let $K\Subset\mathcal Q$, let $r>1$ be the exponent in
\eqref{eq:full-rh}, and put $r'=r/(r-1)$.
H\"older's inequality, the uniform bound in
\eqref{eq:full-rh}, and $|u_j-u_\infty|\leq1$ give
\begin{align*}
 \int_K|u_j-u_\infty|^2w_j
 &\leq
 \left(\int_Kw_j^r\right)^{1/r}
 \left(\int_K|u_j-u_\infty|^{2r'}\right)^{1/r'}\leq C
 \left(\int_K|u_j-u_\infty|^2\right)^{1/r'}
 \longrightarrow0.
\end{align*}
The last limit follows from the strong local unweighted
$\mathrm L^2$ convergence established above.
This proves \eqref{eq:full-weighted-limit}.

\smallskip
\noindent
\emph{Step 6. Passage to the weak equation.}
The $\mathrm L^{q_0}$ bounds and the measure bound give subsequential
weak limits $J_\infty,h_\infty$, and $\nu_\infty$.
The gradient bound \eqref{eq:full-gradient-integrability}, together
with strong local convergence of $u_j$, identifies the weak gradient
limit with $\nabla_xu_\infty$.
For $\varphi\in C_0^\infty(\mathcal Q)$, weighted strong convergence
and Cauchy-Schwarz imply
\[
 \int_{\mathcal Q}w_j(u_j-u_\infty)\varphi\longrightarrow0.
\]
Weak convergence of the weights, tested against the bounded function
\[
 x\mapsto\int_Gu_\infty(x,y,t)\varphi(x,y,t)\d y\d t,
\]
also gives
\[
 \int_{\mathcal Q}(w_j-w_\infty)u_\infty\varphi\longrightarrow0.
\]
Thus $w_ju_j\to w_\infty u_\infty$ in distributions.
Testing with $\theta(x)\eta(y,t)$ identifies every limit in
\eqref{eq:full-moment-compactness} with
$\int_D\theta w_\infty u_\infty\d x$.
Since $|M_{j,\theta}|\leq\|\theta\|_\infty\int_Dw_j\d x
=\|\theta\|_\infty$, this compactness gives strong convergence in
every finite $\mathrm L^q(G')$ along the selected subsequence.
Passing to the limit in \eqref{eq:full-b-equation} proves
\eqref{eq:full-limit-equation}.  Positivity of the measures, when
assumed, is preserved under weak convergence.
\end{proof}
\begin{rem} No ellipticity is used in the compactness argument once
\eqref{eq:weighted-measure-data-bound} is assumed.  More generally,
$B_j\nabla_xu_j$ may be replaced by any vector field $F_j$ satisfying
$\sup_j\int_{\mathcal Q}w_j|F_j|^2<\infty$, while retaining the
gradient bound in \eqref{eq:weighted-measure-data-bound}.  The proposition does not identify
$J_\infty$ with a limiting diffusion matrix times $\nabla_xu_\infty$.
\end{rem}

\subsection{A two-phase no-jump lemma}

We next isolate the qualitative consequence of Proposition
\ref{prop:weighted-measure-data-compactness}.  Write $G_0=D_y\times I$
for the outer $(y,t)$-region in that proposition.  Since the conclusion is
local, we may take the inner set to have the product form
\[
 \mathcal Q'
 =
 D_v'\times D_z'\times G,
\]
where \(D_v'\Subset D_v\) and \(D_z'\Subset D_z\) are connected balls and
\(G\Subset G_0\) is a bounded connected open set.  Let
\[
 \mathcal Q^-
 =
 D_v^-\times D_z^-\times E^-,
 \qquad
 \mathcal Q^+
 =
 D_v^+\times D_z^+\times E^+
\]
be product boxes compactly contained in \(\mathcal Q'\).

We say that \((\mathcal Q^-,\mathcal Q^+)\) is an
\emph{admissible ordered pair} if \(E^-\) lies strictly before \(E^+\) in
time and there is an open ball \(V\Subset D_v'\) such that, for every
\((y_-,t_-)\in E^-\) and \((y_+,t_+)\in E^+\),
\begin{equation}
 t_-<t_+,
 \qquad
 v_{-+}:=\frac{y_+-y_-}{t_+-t_-}\in V,
 \label{eq:no-jump-reachability}
\end{equation}
and the characteristic segment
\begin{equation}
 s\longmapsto
 \bigl(y_-+(s-t_-)v_{-+},s\bigr),
 \qquad t_-\leq s\leq t_+,
 \label{eq:no-jump-characteristic-segment}
\end{equation}
stays a fixed positive distance from $\partial G$, uniformly over the
two endpoints.  This distance and $\operatorname{dist}(V,\partial D_v')$
are among the geometric margins of the ordered configuration.

Geometrically, \(v_{-+}\) is the unique constant velocity for which the
forward characteristic of
\(v\cdot\nabla_y+\partial_t\) starting at \((y_-,t_-)\) reaches
\((y_+,t_+)\).  Thus admissibility means that every point of \(E^-\) can
be connected to every point of \(E^+\) by a forward kinetic
characteristic that remains uniformly inside the domain.  The containment
\(V\Subset D_v'\) keeps all connecting velocities away from the boundary
of the active velocity set and allows the limiting equation to be
tested against diffusive profiles whose weighted velocity averages
approach \(v_{-+}\).  The compact
containment of the characteristic segment in \(G\) prevents boundary
effects in the corresponding propagation argument.  This uniform
reachability is what will allow the directional inequalities to rule out
a jump from the zero phase in \(E^-\) to the one phase in \(E^+\).

The fixed boxes \eqref{eq:Qminus}-\eqref{eq:Qplus} already satisfy this
condition: their time gap is at least $1/4$, while the distance between
their $y$-coordinates is less than $1/32$.  Thus every connecting velocity
has norm less than $1/8$.  No reduction of the stated reference boxes is
needed.

The next lemma is a qualitative rigidity statement.  It rules out a
sequence of bounded subsolutions which is close to the phase zero on a set
of uniformly positive weighted measure in the past, close to the phase one
on a set of uniformly positive weighted measure in the future, and yet
takes values strictly between these phases on a set whose weighted measure
tends to zero.  Indeed, disappearance of the intermediate phase produces a
binary limit.  The sign of the  measure then forces this limit to be
non-increasing along forward kinetic characteristics, so that it cannot
jump from zero in the past to one in the future.  The lemma is qualitative and in the following subsection it will be used in a contradiction argument to
obtain a quantitative lower bound for the measure of the intermediate
phase.

\begin{lem}
\label{lem:two-phase-no-jump}
Let \(u_j,w_j,B_j\), and \(\nu_j\) satisfy the hypotheses of
Proposition \ref{prop:weighted-measure-data-compactness} with
$s_j=0$ and $\nu_j\geq0$,
and let \(\varepsilon_j\downarrow0\).  If
\((\mathcal Q^-,\mathcal Q^+)\) is an admissible ordered pair, then it is
impossible that, for some \(\delta_-,\delta_+>0\),
\begin{align}
 \mu_j(\{u_j\leq\varepsilon_j\}\cap\mathcal Q^-)
 &\geq
 \delta_-\mu_j(\mathcal Q^-),
 \label{eq:no-jump-low-phase}\\
 \mu_j(\{u_j\geq1-\varepsilon_j\}\cap\mathcal Q^+)
 &\geq
 \delta_+\mu_j(\mathcal Q^+),
 \label{eq:no-jump-high-phase}
\end{align}
while
\begin{equation}
 \mu_j
 \bigl(
 \{\varepsilon_j<u_j<1-\varepsilon_j\}\cap\mathcal Q'
 \bigr)
 \longrightarrow0.
 \label{eq:no-jump-intermediate-phase}
\end{equation}
\end{lem}

\begin{proof}
Suppose, to the contrary, that such a sequence exists.  Write
$D'=D_v'\times D_z'$, with the $z$-factor omitted when $k=0$.
Proposition \ref{prop:weighted-measure-data-compactness} gives, after
extraction, $u_j\to u_\infty$ strongly in unweighted local
$\mathrm L^2$ and strongly with respect to the varying measures
$\d\mu_j=w_j(x)\d x\d y\d t$.  It also gives
\begin{equation}
 Y(w_\infty u_\infty)
 =\operatorname{div}_xJ_\infty-\nu_\infty,
 \qquad J_\infty\in\mathrm L^p_{\rm loc},\quad
 \nu_\infty\geq0,
 \label{eq:binary-limit-equation}
\end{equation}
for some $p>1$, with $w_\infty>0$ almost everywhere.

The disappearance of the intermediate phase implies
\[
 \int_{\mathcal Q'}\operatorname{dist}(u_j,\{0,1\})^2\d\mu_j
 \leq\varepsilon_j^2\mu_j(\mathcal Q')
 +\mu_j(\{\varepsilon_j<u_j<1-\varepsilon_j\}\cap\mathcal Q')
 \longrightarrow0.
\]
The distance function is Lipschitz, so weighted strong convergence gives
the same conclusion with $u_j$ replaced by $u_\infty$.  Put $q=\operatorname{dist}(u_\infty,\{0,1\})$.
We have just proved that
$\int_{\mathcal Q'}q^2\d\mu_j\to0$.
To remove the weight, write $q=(qw_j^{1/2})w_j^{-1/2}$
and apply Cauchy-Schwarz:
\begin{align*}
 \int_{\mathcal Q'}q\d x\d y\d t
 &\leq
 \left(\int_{\mathcal Q'}q^2\d\mu_j\right)^{1/2}
 \left(\int_{\mathcal Q'}w_j^{-1}\d x\d y\d t\right)^{1/2}\leq C
 \left(\int_{\mathcal Q'}q^2\d\mu_j\right)^{1/2}
 \longrightarrow0.
\end{align*}
The second inequality follows from the uniform
inverse-weight bound \eqref{eq:full-rh}.
The integral on the left does not depend on $j$,
so it must equal zero.  Since $q\geq0$, this implies
$q=0$ almost everywhere in $\mathcal Q'$.
By the definition of $q$, we conclude that
\begin{equation}
 u_\infty\in\{0,1\}
 \qquad\hbox{almost everywhere in }\mathcal Q'.
 \label{eq:binary-limit}
\end{equation}

By \eqref{eq:full-gradient-integrability}, the limit belongs to
$\mathrm L^p(G;\mathrm W^{1,p}(D'))$ for some $p>1$.
Apply the Sobolev chain rule to $\beta(s)=s(1-s)$.  Since
$\beta(u_\infty)=0$,
\[
 0=\nabla_x\beta(u_\infty)
  =(1-2u_\infty)\nabla_xu_\infty.
\]
The factor $1-2u_\infty$ takes only the values $-1$ and $1$.
Connectedness of $D'$ therefore gives
\begin{equation}
 u_\infty(x,y,t)=\chi_E(y,t)
 \qquad\hbox{almost everywhere in }\mathcal Q'
 \label{eq:binary-limit-indicator}
\end{equation}
for a measurable set $E\subset G$.

The endpoint phases persist in the limit.  Put
$D^\pm=D_v^\pm\times D_z^\pm$.  Normalization on $D$ and
\eqref{eq:full-rh} give
\[
 0<c_\pm\leq w_j(D^\pm)\leq1,
 \qquad
 0<c\leq\mu_j(\mathcal Q^\pm)\leq C.
\]
For the lower bound, use
$|D^\pm|^2\leq w_j(D^\pm)\int_Dw_j^{-1}\d x$.
The past phase assumption gives
\[
 \int_{\mathcal Q^-}(1-u_j)\d\mu_j
 \geq(1-\varepsilon_j)\delta_-\mu_j(\mathcal Q^-).
\]
Weighted strong convergence gives a positive lower bound with
$u_\infty$ in place of $u_j$.  But
\[
 \int_{\mathcal Q^-}(1-u_\infty)\d\mu_j
 =w_j(D^-)|E^-\setminus E|\leq|E^-\setminus E|.
\]
The same argument applied to the future phase gives
\begin{equation}
 |E^-\setminus E|>0,\qquad |E^+\cap E|>0.
 \label{eq:binary-endpoint-phases}
\end{equation}

We next recover all admissible transport directions.  For
$0\leq\theta\in C_0^\infty(D')$, $\theta\not\equiv0$, define
\begin{equation}
 m_\theta:=\int_{D'}\theta w_\infty\d x>0,
 \qquad
 c_\theta:=\frac1{m_\theta}
                  \int_{D'}v\theta w_\infty\d x.
 \label{eq:weighted-barycentre}
\end{equation}
Testing \eqref{eq:binary-limit-equation} against
$\theta(x)\eta(y,t)$, and using \eqref{eq:binary-limit-indicator}, gives
\begin{equation}
 (\partial_t+c_\theta\cdot\nabla_y)\chi_E
 =g_\theta-\sigma_\theta
 \qquad\hbox{in }\mathcal D'(G),
 \label{eq:binary-directional-equation}
\end{equation}
where
\[
 g_\theta(y,t):=-\frac1{m_\theta}
   \int_{D'}J_\infty(x,y,t)\cdot\nabla_x\theta(x)\d x
   \in\mathrm L^p_{\rm loc}(G)
\]
and the non-negative Radon measure $\sigma_\theta$ is defined by
\[
 \int_G\eta\d\sigma_\theta
 :=\frac1{m_\theta}\int_{D'\times G}\theta(x)\eta(y,t)\d\nu_\infty.
\]
All components of the diffusion flux are retained in $g_\theta$.

If the active projection of $\operatorname{supp}\theta$ lies in
$B_\delta(v_0)$, then
\begin{equation}
 |c_\theta-v_0|\leq\delta.
 \label{eq:barycentre-localization}
\end{equation}
Choose $m+1$ affinely independent velocities $v_0,\ldots,v_m\in V$. Thus $v_1-v_0,\ldots,v_m-v_0$ are linearly independent in $\R^m$.
Since affine independence is stable under sufficiently small
perturbations, we may choose $\delta>0$ so small that
$B_\delta(v_i)\Subset V$ for every $i$ and any points
$c_i\in\overline{B_\delta(v_i)}$, $i=0,\ldots,m$, remain affinely
independent.  For each $i$, choose a non-negative smooth profile
$\theta_i$ as above, not identically zero, whose active support lies
in $B_\delta(v_i)$, using a fixed compact passive support when $k>0$.
Set $c_i:=c_{\theta_i}$.  By
\eqref{eq:barycentre-localization}, these weighted velocity averages
are affinely independent.

Since $g_{\theta_i}\in\mathrm L^p_{\rm loc}(G)$, it is locally
integrable and defines a locally finite signed Radon measure
$g_{\theta_i}\d y\d t$.  Equation
\eqref{eq:binary-directional-equation} therefore shows that each
$(\partial_t+c_i\cdot\nabla_y)\chi_E$ is a locally finite signed
Radon measure.  Affine independence of $c_0,\ldots,c_m$ means that
$(c_0,1),\ldots,(c_m,1)$ form a basis of $\R^{m+1}$.
Consequently, each coordinate derivative in $(y,t)$ is a fixed
linear combination of the operators
$\partial_t+c_i\cdot\nabla_y$.  Every distributional coordinate
derivative of $\chi_E$ is therefore a locally finite signed Radon
measure, and hence $\chi_E\in\mathrm{BV}_{\rm loc}(G)$.

Write $D\chi_E=D_{(y,t)}\chi_E$ for the vector-valued Radon measure
representing the distributional gradient of $\chi_E$ in $(y,t)$.
Since $\chi_E$ is a characteristic function, its approximate gradient
vanishes almost everywhere.  The absolutely continuous part of
$D\chi_E$ therefore vanishes, so $D\chi_E$ is singular with respect
to Lebesgue measure, see
\cite[Chapter~3]{AmbrosioFuscoPallara2000}.
For each fixed profile $\theta$, the directional derivative
$(\partial_t+c_\theta\cdot\nabla_y)\chi_E=(c_\theta,1)\cdot D\chi_E$
is consequently a singular measure as well.

Decompose $\sigma_\theta=\sigma_\theta^{\rm ac}+\sigma_\theta^{\rm s}$
into its absolutely continuous and singular parts with respect to
Lebesgue measure on $G$.  Since $g_\theta$ is locally integrable,
$g_\theta\d y\d t$ is absolutely continuous.  Thus the absolutely
continuous part of the right-hand side of
\eqref{eq:binary-directional-equation} is
$g_\theta\d y\d t-\sigma_\theta^{\rm ac}$, whereas its singular
part is $-\sigma_\theta^{\rm s}$.  The left-hand side has no
absolutely continuous part.  By uniqueness of the Lebesgue
decomposition, the absolutely continuous term on the right must
vanish and the singular parts must agree.  Hence
\[
 g_\theta\d y\d t=\sigma_\theta^{\rm ac},
 \qquad
 (\partial_t+c_\theta\cdot\nabla_y)\chi_E
 =-\sigma_\theta^{\rm s}\leq0.
\]
The last inequality follows from the non-negativity of
$\sigma_\theta$, which implies $\sigma_\theta^{\rm s}\geq0$.
In particular, the contribution of $g_\theta$ is exactly cancelled
by the absolutely continuous part of $\sigma_\theta$, and it need not
vanish separately.

Fix $v_0\in V$ and choose non-negative, non-zero profiles $\theta_\ell$
whose active supports lie in $B_{\delta_\ell}(v_0)\Subset V$, where
$\delta_\ell\downarrow0$.  Set $c_\ell:=c_{\theta_\ell}$.
By \eqref{eq:barycentre-localization},
$|c_\ell-v_0|\leq\delta_\ell$, so $c_\ell\to v_0$.
For every fixed non-negative $\varphi\in C_0^\infty(G)$, the
directional inequality proved above gives
\[
-\int_G\chi_E(\partial_t\varphi+c_\ell\cdot\nabla_y\varphi)
\d y\d t\leq0.
\]
Since $\chi_E$ is bounded and $\varphi$ has compact support,
we may pass to the limit as $\ell\to\infty$.  Consequently,
\begin{equation}
 (\partial_t+v_0\cdot\nabla_y)\chi_E\leq0
 \qquad\hbox{in }\mathcal D'(G),\quad v_0\in V.
 \label{eq:all-directional-monotonicity}
\end{equation}
Note that this argument requires no uniform bound on $g_{\theta_\ell}$
as the supports shrink. Instead, we pass to the limit in the directional
inequalities, whose sign has already been established for each fixed
profile.

Let $\rho_\varepsilon$ be a non-negative standard mollifier in $(y,t)$
and set $\chi_{E,\varepsilon}=\chi_E*\rho_\varepsilon$ on the interior
where the convolution is defined.  Convolution preserves
\eqref{eq:all-directional-monotonicity}.  Fix Lebesgue points
$(y_-,t_-)\in E^-$ and $(y_+,t_+)\in E^+$.  By admissibility,
\[
 v_{-+}=\frac{y_+-y_-}{t_+-t_-}\in V,
\]
and the joining characteristic segment is compactly contained in $G$.
For sufficiently small $\varepsilon$,
\[
 \frac{\d}{\d s}\chi_{E,\varepsilon}
   \bigl(y_-+(s-t_-)v_{-+},s\bigr)
 = (\partial_t+v_{-+}\cdot\nabla_y)\chi_{E,\varepsilon}
   \bigl(y_-+(s-t_-)v_{-+},s\bigr)\leq0.
\]
Integrating and letting $\varepsilon\downarrow0$ gives
\begin{equation}
 \chi_E(y_+,t_+)\leq\chi_E(y_-,t_-)
 \label{eq:no-upward-jump}
\end{equation}
for almost every pair of endpoints.  This contradicts
\eqref{eq:binary-endpoint-phases}, since on a set of pairs of positive
measure the left side equals one and the right side equals zero.
\end{proof}

\subsection{The weighted intermediate-value lemma}

We now pass from the qualitative two-phase rigidity of Lemma
\ref{lem:two-phase-no-jump} to a quantitative intermediate-value
estimate for a single supersolution.  Heuristically, a non-negative
supersolution which is large on a substantial part of a past box
and sufficiently small on a substantial part of a future box must
take intermediate values on a definite part of the surrounding
region.  All proportions are measured with respect to $\muW$.
The weight is fixed for each application, but the constants are
uniform over all $w\in A_2(\R^n)$ with $[w]_{A_2}\leq M$.

Recall that $Q^-$ and $Q^+$ are the fixed normalized past and
future reference boxes defined in \eqref{eq:Qminus} and
\eqref{eq:Qplus}.  In particular,
\[
 Q^+
 =
 B_{1/4}^{v}\times B_{1/4}^{z}
 \times B_{1/64}^{y}\times(-1/4,-1/8).
\]
Set
\[
 \widehat Q^+
 :=B_{3/8}^v\times B_{3/8}^z\times B_{1/32}^y
       \times(-9/32,-3/32),
 \qquad
 Q^{\rm mid}:=B_1^v\times B_1^z\times B_1^y
       \times(-7/8,-1/16).
\]
As usual, omit the $z$-factor when $k=0$.  Then
$Q^+\Subset\widehat Q^+$ and
$Q^-\cup\widehat Q^+\Subset Q^{\rm mid}\Subset Q_{R_*}$.
The connecting velocities for $(Q^-,\widehat Q^+)$ have
norm less than $3/14$, so we may take $V=B_{1/2}^v$.
The joining characteristic segments remain in the $(y,t)$
factor of $Q^{\rm mid}$.  These choices fix the geometric
margins for the following proposition and the subsequent
normalized measure-to-point argument.

\begin{prop}
\label{prop:weighted-intermediate-value}
For every \(\delta_-,\delta_+\in(0,1)\), there exist
\(\theta,\nu\in(0,1)\), depending only on
\[
 m,k,M,\lambda,\Lambda,\delta_-,\delta_+,
\]
such that the following holds.  If \(f\geq0\) is an energy supersolution in
\(Q_{R_*}\) and
\begin{align}
 \muW(\{f\geq1\}\cap Q^-)
 &\geq
 \delta_-\muW(Q^-),
 \label{eq:intermediate-high-past}\\
 \muW(\{f\leq\theta\}\cap\widehat Q^+)
 &\geq
 \delta_+\muW(\widehat Q^+),
 \label{eq:intermediate-low-future}
\end{align}
then
\begin{equation}
 \muW(\{\theta<f<1\}\cap Q^{\rm mid})
 \geq
 \nu\muW(Q^{\rm mid}).
 \label{eq:weighted-intermediate-value}
\end{equation}
The assertion also holds after intrinsic translation and dilation,
with the same structural dependence of the constants.
More generally, it holds for any fixed admissible ordered pair of
product boxes, together with intermediate and outer regions satisfying
the preceding compact-containment conditions.  In this case, the
constants may depend additionally on the normalized geometry, but
can be chosen uniformly over compact families of configurations
in the sense of Remark \ref{rem:compact-configurations} below.
\end{prop}

\begin{proof}
Suppose that the assertion is false.  Applying its negation with candidate
constants tending to zero, we obtain weights \(w_j\), coefficients
\(B_j\), non-negative supersolutions \(f_j\), and numbers
\(\theta_j\downarrow0\) such that
\[
 [w_j]_{A_2(\R^n)}\leq M,
\]
the two endpoint assumptions
\eqref{eq:intermediate-high-past}-\eqref{eq:intermediate-low-future} hold
with \(f,\theta\), and \(\muW\) replaced by
\(f_j,\theta_j\), and \(\mu_j\), while
\begin{equation}
 \frac{
 \mu_j(\{\theta_j<f_j<1\}\cap Q^{\rm mid})
 }{
 \mu_j(Q^{\rm mid})
 }
 \longrightarrow0.
 \label{eq:failure-intermediate-value}
\end{equation}

Choose fixed product cylinders
\[
 Q^{\rm mid}\Subset\mathcal Q_0\Subset\mathcal Q_1
 \Subset\mathcal Q_2\Subset Q_{R_*}.
\]
Multiplying each weight by a positive constant does not change its
\(A_2\) characteristic, the phase ratios, or the supersolution inequality.
We therefore normalize $w_j$ on the full diffusive factor of
$\mathcal Q_0$ as in \eqref{eq:compactness-weight-normalization}.  Weighted doubling then gives
uniform upper and lower bounds for the weighted measures of all the fixed
boxes appearing in the argument.  In particular,
\begin{equation}
 0<c\leq\mu_j(Q^{\rm mid})\leq C<\infty.
 \label{eq:uniform-middle-box-measure}
\end{equation}

Set
\[
 u_j:=1-\min\{f_j,1\}.
\]
The concave Kato inequality shows that
\(\min\{f_j,1\}\) remains a supersolution.  Hence \(u_j\) is a
subsolution satisfying \(0\leq u_j\leq1\).  Equivalently, there is a
non-negative Radon measure \(\nu_j\) such that
\begin{equation}
 \operatorname{div}_x
 \bigl(w_jB_j\nabla_xu_j\bigr)
 -Y(w_ju_j)
 =
 \nu_j
 \qquad\hbox{in }\mathcal D'(\mathcal Q_2).
 \label{eq:truncated-defect-equation}
\end{equation}
Here we use the standard fact that every non-negative distribution is a
Radon measure.

We now verify the uniform bounds required by Proposition
\ref{prop:weighted-measure-data-compactness}.  The Caccioppoli inequality,
applied with a cutoff equal to one on \(\mathcal Q_1\) and supported in
\(\mathcal Q_2\), gives
\begin{equation}
 \int_{\mathcal Q_1}|\nabla_xu_j|^2\d\mu_j\leq C.
 \label{eq:truncated-uniform-energy}
\end{equation}
Indeed, \(0\leq u_j\leq1\), and weighted doubling controls
\(\mu_j(\mathcal Q_1)\) after the preceding normalization.

To control the Radon measures, choose
\(0\leq\eta\in C_0^\infty(\mathcal Q_1)\) with
\(\eta\equiv1\) on \(\mathcal Q_0\).  Testing
\eqref{eq:truncated-defect-equation} against \(\eta\), and using the
formal skew-adjointness of \(Y\) with respect to \(\d\mu_j\), gives
\begin{align*}
 \nu_j(\mathcal Q_0)
 &\leq
 \int_{\mathcal Q_1}\eta\d\nu_j\leq
 \int_{\mathcal Q_1}
 |B_j\nabla_xu_j|\,|\nabla_x\eta|\d\mu_j
 +
 \int_{\mathcal Q_1}
 u_j|Y\eta|\d\mu_j.
\end{align*}
The right-hand side is uniformly bounded by
\eqref{eq:truncated-uniform-energy}, the coefficient bounds,
\(0\leq u_j\leq1\), weighted doubling, and the fixed cutoff estimates.
Thus
\begin{equation}
 \sup_j\left(
 \int_{\mathcal Q_0}|\nabla_xu_j|^2\d\mu_j
 +
 \nu_j(\mathcal Q_0)
 \right)<\infty.
 \label{eq:truncated-compactness-bounds}
\end{equation}
Proposition \ref{prop:weighted-measure-data-compactness} therefore applies
on \(\mathcal Q_0\), with inner set \(Q^{\rm mid}\) and \(s_j=0\).

We now translate the phase information for \(f_j\) into the corresponding
information for \(u_j\).  On \(\{f_j\geq1\}\) we have \(u_j=0\), while on
\(\{f_j\leq\theta_j\}\) we have
\[
 u_j=1-f_j\geq1-\theta_j.
\]
Moreover,
\begin{equation}
 \{\theta_j<f_j<1\}
 =
 \{0<u_j<1-\theta_j\}.
 \label{eq:transition-set-under-truncation}
\end{equation}
Consequently,
\[
 \{\theta_j<u_j<1-\theta_j\}
 \subset
 \{\theta_j<f_j<1\}.
\]
By \eqref{eq:failure-intermediate-value} and
\eqref{eq:uniform-middle-box-measure},
\[
 \mu_j
 \bigl(
 \{\theta_j<u_j<1-\theta_j\}\cap Q^{\rm mid}
 \bigr)
 \longrightarrow0.
\]
The past assumption gives
\[
 \mu_j(\{u_j\leq\theta_j\}\cap Q^-)
 \geq
 \delta_-\mu_j(Q^-),
\]
because \(u_j=0\) on \(\{f_j\geq1\}\).  Similarly, the future assumption
gives
\[
 \mu_j
 \bigl(
 \{u_j\geq1-\theta_j\}\cap\widehat Q^+
 \bigr)
 \geq
 \delta_+\mu_j(\widehat Q^+).
\]
Thus all the hypotheses of Lemma \ref{lem:two-phase-no-jump} hold with
\[
 \varepsilon_j=\theta_j,
 \qquad
 \mathcal Q^-=Q^-,
 \qquad
 \mathcal Q^+=\widehat Q^+,
 \qquad
 \mathcal Q'=Q^{\rm mid}.
\]
This contradicts that lemma and proves
\eqref{eq:weighted-intermediate-value}.

Finally, intrinsic translations and dilations preserve independence of
the weight from $(y,t)$ and its $A_2$ characteristic, as well as the
normalized past-future geometry and all phase ratios.  The same
argument therefore applies at every center and scale with the same
structural constants.

For a compact family as in Remark \ref{rem:compact-configurations} below,
suppose that uniformity fails and choose a contradicting sequence
with varying geometric parameters.  After passing to a subsequence,
these parameters converge.  The uniform positive margins ensure
that the limiting configuration remains admissible.  Choose slightly
smaller endpoint boxes and a common outer work region within these
margins, so that the required containments hold for all sufficiently
large indices.

Normalize each weight to have integral one on a fixed diffusive ball
containing the diffusive work regions.  The $A_2$ inequality then gives
uniform positive lower bounds for the weighted measures of the
endpoint boxes.  The discarded boundary layers have arbitrarily
small Lebesgue measure when the shrinkage is sufficiently small and
the index sufficiently large.  Quantitative $A_\infty$ absolute
continuity and the product structure of the measure make their
weighted measures uniformly small as well.  We may therefore choose
the smaller endpoint boxes so that the phase bounds persist with,
for example, half their original proportions.
On the fixed common diffusive work domain $D$, the $A_2$ bounds give
$c\leq\int_Dw_j\d x\leq C$.  We may therefore renormalize on $D$
as in \eqref{eq:compactness-weight-normalization}, preserving all
required uniform bounds and phase ratios.

Apply the preceding compactness argument on a fixed inner work
region $D'\times G'$, compactly contained in all sufficiently late
middle regions, with $D'$ connected and $G'$ containing the smaller
endpoint factors and their joining characteristic segments.
The existence of this common region follows from the uniform
margins.  The surviving endpoint phases give the same contradiction
as above, proving uniformity over the compact family.
\end{proof}

\begin{rem}
\label{rem:compact-configurations}
In Proposition \ref{prop:weighted-intermediate-value} and below, a compact family of normalized configurations is
understood to satisfy the following uniformity conditions.
The spatial centers, radii or side lengths, and time endpoints
range over compact parameter sets, while the spatial sizes and
time gaps are bounded away from zero.  All compact containments
hold with uniform positive margins.  In particular, the joining
characteristic segments stay uniformly away from the boundary
of the transported work region, and the required diffusive
containments involving the active velocity ball have uniform
margins.  The prescribed margins between the phase, intermediate,
and outer regions are also uniform.
\end{rem}

\section{Expansion of positivity}
\label{Section_7}

We now derive expansion of positivity from Proposition
\ref{prop:weighted-intermediate-value} and the local boundedness
estimate. The passage from the intermediate-value principle to
a measure-to-point estimate follows the standard De Giorgi
argument, compare \cite[Section~3.3]{GuerandMouhot2022} and
\cite[Section~4.8]{Imbert2026}. The subsequent propagation and
covering arguments follow the stacked-cylinder and ink-spots
strategy in \cite[Sections~4.9-4.10]{Imbert2026}.
We include the details for completeness and to verify that the
arguments apply to $\d\muW$, with constants uniform over
$[w]_{A_2}\leq M$.

In \cite[Section~4.7]{Imbert2026}, the intermediate-value principle
is obtained from a weak kinetic Poincar\'e-Wirtinger inequality.
Here this principle has already
been proved by the weighted compactness argument in Section
\ref{Section_6}, so no separate kinetic Poincar\'e inequality
is needed in this section. We first convert information in
weighted measure on a past box into a pointwise lower bound
on a later box. We then introduce the geometric constructions
needed to propagate this bound through admissible ordered
pairs, keeping track of the containment margins and the number
of applications.

The weighted measure comparisons require explicit verification.
The $A_2$ condition gives uniform comparisons between the root
cells and their surrounding boxes, while weighted doubling and
quantitative $A_\infty$ control the covering estimates and
boundary layers. In the kinetic ink-spots argument, we hold
$x=(v,z)$ fixed and use free-transport coordinates in $(y,t)$.
The change of variables has Jacobian one and preserves
$\d\muW$, since $w$ is independent of $(y,t)$. Thus this step
allows the weight to depend on all diffusive variables.
The resulting covering lemma prepares the iteration over
superlevel sets in Section \ref{Section_8}.

\subsection{From measure to point}

We first record a consequence of local boundedness: a
non-negative subsolution with sufficiently small weighted
$\mathrm L^2$ average is uniformly small on a compactly
contained box.

\begin{lem}
\label{lem:small-subsolution-point}
Let
\[
 \mathcal V'
 =
 D_v'\times D_z'\times D_y'\times I',
 \qquad
 \mathcal V
 =
 D_v\times D_z\times D_y\times I,
\]
where \(D_v',D_v\subset\R^m\), \(D_z',D_z\subset\R^k\), and
\(D_y',D_y\subset\R^m\) are bounded open balls, and \(I',I\subset\R\)
are bounded open intervals.  Assume that
\[
 D_v'\Subset D_v,
 \qquad
 D_z'\Subset D_z,
 \qquad
 D_y'\Subset D_y,
 \qquad
 I'\Subset I,\qquad \mathcal V\Subset Q_{R_*}.
\]
Then there exists \(\varepsilon_*>0\), depending only on the structural
parameters and on the fixed geometric data of the pair
\((\mathcal V',\mathcal V)\), such that every non-negative energy
subsolution \(g\) in \(\mathcal V\) satisfying
\[
 \avgint_{\mathcal V}g^2\,\d\muW
 \leq
 \varepsilon_*
\]
satisfies
\[
 \esssup_{\mathcal V'}g\leq 1/2.
\]
\end{lem}

\begin{proof}
Since \(\overline{\mathcal V'}\) is compactly contained in
\(\mathcal V\), it can be covered by finitely many intrinsic cylinders
\[
 \overline{\mathcal V'}
 \subset
 \bigcup_{i=1}^N Q_{\rho_i}(P_i)
\]
for which there exist radii \(R_i>\rho_i\) satisfying $Q_{R_i}(P_i)\Subset\mathcal V$. All the centers, radii, and containment ratios may be fixed in terms of
the geometry of \((\mathcal V',\mathcal V)\). Proposition \ref{prop:localboundedness}, applied with \(p=2\) in each
\(Q_{R_i}(P_i)\), gives
\[
 \esssup_{Q_{\rho_i}(P_i)}g
 \leq
 C_i
 \left(
 \avgint_{Q_{R_i}(P_i)}g^2\,\d\muW
 \right)^{1/2},
\]
where $C_i=C_{2,\rho_i/R_i}$ depends only on the structural parameters and the fixed ratio
\(\rho_i/R_i\). Since \(Q_{R_i}(P_i)\subset\mathcal V\),
\begin{align*}
 \avgint_{Q_{R_i}(P_i)}g^2\,\d\muW
 &\leq
 \frac{\muW(\mathcal V)}
      {\muW(Q_{R_i}(P_i))}
 \avgint_{\mathcal V}g^2\,\d\muW.
\end{align*}
The sets involved are fixed.  Comparing their diffusive projections by
Euclidean weighted doubling and their \((y,t)\)-sections by the fixed
Lebesgue ratios gives
\[
 \frac{\muW(\mathcal V)}
      {\muW(Q_{R_i}(P_i))}
 \leq C
\]
uniformly in \(i\), where \(C\) depends only on the structural parameters
and the geometry of \((\mathcal V',\mathcal V)\).  Since the covering is
finite, it follows that
\[
 \esssup_{\mathcal V'}g
 \leq
 C_*
 \left(
 \avgint_{\mathcal V}g^2\,\d\muW
 \right)^{1/2}
\]
for a constant \(C_*\) with the same dependence.  Choosing $\varepsilon_*\leq {1}/{(4C_*^2)}$ proves the assertion.
\end{proof}

The next proposition converts positivity on a fixed weighted
proportion of the past box into a uniform positive lower bound
almost everywhere in the future box.  The proof combines the
intermediate-value principle with the preceding smallness
estimate for subsolutions.

\begin{prop}
\label{prop:measure-to-point}
Let
\begin{align*}
 Q^-
 &=
 B_{1/4}^{v}\times B_{1/4}^{z}
 \times B_{1/64}^{y}\times(-3/4,-1/2),\\
 Q^+
 &=
 B_{1/4}^{v}\times B_{1/4}^{z}
 \times B_{1/64}^{y}\times(-1/4,-1/8)
\end{align*}
be the fixed normalized past and future reference boxes introduced in
\eqref{eq:Qminus}-\eqref{eq:Qplus}.  When \(k=0\), the \(z\)-factors
are omitted.  For every \(\delta\in(0,1)\), there exists
\[
 \ell=\ell(m,k,M,\lambda,\Lambda,\delta)\in(0,1)
\]
such that every non-negative energy supersolution \(f\) in \(Q_{R_*}\)
satisfying
\begin{equation}
 \muW(\{f\geq1\}\cap Q^-)
 \geq
 \delta\,\muW(Q^-)
 \label{eq:mtp-hypothesis}
\end{equation}
obeys
\begin{equation}
 \essinf_{Q^+}f\geq\ell.
 \label{eq:mtp-conclusion}
\end{equation}
More generally, let $Q_r^\pm(P_0):=T_{P_0,r}(Q^\pm)$. If \(f\) is a non-negative energy supersolution in
\(Q_{R_*r}(P_0)\) and
\[
 \muW(\{f\geq1\}\cap Q_r^-(P_0))
 \geq
 \delta\,\muW(Q_r^-(P_0)),
\]
then
\[
 \essinf_{Q_r^+(P_0)}f\geq\ell
\]
with the same constant \(\ell\).
More generally, for every fixed compact family \(\mathfrak P\) with uniform positive margins of admissible
ordered configurations, the same conclusion holds uniformly for all members
of \(\mathfrak P\) and their intrinsic translates and dilates, with
\(\ell=\ell(m,k,M,\lambda,\Lambda,\delta,\mathfrak P)>0\), provided the
required outer enlargement lies in the solution domain.
\end{prop}

\begin{proof}
Let \(\widehat Q^+\) be the fixed enlarged future box satisfying
\[
 Q^+\Subset\widehat Q^+\Subset Q^{\rm mid}.
\]
Apply Proposition \ref{prop:weighted-intermediate-value} with $(\delta_-,\delta_+)=(\delta,\varepsilon_*)$, where \(\varepsilon_*\) is the constant in Lemma
\ref{lem:small-subsolution-point} for the fixed pair
\[
 \mathcal V'=Q^+,
 \qquad
 \mathcal V=\widehat Q^+.
\]
Let \(\theta,\nu\in(0,1)\) be the constants supplied by the
intermediate-value proposition.

For each integer $j\geq0$, define $f_j:=\theta^{-j}f$.
Since $f_j$ is a positive constant multiple of $f$, it is
again a non-negative energy supersolution in $Q_{R_*}$.
Set $N:=\lceil\nu^{-1}\rceil$.  We show that the low-value
set has sufficiently small weighted measure at one of the
levels indexed by $j=0,\ldots,N$.

Suppose, to the contrary, that
\begin{equation}
 \muW(\{f\leq\theta^{j+1}\}\cap\widehat Q^+)
 \geq
 \varepsilon_*\muW(\widehat Q^+)
 \label{eq:mtp-low-level}
\end{equation}
for every $j=0,\ldots,N$.
Since $\{f\geq1\}\subset\{f_j\geq1\}$ and
$\{f_j\leq\theta\}=\{f\leq\theta^{j+1}\}$, Proposition
\ref{prop:weighted-intermediate-value} applies to each $f_j$
and gives
\begin{equation}
 \muW(\{\theta^{j+1}<f<\theta^j\}\cap Q^{\rm mid})
 \geq\nu\muW(Q^{\rm mid}).
 \label{eq:mtp-disjoint-strip}
\end{equation}
These level strips are pairwise disjoint.  Summing
\eqref{eq:mtp-disjoint-strip} over $j=0,\ldots,N$ therefore
gives $1\geq(N+1)\nu>1$, a contradiction.
Hence there exists $j_0\in\{0,\ldots,N\}$ such that
\begin{equation}
 \muW(\{f\leq\theta^{j_0+1}\}\cap\widehat Q^+)
 <
 \varepsilon_*\muW(\widehat Q^+).
 \label{eq:mtp-failed-low-level}
\end{equation}

Consider the truncation
\[
 g:=\left(1-\theta^{-(j_0+1)}f\right)_+.
\]
By the standard truncation property, $g$ is a non-negative
energy subsolution in $\widehat Q^+$.  Moreover,
\[
 0\leq g\leq1,
 \qquad
 \{g>0\}=\{f<\theta^{j_0+1}\}.
\]
Thus $g^2\leq\mathbf1_{\{f<\theta^{j_0+1}\}}$, and
\eqref{eq:mtp-failed-low-level} implies
\[
 \avgint_{\widehat Q^+}g^2\d\muW<\varepsilon_*.
\]
Lemma \ref{lem:small-subsolution-point}, applied to
$Q^+\Subset\widehat Q^+$, now yields
\[
 \esssup_{Q^+}g\leq1/2.
\]
Consequently, $f\geq\theta^{j_0+1}/2$ almost everywhere
in $Q^+$.  Since $j_0\leq N$ and $0<\theta<1$, we obtain
\[
 \essinf_{Q^+}f
 \geq\theta^{j_0+1}/2
 \geq\theta^{N+1}/2.
\]
The normalized conclusion follows with
$\ell:=\theta^{N+1}/2$.

Finally, intrinsic translations and dilations preserve the form of the
equation, independence of the weight from $(y,t)$, its $A_2$ characteristic,
and all normalized phase ratios.  Applying the normalized result to
\(f\circ T_{P_0,r}\) proves the translated and dilated assertion.  For a
fixed compact family with uniform margins, use the uniform constants in
Proposition \ref{prop:weighted-intermediate-value} and the uniform local
boundedness estimate.
\end{proof}

\subsection{Auxiliary kinetic geometry}
\label{subsec:auxiliary-kinetic-geometry}

We collect the geometric facts used in the expansion of
positivity and the subsequent covering argument. The constructions
depend only on the Kolmogorov group structure and its translations
and dilations. The weighted measure comparisons follow from the
$A_2$ condition, weighted doubling, and quantitative $A_\infty$.
The proofs and accompanying illustrations are given in Appendix
\ref{app:positivity-figures}.

For an intrinsic cylinder $Q_r(P)$, with
$P=(v_P,z_P,y_P,t_P)$, and an integer $L\geq1$, define
the associated forward stack by
\begin{equation}
 \overline Q_r^{\,L}(P)
 :=
 \left\{
 \begin{array}{l}
 |(v,z)-(v_P,z_P)|<r,\\[2pt]
 |y-y_P-(t-t_P)v_P|<(L+2)r^3,\\[2pt]
 t_P<t<t_P+Lr^2
 \end{array}
 \right\}.
 \label{eq:stacked-cylinder}
\end{equation}
The factor \(L+2\) in the \(y\)-radius allows for the kinetic drift over
the time interval of length \(Lr^2\).

We also fix, once and for all, a past work cylinder
$\mathcal Q_{\rm past}=Q_{R_{\rm p}}(P_{\rm p})$, with
$R_{\rm p}=1$ and $P_{\rm p}=(0,0,0,-3/8)$, so that
\begin{equation}
 Q^-\Subset\mathcal Q_{\rm past}\Subset Q_{R_*},
 \label{eq:past-work-cylinder}
\end{equation}
and the closure of $\mathcal Q_{\rm past}$ lies strictly before $Q^+$ in time.  The normalized
geometry is chosen so that all chaining boxes and their required outer
enlargements remain in \(Q_{R_*}\).  \(Q^-\) and \(Q^+\) continue to denote the fixed normalized past and future
reference boxes defined in \eqref{eq:Qminus} and \eqref{eq:Qplus}.

The first geometric fact concerns differentiation, covering, and doubling.

\begin{lem}
\label{lem:kinetic-covering-geometry}
Intrinsic cylinders differentiate $\mathrm L^1_{\rm loc}(\d\muW)$.
For a family of bounded radii, a pairwise disjoint subfamily can be chosen
whose enlargements
\begin{equation}
 Q_r^*(P):=Q_{5r}(\Phi_{12r^2}(P)),
 \qquad \Phi_s(P)=(v_P,z_P,y_P+sv_P,t_P+s),
 \label{eq:covering-dilation}
\end{equation}
cover its union.  Moreover,
\begin{equation}
 \muW(Q_{2r}(P))\leq C\muW(Q_r(P)),
 \qquad \muW(Q_r^*(P))\leq C\muW(Q_r(P)),
 \label{eq:intrinsic-measure-doubling}
\end{equation}
where $C=C(m,k,M)$.  Uniformly comparable intrinsic rectangles have the
same properties, with possibly different constants.
\end{lem}

The next lemma provides the finite geometric constructions used to
propagate positivity.  Its first construction leads to a forward
stack, and its second connects a small cylinder in
$\mathcal Q_{\rm past}$ to the fixed future box $Q^+$.
We describe these constructions using the following terminology.

A \emph{finite rooted directed tree of admissible ordered pairs}
is a finite rooted tree $\mathscr T$ in which each vertex $\beta$
carries an admissible ordered pair $(Q_\beta^-,Q_\beta^+)$.
Every non-root vertex has a unique parent, and edges are directed
from parents to children.  A vertex with no children is called
\emph{terminal}.  If $\gamma$ is a child of $\beta$, we require
\[
 Q_\gamma^-\Subset Q_\beta^+.
\]
This containment allows positivity to pass from one ordered pair
to the next.  Indeed, a lower bound $u\geq a>0$ almost everywhere
on $Q_\beta^+$ also holds on $Q_\gamma^-$.  Applying Proposition
\ref{prop:measure-to-point} at $\gamma$ then gives a positive
lower bound on $Q_\gamma^+$, although its value may be smaller.

A \emph{branch} is a directed path from the root to a terminal
vertex.  Its \emph{length} is the number of vertices along the
path, and hence the number of measure-to-point applications.
An edge records the containment needed to pass between two
successive applications. Figure \ref{fig:rooted-positivity-tree}
in Appendix \ref{app:positivity-figures} illustrates this structure.

Let \(Q_1^{\rm ear}:=Q_1\cap\{t<-1/4\}\) denote the early part of $Q_1$. Since the weight is independent
of time,
\[
 \muW(Q_1^{\rm ear})=\frac34\muW(Q_1),
 \qquad
 \muW(Q_1\setminus Q_1^{\rm ear})=\frac14\muW(Q_1).
\]

\begin{lem}
\label{lem:positivity-tree-geometry}
For every integer \(L\geq1\), there exist a finite measurable partition
\(\{E_\alpha\}\) of \(Q_1^{\rm ear}\), up to null sets, associated past boxes
\(S_\alpha^-\), a bounded open set \(\mathcal U^L\), and, for every
\(\alpha\), a finite rooted directed tree
\(\mathscr T_\alpha\) of admissible ordered pairs
\[
 \bigl\{(Q_\beta^-,Q_\beta^+):
 \beta\in\mathscr T_\alpha\bigr\}
\]
with the following properties:

\begin{enumerate}[label=\textnormal{(\roman*)}]
\item For every $\alpha$, $E_\alpha\subset S_\alpha^-$.
If $\beta_\alpha^0$ denotes the root of $\mathscr T_\alpha$,
then $Q_{\beta_\alpha^0}^-=S_\alpha^-$.
Moreover, for every $w\in A_2(\R^n)$ satisfying $[w]_{A_2}\leq M$,
\begin{equation}
 \muW(S_\alpha^-)
 \leq
 C_L\muW(E_\alpha).
 \label{eq:positivity-root-comparison}
\end{equation}
The root boxes have a uniform time margin: for some
\(\tau_*>0\), depending only on the normalized geometry,
\[
 \sup\{t:(v,z,y,t)\in S_\alpha^-\}\leq-\tau_*
\]
for every \(\alpha\).

\item If \(\gamma\) is a child of \(\beta\), then
\[
 Q_\gamma^-\Subset Q_\beta^+
\]
with a fixed positive containment margin.

\item The future members at the terminal vertices cover the forward
stack:
\[
 \overline Q_1^{\,L}(0)
 \subset
 \bigcup_{\substack{\beta\in\mathscr T_\alpha\\
                     \beta\ {\rm terminal}}}
 Q_\beta^+.
\]

\item Every box \(Q_\beta^\pm\), together with every outer cylinder
required when applying the intermediate-value principle, is contained in
\(\mathcal U^L\).  We also require
$Q_1\cup\overline Q_1^{\,L}(0)\Subset\mathcal U^L$.
The set \(\mathcal U^L\) depends only on the fixed
normalized geometry and \(L\).

\item Every intrinsic cylinder
\(Q_r(P)\subset\mathcal Q_{\rm past}\) admits an analogous finite
partition of \(T_{P,r}(Q_1^{\rm ear})\) into root cells and corresponding
rooted trees of admissible
ordered pairs whose terminal future members cover \(Q^+\).  The root
comparison is uniform over all $w\in A_2(\R^n)$ satisfying
$[w]_{A_2}\leq M$. The containment in \textnormal{(ii)} holds at
every edge, and all required outer enlargements remain in \(Q_{R_*}\).
If \(N(r)\) denotes the maximal number of vertices along a branch, then
\[
 N(r)\leq C_1+C_2\log^+\frac1r.
\]
Here $\log^+s:=\max\{\log s,0\}$.
\end{enumerate}
In \textnormal{(i)}-\textnormal{(iv)}, each tree may be chosen
to consist of the single ordered pair $(S_\alpha^-,F^L)$,
where the common future box $F^L$ contains the closure of
the forward stack. For the construction in \textnormal{(i)}-\textnormal{(iv)}, the number
of vertices, the geometric margins, and the set \(\mathcal U^L\) depend
only on \(L,m,\) and \(k\).  The constant in
\eqref{eq:positivity-root-comparison} also depends on \(M\).  The
constants in \textnormal{(v)} depend only on the fixed normalized
geometry, \(m,k,\) and \(M\).  Both constructions commute with intrinsic
translations and dilations.  After intrinsic normalization, the ordered
pairs and auxiliary enlargements in \textnormal{(i)}-\textnormal{(iv)}
belong to a finite family $\mathfrak P_L$, while those in
\textnormal{(v)} belong to a compact structural family $\mathfrak P_*$,
with uniform positive margins, independent of $P$, $r$, and the generation.
\end{lem}

Figure~\ref{fig:positivity-tree-geometry} in Appendix
\ref{app:positivity-figures} illustrates the two constructions:
panel (a) shows propagation to a forward stack, while panel (b)
shows the successive enlargements connecting a small cylinder
in $\mathcal Q_{\rm past}$ to the fixed future box $Q^+$.

\subsection{Expansion on forward stacks}

The estimates in this subsection follow from Proposition
\ref{prop:measure-to-point}, combined with the geometric
constructions in Lemma \ref{lem:positivity-tree-geometry}.
We first propagate positivity to a forward stack and then
quantify the loss when propagating it from a small cylinder
to the fixed future box $Q^+$.

Recall that the forward stack \(\overline Q_r^{\,L}(P)\) is defined in
\eqref{eq:stacked-cylinder}.  It extends forward from \(Q_r(P)\) over
a time interval of length \(Lr^2\), with the enlarged \(y\)-radius
accounting for the kinetic drift. For every integer \(L\geq1\), Lemma
\ref{lem:positivity-tree-geometry} provides a bounded open set
\(\mathcal U^L\) containing every box and every outer enlargement used in
the corresponding normalized positivity trees.  For \(P\in\R^{2m+k+1}\)
and \(r>0\), set
\[
 \mathcal U_r^L(P):=T_{P,r}(\mathcal U^L).
\]
Thus \(\mathcal U_r^L(P)\) contains the translated and dilated positivity
trees issuing from \(Q_r(P)\).

\begin{prop}
\label{prop:stacked-expansion}
For every integer \(L\geq1\), there exists
\[
 \mathsf M_L
 =
 \mathsf M_L(m,k,M,\lambda,\Lambda,L)>1
\]
such that the following holds.  Let \(A>0\), and let \(f\geq0\) be an
energy supersolution in \(\mathcal U_r^L(P)\).  If
\begin{equation}
 \muW(\{f>\mathsf M_LA\}\cap Q_r(P))
 \geq
 \frac12\muW(Q_r(P)),
 \label{eq:stacked-expansion-hypothesis}
\end{equation}
then
\begin{equation}
 \essinf_{\overline Q_r^{\,L}(P)}f\geq A.
 \label{eq:stacked-expansion-conclusion}
\end{equation}
\end{prop}

\begin{proof}
By intrinsic normalization and division by $A$, it suffices to consider
$P=0$, $r=1$, and $A=1$. Write $\widetilde f=A^{-1}f\circ T_{P,r}$.
The rescaled weight has the same $A_2$ characteristic, and weighted
phase ratios are unchanged.  We write $\muW$ for this rescaled measure.

Use the construction in Lemma \ref{lem:positivity-tree-geometry}(i)-(iv),
in which every root box $S_\alpha^-$ is paired directly with the same
future box $F^L$ containing $\overline Q_1^{\,L}(0)$.
Let $C_L\geq1$ be the root-comparison constant and put
$\delta_L=(3C_L)^{-1}$.  Proposition \ref{prop:measure-to-point}, applied
with density $\delta_L$ uniformly to these finitely many pairs, gives
$\ell_L\in(0,1)$.  Set $\mathsf M_L=\ell_L^{-1}$.

Let $E=\{\widetilde f>\mathsf M_L\}\cap Q_1$.  The hypothesis and
$\muW(Q_1\setminus Q_1^{\rm ear})=\muW(Q_1)/4$ give
\[
 \muW(E\cap Q_1^{\rm ear})
 \geq\frac14\muW(Q_1)=\frac13\muW(Q_1^{\rm ear}).
\]
Since the cells $E_\alpha$ partition $Q_1^{\rm ear}$, one of them
satisfies $\muW(E\cap E_\alpha)\geq\muW(E_\alpha)/3$.
The containment and root comparison therefore imply
\[
 \muW(\{\widetilde f>\mathsf M_L\}\cap S_\alpha^-)
 \geq\muW(E\cap E_\alpha)
 \geq\frac{1}{3C_L}\muW(S_\alpha^-).
\]
Apply Proposition \ref{prop:measure-to-point} to
$\widetilde f/\mathsf M_L$ on $(S_\alpha^-,F^L)$.
It gives $\widetilde f\geq\ell_L\mathsf M_L=1$ almost everywhere on
$F^L$, hence on the entire forward stack.  Returning to the original
variables proves \eqref{eq:stacked-expansion-conclusion}.
\end{proof}

We next propagate positivity from a cylinder of arbitrary admissible scale
in the fixed past work region to the future reference box $Q^+$
defined in \eqref{eq:Qplus}.  Recall that
\(\mathcal Q_{\rm past}\) was fixed in
\eqref{eq:past-work-cylinder}, with
\[
 Q^-\Subset\mathcal Q_{\rm past}\Subset Q_{R_*},
\]
and with its closure lying strictly before \(Q^+\) in time.  The following
estimate quantifies the loss incurred when positivity is propagated from a
small cylinder in \(\mathcal Q_{\rm past}\) to \(Q^+\).  It will later
control the radius of a cylinder on which a high superlevel set occupies
more than half of the weighted measure.

\begin{prop}
\label{prop:iterated-expansion}
There exist constants \(c_0>0\) and \(\gamma_0>0\), depending only on
\(m,k,M,\lambda,\Lambda\) and the fixed normalized geometry, such that the
following holds.  Let
\[
 Q_r(P)\subset\mathcal Q_{\rm past},
 \qquad
 A>0,
\]
and let \(f\geq0\) be an energy supersolution in \(Q_{R_*}\).  If
\begin{equation}
 \muW(\{f>A\}\cap Q_r(P))
 >
 \frac12\muW(Q_r(P)),
 \label{eq:iterated-expansion-hypothesis}
\end{equation}
then
\begin{equation}
 \essinf_{Q^+}f
 \geq
 c_0Ar^{\gamma_0}.
 \label{eq:iterated-expansion}
\end{equation}
Consequently, if \(H>1\), \(j\geq0\),
\[
 \muW(\{f>H^{j+1}\}\cap Q_r(P))
 >
 \frac12\muW(Q_r(P))\quad\mbox{and}\quad\essinf_{Q^+}f\leq1,
\]
then
\begin{equation}
 r
 \leq
 C H^{-(j+1)/\gamma_0},
 \label{eq:radius-from-future-infimum}
\end{equation}
where \(C\) has the same structural dependence.
\end{prop}

\begin{proof}
Set $h:=A^{-1}f$. Then \(h\) is a non-negative energy supersolution in \(Q_{R_*}\), and
\eqref{eq:iterated-expansion-hypothesis} becomes
\[
 \muW(\{h>1\}\cap Q_r(P))
 >
 \frac12\muW(Q_r(P)).
\]
We now apply the translated and dilated root partition from Lemma
\ref{lem:positivity-tree-geometry} to \(Q_r(P)\).  The complement of its
early root region has one quarter of the measure of \(Q_r(P)\).  Hence
there exists a root cell \(E_\alpha^{P,r}\) such that
\[
 \muW(\{h>1\}\cap E_\alpha^{P,r})
 >
 \frac13\muW(E_\alpha^{P,r}).
\]
Let \(S_\alpha^{P,r}\) be the corresponding root past box.  The uniform
root comparison gives
\[
 \muW(S_\alpha^{P,r})
 \leq
 C_{\rm root}\muW(E_\alpha^{P,r})
\]
for a structural constant \(C_{\rm root}\geq1\).  Therefore
\[
 \muW(\{h>1\}\cap S_\alpha^{P,r})
 \geq
 \delta_0\muW(S_\alpha^{P,r}),
 \qquad
 \delta_0:=\frac{1}{3C_{\rm root}}.
\]

The same geometric lemma provides a rooted tree of admissible ordered
pairs beginning with \(S_\alpha^{P,r}\), whose terminal future members
cover \(Q^+\) and whose required outer enlargements remain in
\(Q_{R_*}\).  If \(N(r)\) denotes the maximal number of vertices along a
branch, then
\begin{equation}
 N(r)
 \leq
 C_1+C_2\log^+\frac1r.
 \label{eq:positivity-chain-depth}
\end{equation}

By Lemma \ref{lem:positivity-tree-geometry}, all normalized ordered-pair
geometries belong to the compact structural family \(\mathfrak P_*\),
independently of \(P\), \(r\), and the generation.  Proposition
\ref{prop:measure-to-point}, applied uniformly to this family with past
density \(\delta_0\) at the roots and \(1/2\) at all descendants, gives a
structural constant $\ell_0\in(0,1)$. Iterating Proposition \ref{prop:measure-to-point} along every branch gives
\[
 \essinf_{Q^+}h
 \geq
 \ell_0^{N(r)}.
\]

Since $\mathcal Q_{\rm past}$ has radius $R_{\rm p}=1$, its contained
cylinders satisfy $r\leq1$.  Put $\gamma_0=-C_2\log\ell_0>0$ and
$c_0=\ell_0^{C_1}$.  Equation \eqref{eq:positivity-chain-depth} gives
\[
 \ell_0^{N(r)}
 \geq\ell_0^{C_1+C_2\log(1/r)}
 =c_0r^{\gamma_0}.
\]
Consequently,
\[
 \essinf_{Q^+}f
 =
 A\essinf_{Q^+}h
 \geq
 c_0Ar^{\gamma_0},
\]
which proves \eqref{eq:iterated-expansion}.

For the final assertion, apply \eqref{eq:iterated-expansion} with
\(A=H^{j+1}\).  The assumption
\(\essinf_{Q^+}f\leq1\) gives
\[
 1
 \geq
 c_0H^{j+1}r^{\gamma_0}\Longleftrightarrow r
 \leq
 c_0^{-1/\gamma_0}
 H^{-(j+1)/\gamma_0},
\]
which is \eqref{eq:radius-from-future-infimum} after setting
\(C=c_0^{-1/\gamma_0}\).
\end{proof}

\subsection{A weighted kinetic ink-spots lemma}

We conclude the section with the covering lemma used in the
weak Harnack iteration.  It is a weighted version of the
kinetic ``ink spots in the wind'' principle, compare
\cite[Theorem~4.9.3 and Section~4.10]{Imbert2026}.
The lemma converts a local spreading property into a
comparison of weighted measures.

In the application, $E$ and $F$ are superlevel sets of the
same supersolution, with $E$ corresponding to the higher
level.  For example, $E=\{f>b\}$ and $F=\{f>a\}$, where
$0<a<b$, so that $E\subset F$.  For suitably separated
levels, the expansion of positivity estimates ensure that
if $f>b$ on more than half the weighted measure of a past
cylinder, then $f>a$ almost everywhere throughout its
forward stack.  Thus the whole stack is covered by the
lower-level set $F$, although the higher-level set $E$
need only occupy a substantial part of the past cylinder.

The cylinders satisfying this density condition are the
ink spots, and the displacement of their forward stacks
along the free transport explains the reference to wind.
The following lemma assumes this spreading property,
together with an upper bound on the radii of such cylinders.
It then compares the weighted measures of $E$ and $F$
inside a fixed reference cylinder.

\begin{lem}
\label{lem:weighted-kinetic-inkspots}
Let $\mathcal Q_{\rm ref}=Q_R(P_0)$ be a fixed reference intrinsic cylinder, and let \(E\subset F\) be
measurable sets.  Suppose that there exist an integer \(L\geq1\) and a
number \(r_0\in(0,1)\) satisfying
\begin{equation}
 Lr_0^2\leq1
 \label{eq:inkspot-small-leakage}
\end{equation}
such that the following implication holds: whenever $Q_r(P)\subset\mathcal Q_{\rm ref}$ and
\begin{equation}
 \muW(E\cap Q_r(P))
 >
 \frac12\muW(Q_r(P)),
 \label{eq:inkspot-density}
\end{equation}
then
\begin{equation}
 r<r_0R,
 \qquad
 \muW\bigl(\overline Q_r^{\,L}(P)\setminus F\bigr)=0.
 \label{eq:inkspot-stack-condition}
\end{equation}
Under these assumptions there exist $c\in(0,1)$, $C<\infty$, and $\sigma_0\in(0,2]$,
depending only on $m,k,M$ and the fixed reference geometry, such that
\begin{equation}
 \muW(E\cap\mathcal Q_{\rm ref})
 \leq
 \frac{L+1}{L}(1-c)
 \muW(F\cap\mathcal Q_{\rm ref})
 +
 CLr_0^{\sigma_0}\muW(\mathcal Q_{\rm ref}).
 \label{eq:weighted-inkspots}
\end{equation}
\end{lem}

\begin{rem}
The conclusion quantifies the enlargement from $E$ to $F$
inside the reference cylinder.  Since $c$ is independent
of $L$, the coefficient $(L+1)(1-c)/L$ is strictly less
than one for sufficiently large $L$.  In the application,
we first fix such an $L$ and verify the hypotheses for
suitably chosen superlevel sets.  Estimate
\eqref{eq:weighted-inkspots} then gives the contraction
needed in the weak Harnack iteration, up to the boundary
error. The restriction $r<r_0R$ ensures that every cylinder
meeting the density condition is small relative to the
reference cylinder.  Its forward stack need not be
contained in that cylinder, whereas the conclusion
measures only the portions of $E$ and $F$ inside it.
The error term accounts for this possible escape and
for boundary layers in the covering argument.  For fixed
$L$, the error tends to zero as $r_0\downarrow0$, provided
the hypotheses hold with the corresponding radius bound.
\end{rem}

\begin{proof}
Weighted doubling and the Vitali covering property give the
covering estimate, with boundary errors controlled by quantitative
$A_\infty$ estimates.  We then compare unions of past cylinders
with unions of their forward stacks by fixing $x=(v,z)$ and
using $\xi=y-tv$ and $t$.  This change of variables has Jacobian
one and preserves $\d\muW$, since $w$ is independent of $(y,t)$. Applying $T_{P_0,R}^{-1}$ reduces to $\mathcal Q_{\rm ref}=Q_1$.
The transformed weight still depends only on $x$ and has the
same $A_2$ characteristic.  Weighted measure ratios are preserved,
and the radius condition becomes $r<r_0$.

Define the family of high-density cylinders
\[
 \mathscr C
 :=
 \left\{
 Q_r(P)\subset Q_1:
 \muW(E\cap Q_r(P))
 >
 \frac12\muW(Q_r(P))
 \right\}
\]
and its union
\[
 \mathcal G
 :=
 \bigcup_{Q\in\mathscr C}Q.
\]
The root cylinder \(Q_1\) cannot belong to \(\mathscr C\), since otherwise
\eqref{eq:inkspot-stack-condition} would imply \(1<r_0\).  Consequently,
\begin{equation}
 \muW(E\cap Q_1)
 \leq
 \frac12\muW(Q_1).
 \label{eq:inkspot-root-density}
\end{equation}

We first estimate the weighted measure of $E\cap Q_1$
in terms of $\muW(\mathcal G)$, allowing an error from
a boundary layer of $Q_1$. The case
$r_0\geq1/8$ is immediate after increasing $C$, so assume $r_0<1/8$.
Put
\[
 \mathcal D_{r_0}
 :=\{(x,y,t)\in Q_1:
 |x|<1-2r_0,\ |y|<1-2r_0^2,
 -1+r_0^2<t<-r_0^2\}.
\]
Quantitative $A_\infty$ on the diffusive ball and the product structure
in $(y,t)$ give
\begin{equation}
 \muW(Q_1\setminus\mathcal D_{r_0})
 \leq C r_0^{\sigma_0}\muW(Q_1),
 \qquad \sigma_0:=\min\{\delta_1,2\}>0,
 \label{eq:reference-boundary-layer}
\end{equation}
where $\delta_1$ is an admissible exponent in
\eqref{eq:Ainftyquant}.

For a density point $P\in E\cap\mathcal D_{r_0}$, consider the centered
boxes
\[
 \mathcal B_r(P):=Q_r(\Phi_{r^2/2}(P)),\qquad0<r\leq r_0.
\]
They are nested, contain $P$ in their interiors, and lie in $Q_1$.
Indeed, their time coordinates differ from $t_P$ by at most $r^2/2$,
and their $y$-coordinates differ from $y_P$ by at most
$r^2|v_P|/2+r^3<2r_0^2$.
Their $E$-density tends to one as $r\downarrow0$, whereas at $r=r_0$
it is at most $1/2$ by \eqref{eq:inkspot-stack-condition}.
The density depends continuously on $r$, since cylinder boundaries are
null.  Choose a radius with density exactly $3/4$.

Apply Lemma \ref{lem:kinetic-covering-geometry} to these chosen cylinders.
It gives disjoint cylinders $\widetilde Q_i\in\mathscr C$ whose covering
enlargements cover $E\cap\mathcal D_{r_0}$ up to a null set, with
\[
 \muW(E\cap\mathcal D_{r_0})\leq C_{\rm cov}\sum_i\muW(\widetilde Q_i),
 \qquad
 \muW(\mathcal G\setminus E)\geq\frac14\sum_i\muW(\widetilde Q_i).
\]
Also $E\cap Q_1\subset\mathcal G$ up to a null set, by differentiation.
Combining these facts with \eqref{eq:reference-boundary-layer} and
rearranging gives
\begin{equation}
 \muW(E\cap Q_1)
 \leq(1-c)\muW(\mathcal G)
       +C r_0^{\sigma_0}\muW(Q_1),
 \label{eq:weighted-crawling-inkspots}
\end{equation}
where $c=(1+4C_{\rm cov})^{-1}$.  This stopping construction uses only
cylinders contained in the reference region, including near its boundary.

Since \(\mathcal G\) is open, there is a countable
subfamily
\[
 \{Q_i\}_{i\geq1}\subset\mathscr C,
 \qquad
 \mathcal G=\bigcup_{i\geq1}Q_i.
\]
Write
\[
 Q_i=Q_{r_i}(P_i),
 \qquad
 P_i=(v_i,z_i,y_i,t_i),
\]
and define
\[
 \overline{\mathcal G}^{\,L}
 :=
 \bigcup_{i\geq1}\overline Q_{r_i}^{\,L}(P_i).
\]
By \eqref{eq:inkspot-stack-condition}, every set
\(\overline Q_{r_i}^{\,L}(P_i)\setminus F\) has zero weighted measure.
Since the family is countable,
\begin{equation}
 \muW(\overline{\mathcal G}^{\,L}\setminus F)=0.
 \label{eq:stacked-union-contained-in-F}
\end{equation}

We next prove the stacked-union estimate
\begin{equation}
 \muW(\overline{\mathcal G}^{\,L})
 \geq
 \frac{L}{L+1}\muW(\mathcal G).
 \label{eq:weighted-stacked-union}
\end{equation}
It suffices first to consider the union of the first \(N\) cylinders and
then let \(N\to\infty\).

Fix the diffusive variables \(x=(v,z)\).  Only those indices \(i\) for
which
\[
 |x-(v_i,z_i)|<r_i
\]
contribute to the \(x\)-slice.  Introduce the free-transport coordinate $\xi:=y-tv$ and, for each contributing index, set $\xi_i(v):=y_i-t_iv$. If \((v,z,y,t)\in Q_i\), then
\begin{align*}
 |\xi-\xi_i(v)|
 &=
 \bigl|
 y-y_i-(t-t_i)v
 \bigr|\\
 &\leq
 \bigl|
 y-y_i-(t-t_i)v_i
 \bigr|
 +
 |t-t_i|\,|v-v_i|<
 r_i^3+r_i^2r_i
 =
 2r_i^3.
\end{align*}
Thus the \(x\)-slice of \(Q_i\), expressed in \((\xi,t)\), is contained
in
\[
 A_i^-
 :=
 (t_i-r_i^2,t_i]\times B_{2r_i^3}(\xi_i(v)).
\]
Conversely, if
\[
 t_i<t<t_i+Lr_i^2,
 \qquad
 |\xi-\xi_i(v)|<2r_i^3,
\]
then
\begin{align*}
 |y-y_i-(t-t_i)v_i|
 &\leq
 |\xi-\xi_i(v)|
 +
 |t-t_i|\,|v-v_i|<
 2r_i^3+Lr_i^3
 =
 (L+2)r_i^3.
\end{align*}
Hence the \(x\)-slice of
\(\overline Q_{r_i}^{\,L}(P_i)\) contains
\[
 A_i^+
 :=
 (t_i,t_i+Lr_i^2)
 \times B_{2r_i^3}(\xi_i(v)).
\]

We next use the elementary one-dimensional inequality
\begin{equation}
 \biggl|
 \bigcup_i(t_i,t_i+La_i)
 \biggr|
 \geq
 \frac{L}{L+1}
 \biggl|
 \bigcup_i(t_i-a_i,t_i]
 \biggr|,
 \qquad a_i>0.
 \label{eq:forward-backward-interval-union}
\end{equation}
To see this, decompose the union of the forward intervals into its
connected components.  If \(J=(a,b)\) is one such component, then every
corresponding interval satisfies
\[
 a_i\leq{(b-t_i)}/{L}\leq{(b-a)}/{L}.
\]
The union of the associated backward intervals is therefore contained in $(a-{(b-a)}/{L},b)$, whose length is \((L+1)|J|/L\).  Summing over the components proves
\eqref{eq:forward-backward-interval-union}.

For fixed \((v,z,\xi)\), apply
\eqref{eq:forward-backward-interval-union} to the indices satisfying
\[
 |x-(v_i,z_i)|<r_i,
 \qquad
 |\xi-\xi_i(v)|<2r_i^3,
\]
with \(a_i=r_i^2\).  Integrating first in \((\xi,t)\) and then in
\((v,z)\) against \(w(v,z)\,\d v\,\d z\) gives
\eqref{eq:weighted-stacked-union}.  Here the change of variables
\[
 y\longmapsto\xi=y-tv
\]
has Jacobian one, and the weighted measure is unchanged because the weight
is independent of \((y,t)\).  Passing from finite unions to the countable
union follows by monotone convergence.

Combining
\eqref{eq:weighted-crawling-inkspots} and
\eqref{eq:weighted-stacked-union}, we obtain
\begin{equation}
 \muW(E\cap Q_1)
 \leq
 \frac{L+1}{L}(1-c)
 \muW(\overline{\mathcal G}^{\,L})
 +C r_0^{\sigma_0}\muW(Q_1).
 \label{eq:inkspot-before-leakage}
\end{equation}

It remains to estimate the portion of the stacked union lying outside
\(Q_1\).  Every \(Q_i\) is contained in \(Q_1\) and has radius
\(r_i<r_0\).  Its forward stack has the same diffusive section, so no
leakage occurs through the diffusive boundary.  Since \(Q_i\subset Q_1\),
its center satisfies
\[
 -1+r_i^2\leq t_i\leq0,
 \qquad
 |v_i|\leq1,
 \qquad
 |y_i|\leq1.
\]
If \((v,z,y,t)\in\overline Q_{r_i}^{\,L}(P_i)\), then $-1<t<Lr_0^2$ and
\begin{align*}
 |y-y_i|
 &\leq
 |t-t_i|\,|v_i|+(L+2)r_i^3\leq
 Lr_0^2+(L+2)r_0^3
 \leq
 CLr_0^2.
\end{align*}
Thus any point of the stack lying outside \(Q_1\) belongs either to a
terminal time layer of thickness at most \(Lr_0^2\) or to a transported
\(y\)-boundary layer of thickness at most \(CLr_0^2\).  By
\eqref{eq:inkspot-small-leakage}, these are fixed small enlargements of
the normalized time interval and \(y\)-ball.  Since the weight depends
only on the diffusive variables,
\begin{equation}
 \muW(\overline{\mathcal G}^{\,L}\setminus Q_1)
 \leq
 CLr_0^2\muW(Q_1).
 \label{eq:inkspot-leakage}
\end{equation}

By \eqref{eq:stacked-union-contained-in-F},
\[
 \overline{\mathcal G}^{\,L}\cap Q_1
 \subset F\cap Q_1
\]
up to a null set.  Consequently,
\[
 \muW(\overline{\mathcal G}^{\,L})
 \leq
 \muW(F\cap Q_1)
 +
 CLr_0^2\muW(Q_1).
\]
Substituting this estimate into
\eqref{eq:inkspot-before-leakage} and absorbing the fixed factor
$\frac{L+1}{L}(1-c)\leq2$ into $C$, together with
$Lr_0^2+r_0^{\sigma_0}\leq2Lr_0^{\sigma_0}$, proves
\eqref{eq:weighted-inkspots}.  Scaling back completes the proof.
\end{proof}

\section{Weak Harnack and Harnack inequalities}
\label{Section_8}

We now combine the expansion of positivity with the weighted kinetic
ink-spots lemma.  The argument is a level-set iteration.  If a high
superlevel set occupies more than half of a sufficiently small cylinder
in the fixed past work region, Proposition
\ref{prop:stacked-expansion} propagates a lower level throughout its
forward stack.  Proposition \ref{prop:iterated-expansion} controls the
radius of such a cylinder in terms of its level, and Lemma
\ref{lem:weighted-kinetic-inkspots} converts these two facts into geometric
decay of the superlevel measures.

Recall that
\[
 \mathcal Q_{\rm past}=Q_{R_{\rm p}}(P_{\rm p})
\]
is the fixed past work cylinder introduced in
\eqref{eq:past-work-cylinder}.  In particular,
\[
 Q^-\Subset\mathcal Q_{\rm past}\Subset Q_{R_*},
\]
and the closure of \(\mathcal Q_{\rm past}\) lies strictly before \(Q^+\)
in time.

\begin{thm}
\label{thm:weakharnack}
Under \eqref{eq:A2} and \eqref{eq:ellipticity2}, there exist \(p_0>0\) and \(C<\infty\), depending only on
\(m,k,M,\lambda,\) and \(\Lambda\), such that every non-negative energy
supersolution \(u\) in \(Q_{R_*}\) satisfies
\begin{equation}
 \left(
 \avgint_{Q^-}u^{p_0}\,\d\muW
 \right)^{1/p_0}
 \leq
 C\essinf_{Q^+}u.
 \label{eq:weakharnack}
\end{equation}
The same conclusion holds for any fixed admissible ordered pair of product
boxes with the prescribed normalized geometric margins, provided that all
outer enlargements required in the proof remain in the solution domain.
In this case, $p_0$ and $C$ may depend additionally on the normalized
geometry, but can be chosen uniformly over compact families in the sense
of Remark \ref{rem:compact-configurations}.
\end{thm}

\begin{proof}
Fix \(\varepsilon>0\) and define
\[
 h
 :=
 \frac{u}{\essinf_{Q^+}u+\varepsilon}.
\]
Then \(h\) is a non-negative energy supersolution in \(Q_{R_*}\), and
\[
 \essinf_{Q^+}h
 =
 \frac{\essinf_{Q^+}u}
      {\essinf_{Q^+}u+\varepsilon}
 \leq1.
\]
It is enough to prove a bound
\begin{equation}
 \avgint_{\mathcal Q_{\rm past}}h^{p_0}\,\d\muW\leq C
 \label{eq:normalized-positive-moment}
\end{equation}
uniformly in $\varepsilon$.  This implies the required bound on $Q^-$
by weighted doubling and its fixed containment in $\mathcal Q_{\rm past}$.

Let \(c\in(0,1)\) be the constant in Lemma
\ref{lem:weighted-kinetic-inkspots}.  Choose an integer \(L\geq1\) so
large that
\begin{equation}
 \vartheta
 :=
 \frac{L+1}{L}(1-c)
 <1.
 \label{eq:inkspot-contraction-factor}
\end{equation}
Once \(L\) has been fixed, compact containment provides a geometric radius
\[
 r_{\rm geo}=r_{\rm geo}(L)\in(0,1)
\]
such that, whenever
\[
 Q_r(P)\subset\mathcal Q_{\rm past},
 \qquad
 r<r_{\rm geo},
\]
the forward stack \(\overline Q_r^{\,L}(P)\), the corresponding set
\(\mathcal U_r^L(P)\), and every outer enlargement needed in Proposition
\ref{prop:stacked-expansion} are contained in \(Q_{R_*}\).

Let \(\mathsf M_L>1\) be the constant in Proposition
\ref{prop:stacked-expansion}, and let \(C_{\rm rad}\) and \(\gamma_0\) be
the constants in the radius estimate
\eqref{eq:radius-from-future-infimum}.  Choose
\(C_*>C_{\rm rad}\).  We then choose \(K>1\) sufficiently large that
\begin{equation}
 K>\mathsf M_L,
 \qquad
 C_{\rm rad}K^{-1/\gamma_0}<r_{\rm geo},
 \qquad
 L\left(
 \frac{C_*}{R_{\rm p}}K^{-1/\gamma_0}
 \right)^2
 <1.
 \label{eq:choice-level-factor}
\end{equation}
All these choices depend only on the structural parameters and the fixed
normalized geometry.
In particular, $R_*$ was fixed by the geometry in Lemma
\ref{lem:positivity-tree-geometry} before choosing $L$.  The $L$-dependent
work regions are accommodated by reducing $r_{\rm geo}$, so $R_*$ need
not be enlarged.  The covering constant $c$ is independent of $L$.
After $L$ is fixed, $\mathsf M_L$, $r_{\rm geo}$, $C_{\rm rad}$, and
$\gamma_0$ are fixed before $K$ is chosen.  Thus the choices are
non-circular and uniform over $[w]_{A_2}\leq M$.

For \(j\geq0\), define
\[
 E_j
 :=
 \{h>K^j\}\cap\mathcal Q_{\rm past},
 \qquad
 F_j
 :=
 \{h>K^j\}\cap Q_{R_*},
\]
and set
\[
 a_j
 :=
 {\muW(E_j)}/
      {\muW(\mathcal Q_{\rm past})}.
\]
Clearly, $E_{j+1}\subset F_j$.

Fix \(j\geq0\), and suppose that
\[
 Q_r(P)\subset\mathcal Q_{\rm past}
\]
satisfies
\begin{equation}
 \muW(E_{j+1}\cap Q_r(P))
 >
 \frac12\muW(Q_r(P)).
 \label{eq:high-level-density-cylinder}
\end{equation}
Since \(E_{j+1}\subset\{h>K^{j+1}\}\), Proposition
\ref{prop:iterated-expansion}, applied with \(A=K^{j+1}\), gives
\[
 \essinf_{Q^+}h
 \geq
 c_0K^{j+1}r^{\gamma_0}.
\]
Together with \(\essinf_{Q^+}h\leq1\), this yields
\begin{equation}
 r
 \leq
 C_{\rm rad}K^{-(j+1)/\gamma_0}
 <
 r_j,
 \qquad
 r_j:=C_*K^{-(j+1)/\gamma_0}.
 \label{eq:level-dependent-radius}
\end{equation}
In particular, by \eqref{eq:choice-level-factor},
\[
 r<r_{\rm geo}.
\]
Hence all the geometric sets required to apply Proposition
\ref{prop:stacked-expansion} are contained in \(Q_{R_*}\).

Apply that proposition to \(h\) with
\[
 A:={K^{j+1}}/{\mathsf M_L}.
\]
The density assumption \eqref{eq:high-level-density-cylinder} is precisely
\[
 \muW(\{h>\mathsf M_LA\}\cap Q_r(P))
 >
 \frac12\muW(Q_r(P)).
\]
We therefore obtain
\[
 h
 \geq
 {K^{j+1}}/{\mathsf M_L}
 >
 K^j
\]
almost everywhere on \(\overline Q_r^{\,L}(P)\), where the last inequality
uses \(K>\mathsf M_L\).  Consequently,
\begin{equation}
 \muW\bigl(
 \overline Q_r^{\,L}(P)\setminus F_j
 \bigr)=0.
 \label{eq:level-stack-contained}
\end{equation}

We now apply Lemma \ref{lem:weighted-kinetic-inkspots} with
\[
 E=E_{j+1},
 \qquad
 F=F_j,
 \qquad
 \mathcal Q_{\rm ref}=\mathcal Q_{\rm past},
\]
and with the relative radius cutoff
\[
 r_{0,j}
 :=
 \frac{r_j}{R_{\rm p}}
 =
 \frac{C_*}{R_{\rm p}}
 K^{-(j+1)/\gamma_0}.
\]
The choice \eqref{eq:choice-level-factor} ensures that
\[
 Lr_{0,j}^2\leq Lr_{0,0}^2<1.
\]
The radius condition follows from
\eqref{eq:level-dependent-radius}, and the stack condition follows from
\eqref{eq:level-stack-contained}.  The ink-spots lemma gives
\begin{align*}
 \muW(E_{j+1})
 &\leq
 \vartheta\,
 \muW(F_j\cap\mathcal Q_{\rm past})
 +
 CLr_{0,j}^{\sigma_0}
 \muW(\mathcal Q_{\rm past})=
 \vartheta\,\muW(E_j)
 +
 CLr_{0,j}^{\sigma_0}
 \muW(\mathcal Q_{\rm past}).
\end{align*}
After division by \(\muW(\mathcal Q_{\rm past})\), and absorption of the
fixed factors \(L\), \(C_*\), and \(R_{\rm p}\) into \(C\), we obtain
\begin{equation}
 a_{j+1}
 \leq
 \vartheta a_j
 +
 C K^{-\sigma_0(j+1)/\gamma_0}.
 \label{eq:level-set-recurrence}
\end{equation}

Set $q:=K^{-\sigma_0/\gamma_0}\in(0,1)$, and choose $\rho\in\bigl(\max\{\vartheta,q\},1\bigr)$. Since \(a_0\leq1\), iteration of
\eqref{eq:level-set-recurrence} gives
\begin{align*}
 a_j
 &\leq
 \vartheta^ja_0
 +
 C\sum_{i=1}^j
 \vartheta^{j-i}q^i\leq
 C\rho^j.
\end{align*}
Thus
\begin{equation}
 a_j\leq C\rho^j
 \qquad\hbox{for every }j\geq0.
 \label{eq:level-set-geometric-decay}
\end{equation}

Choose
\[
 p_0:=\frac{-\log\rho}{2\log K}>0,
 \qquad K^{p_0}\rho=\sqrt\rho<1.
\]
Decomposing $\mathcal Q_{\rm past}$ into $\{h\leq1\}$ and
$\{K^j<h\leq K^{j+1}\}$, $j\geq0$, and using
\eqref{eq:level-set-geometric-decay}, we obtain
\begin{align*}
 \avgint_{\mathcal Q_{\rm past}}h^{p_0}\,\d\muW
 &\leq
 1+
 \sum_{j=0}^\infty
 K^{p_0(j+1)}a_j\leq
 1+
 C K^{p_0}
 \sum_{j=0}^\infty
 (K^{p_0}\rho)^j
 \leq C.
\end{align*}
This proves \eqref{eq:normalized-positive-moment}.  Returning to $u$ and
letting $\varepsilon\downarrow0$ gives the stronger estimate
\begin{equation}
 \left(
 \avgint_{\mathcal Q_{\rm past}}u^{p_0}\,\d\muW
 \right)^{1/p_0}
 \leq
 C\essinf_{Q^+}u.
 \label{eq:work-region-weak-harnack}
\end{equation}
The fixed measure ratio $\muW(\mathcal Q_{\rm past})/\muW(Q^-)\leq C$
now proves \eqref{eq:weakharnack}.

For any other fixed admissible ordered pair, choose the auxiliary past
work region and the chaining configurations relative to that pair, with
the prescribed positive margins.  The preceding argument then applies
with these geometric choices.  Its constants may depend additionally on
the normalized geometry and are uniform over compact families in the
sense of Remark \ref{rem:compact-configurations}.
\end{proof}

We now combine the weak Harnack inequality with the local boundedness
estimate proved under the general \(A_2\) hypothesis.

\begin{proof}[Proof of Theorem \ref{thm:main}]
By intrinsic translation and dilation, it is enough to prove the estimate
at unit scale.  Take
\[
 \widehat Q^-:=B_{3/8}^v\times B_{3/8}^z\times B_{1/32}^y
                    \times(-13/16,-7/16).
\]
Then $Q^-\Subset\widehat Q^-\Subset\mathcal Q_{\rm past}$, and
$(\widehat Q^-,Q^+)$ is admissible with fixed margins.  The $z$-factor
is omitted when $k=0$.

A finite-cover consequence of Proposition
\ref{prop:localboundedness}, applied with exponent \(p_0\) to the fixed
containment \(Q^-\Subset\widehat Q^-\), gives
\begin{equation}
 \esssup_{Q^-}u
 \leq
 C
 \left(
 \avgint_{\widehat Q^-}u^{p_0}\,\d\muW
 \right)^{1/p_0}.
 \label{eq:harnack-local-boundedness-step}
\end{equation}
Indeed, one covers \(Q^-\) by finitely many smaller intrinsic cylinders
whose fixed outer enlargements are compactly contained in
\(\widehat Q^-\), applies Proposition
\ref{prop:localboundedness} on each enlargement, and uses weighted
doubling to compare the resulting averages with the average over
\(\widehat Q^-\).

Estimate \eqref{eq:work-region-weak-harnack} and the fixed measure ratio
$\muW(\mathcal Q_{\rm past})/\muW(\widehat Q^-)\leq C$ yield
\begin{equation}
 \left(
 \avgint_{\widehat Q^-}u^{p_0}\,\d\muW
 \right)^{1/p_0}
 \leq
 C\essinf_{Q^+}u.
 \label{eq:harnack-weak-step}
\end{equation}
Combining \eqref{eq:harnack-local-boundedness-step} and
\eqref{eq:harnack-weak-step} gives
\[
 \esssup_{Q^-}u
 \leq
 C\essinf_{Q^+}u,
\]
which is \eqref{eq:harnack} at unit scale.

For general \(P_0\) and \(r>0\), apply the unit-scale estimate to
\(u\circ T_{P_0,r}\).  Intrinsic translation and dilation preserve the
ellipticity constants, independence of the weight from $(y,t)$,
its $A_2$ characteristic, and the normalized past-future geometry.  This proves
Theorem \ref{thm:main} with the same structural dependence of the
constant.
\end{proof}

\begin{rem}
The constants are uniform over the class $[w]_{A_2}\leq M$.
The full-variable compactness theorem permits the weight to vary in the
intermediate-value contradiction, and the weighted mean velocity argument
uses only positivity of the limiting weight.  The later geometric
arguments require weighted doubling and invariance under the flow of
$Y$, both of which hold for every $w=w(x)\in A_2(\R^n)$.
\end{rem}

\section{H\"older continuity and the strong minimum principle:
proof of the corollaries}
\label{Section_9}

We first prove the oscillation decay under the standing assumptions
\eqref{eq:A2} and \eqref{eq:ellipticity2}.  All oscillations in the argument are understood as
essential oscillations until the continuous representative has been
constructed.

Choose, once and for all, an integer \(L\geq1\), numbers
\(\eta,\vartheta\in(0,1)\), and a point \(P_{\rm b}\) with negative time
coordinate such that
\begin{equation}
 \overline{\mathcal W_\vartheta(0)}
 \Subset
 \overline Q_\eta^{\,L}(P_{\rm b}),
 \qquad
 \mathcal U_\eta^L(P_{\rm b})\Subset\mathcal W.
 \label{eq:holder-buffered-geometry}
\end{equation}
For example, take $L=1$ and
$P_{\rm b}=(0,0,0,-\eta^2/2)$.  The forward stack then has time
interval $(-\eta^2/2,\eta^2/2)$, diffusive radius $\eta$, and
transported radius $3\eta^3$.  Thus any $0<\vartheta<\eta/2$
gives the first containment.  Since $\mathcal U^1$ is fixed and
bounded, taking $\eta$ sufficiently small also gives the second.

Let \(u\) be an energy solution of \eqref{eq:original} in an open set
\(\Omega\), and suppose that
\[
 \overline{\mathcal W_r(P_0)}\Subset\Omega.
\]
A finite covering of this compact set and Proposition
\ref{prop:localboundedness}, applied to \(u_+\) and \((-u)_+\), show that
\(u\) is essentially bounded in \(\mathcal W_r(P_0)\).  Define
\[
 u^*:=\esssup_{\mathcal W_r(P_0)}u,
 \qquad
 u_*:=\essinf_{\mathcal W_r(P_0)}u,
 \qquad
 \omega:=u^*-u_*.
\]
If \(\omega=0\), there is nothing to prove.  Otherwise, both \(u-u_*\) and
\(u^*-u\) are non-negative energy solutions in
\(\mathcal W_r(P_0)\).

Let
\[
 G_r(P_0):=T_{P_0,r}\bigl(Q_\eta(P_{\rm b})\bigr).
\]
At least one of the sets
\[
 \{u-u_*\geq\omega/2\}\cap G_r(P_0),
 \qquad
 \{u^*-u\geq\omega/2\}\cap G_r(P_0)
\]
has measure at least \(\muW(G_r(P_0))/2\).  Denote the corresponding
non-negative solution by \(h\).  Since
\[
 \muW(\{h>\omega/4\}\cap G_r(P_0))
 \geq\frac12\muW(G_r(P_0)),
\]
Proposition \ref{prop:stacked-expansion}, applied with
\[
 A:={\omega}/{(4\mathsf M_L)},
\]
gives \(h\geq A\) almost everywhere on the forward stack of
\(G_r(P_0)\).  Here the work-set condition follows from
\eqref{eq:holder-buffered-geometry} and intrinsic covariance.  The same
geometry shows that this stack contains
\(\mathcal W_{\vartheta r}(P_0)\).  Consequently,
\begin{equation}
 \operatorname*{ess\,osc}_{\mathcal W_{\vartheta r}(P_0)}u
 \leq
 q\operatorname*{ess\,osc}_{\mathcal W_r(P_0)}u,
 \qquad
 q:=1-\frac{1}{4\mathsf M_L}\in(0,1).
 \label{eq:oscillation}
\end{equation}

Iteration gives
\[
 \operatorname*{ess\,osc}_{\mathcal W_{\vartheta^jr}(P_0)}u
 \leq
 q^j\operatorname*{ess\,osc}_{\mathcal W_r(P_0)}u.
\]
Set
\[
 \alpha_0:=\frac{\log q}{\log\vartheta}>0,
 \qquad
 \alpha:=\min\{\alpha_0,1/2\}\in(0,1).
\]
For \(0<\rho\leq r\), choose \(j\) such that
\(\vartheta^{j+1}r<\rho\leq\vartheta^jr\).  Monotonicity of the oscillation
and \(\alpha\leq\alpha_0\) then give
\[
 \operatorname*{ess\,osc}_{\mathcal W_\rho(P_0)}u
 \leq
 C\left(\frac{\rho}{r}\right)^\alpha
 \operatorname*{ess\,osc}_{\mathcal W_r(P_0)}u.
\]
This is \eqref{eq:holder-oscillation}.

One convenient symmetric intrinsic quasi-distance is
\begin{align}
 d_K(P,\widetilde P)
 :=\max\bigl\{&
 |x-\widetilde x|,
 |t-\widetilde t|^{1/2},|y-\widetilde y-(t-\widetilde t)\widetilde v|^{1/3},
 |y-\widetilde y-(t-\widetilde t)v|^{1/3}
 \bigr\}.
 \label{eq:intrinsic-quasidistance}
\end{align}
Its balls are comparable, up to fixed structural enlargements, with the
sets \(\mathcal W_r(P)\).  Applying \eqref{eq:holder-oscillation} after
intrinsic translations therefore
gives, on every compact subset of \(\Omega\),
\[
 |u(P)-u(\widetilde P)|
 \leq C_K (d_K(P,\widetilde P))^\alpha
\]
for almost every \(P,\widetilde P\).  The standard modification on a set of
measure zero produces a locally H\"older continuous representative of
\(u\).  This proves Corollary \ref{cor:holder}.

\subsection{Proof of the strong minimum principle}

We prove Corollary \ref{cor:minimum}.  Let \(u\geq0\) and suppose that its
locally H\"older continuous representative satisfies \(u(P_0)=0\).  If
\(Q^+\) is the future member of a sufficiently small intrinsic
translate and dilate of the fixed Harnack configuration, chosen so that
it contains \(P_0\) in its interior and its outer cylinder is compactly
contained in \(\Omega\), then continuity and the positivity of
\(\d\muW\) on open sets give
\[
 \essinf_{Q^+}u=0.
\]
Theorem \ref{thm:main} then implies
\[
 \esssup_{Q^-}u=0
\]
in the corresponding earlier box.  Hence \(u=0\) there almost everywhere,
and therefore everywhere there by continuity.

Let $\gamma:[0,T]\to\Omega$ be any backward admissible curve from $P_0$.
Its image is compactly contained in $\Omega$.  Keep its diffusive controls
fixed and replace $\alpha$ by $\alpha_\varepsilon=\alpha+\varepsilon$.
The resulting curves converge uniformly to $\gamma$ and remain in
$\Omega$ for small $\varepsilon>0$.

For fixed $\varepsilon$, time is a valid parameter, and the derivatives
of $v$ and $z$ with respect to time are bounded.  Over a time interval
of length $h$, the diffusive displacement is therefore $O(h)$, and the
transported displacement from the free characteristic is $O(h^2)$.
For sufficiently small $h$, these are smaller than the spatial margins
of a Harnack configuration of scale $r\simeq\sqrt h$, whose diffusive
and transported radii are of orders $h^{1/2}$ and $h^{3/2}$.
The constants are uniform along this fixed perturbed curve.  Its positive
distance from the complement of $\Omega$ also keeps all outer cylinders
inside $\Omega$.  A finite time partition consequently gives an ordered
Harnack chain from $P_0$ to the terminal point of the perturbed curve.
Repeated application of the preceding zero-propagation implication gives
$u=0$ at that terminal point.

Letting $\varepsilon\downarrow0$ and using continuity gives $u=0$ at
the endpoint of $\gamma$.  A further passage to limits proves vanishing
on the relative closure of the backward attainable set, as asserted in
Corollary \ref{cor:minimum}.

\subsection{Fractional powers of the Kolmogorov operator}
\label{subsec:fractional-extension-proof}

\begin{proof}[Proof of Corollary \ref{cor:fractional-harnack}]
The representation \eqref{eq:GT-extension} gives $U\geq0$ because
$f\geq0$ globally.  Set $U^{\rm e}(v,\lambda,y,t)=U(v,|\lambda|,y,t)$.
For $\varphi\in C_0^\infty(\mathcal O)$, use
\[
 \psi(v,\lambda,y,t)
 =\varphi(v,\lambda,y,t)+\varphi(v,-\lambda,y,t),\qquad\lambda\geq0,
\]
in \eqref{eq:extension-weak-conormal}.  Its right-hand side is zero.
Changing $\lambda$ to $-\lambda$ in the second summand gives exactly the
weak formulation of \eqref{eq:GT-weighted-form} for $U^{\rm e}$.
The one-sided energy doubles under reflection, so $U^{\rm e}$ belongs
to the required local energy space.  Indeed, the weighted one-dimensional
trace across $\lambda=0$ is the same from both sides, since $-1<a<1$.
There is no jump term in the weak derivative.  Thus, writing $U$ again for the
reflected function, $U$ is an energy solution of
\eqref{eq:original} with
\[
 x=(v,\lambda),
 \qquad
 z=\lambda,
 \qquad
 w(v,\lambda)=w_0(\lambda)=|\lambda|^{1-2s},
 \qquad
 A=w_0I_{m+1}.
\]
Since $-1<1-2s<1$, the weight $w_0$ belongs to $A_2(\R)$,
and its extension $w(v,\lambda)$ belongs to $A_2(\R^{m+1})$ with
characteristic controlled by $m$ and $s$.  Theorem
\ref{thm:main} and Corollary \ref{cor:holder} therefore apply to the even
extension.  Using its locally H\"older continuous representative, we obtain
\begin{align*}
 \sup_{Q_r^-(\widetilde P_0)\cap\{\lambda=0\}}U
 &\leq
 \sup_{Q_r^-(\widetilde P_0)}U\leq
 C\inf_{Q_r^+(\widetilde P_0)}U\leq
 C\inf_{Q_r^+(\widetilde P_0)\cap\{\lambda=0\}}U.
\end{align*}
Since \(U(v,0,y,t)=f(v,y,t)\), this proves
\eqref{eq:fractional-harnack}.  The H\"older continuity of \(f\) follows by
restricting the H\"older continuous representative of \(U\) to
\(\{\lambda=0\}\).
\end{proof}

\section{Concluding remarks and further research}
\label{Section_10}

The interior regularity theory developed here applies to general
weights $w=w(x)\in A_2(\R^n)$ under the weighted ellipticity
assumptions on $A=wB$. It includes local boundedness, weak Harnack
and Harnack inequalities, H\"older continuity, and the strong
minimum principle. We retain the equation in the form
\[
 \operatorname{div}_x(wB\nabla_xu)-Y(wu)=0,
\]
with the measure $\d\muW=w(x)\d x\d y\d t$. The argument uses
$Yw=0$ without differentiating the weight in the diffusive
variables or requiring pointwise comparisons of translated weights.

For compactness, we combine $\mathrm L^p$ averaging of the weighted
densities with weighted Poincar\'e estimates in all diffusive
variables. At a binary limit, the diffusive energy removes
dependence on $x$, and weighted velocity averages recover the
admissible transport directions. Separating the absolutely
continuous and singular parts of the tested limiting equation,
and using the sign of the non-negative measure, shows that the
binary limit is non-increasing along these directions.
The theory therefore allows degeneracy in the active variables.
It also applies to the fractional-extension weight
$|\lambda|^{1-2s}$, $0<s<1$. The application to fractional powers
uses the energy and weak conormal assumptions stated in
Corollary \ref{cor:fractional-harnack}.

A remaining question is sharp regularity transfer. The kinetic
Sobolev inequality gives the strict gain needed for Moser
iteration, but its exponents are not claimed to be optimal.
Critical kinetic trajectories yield sharp transfer and Sobolev
estimates in the unweighted theory, see
\cite{DietertMouhotNiebelZacher2025,Niebel2026}. A construction
compatible with general $A_2$ weights could clarify the loss
in the resonant-slab argument and help determine optimal
weighted exponents.

Further extensions concern lower-order terms and nonhomogeneous
data. The compactness theorem already permits an
$\mathrm L^2(\d\muW)$ source and signed measure data with uniformly
bounded total variation. Extending the Harnack estimates to
diffusive drifts, potentials, and forcing terms would require
suitable scale-invariant assumptions. Critical cases may also
require smallness. Weights depending on $(y,t)$ present a
different difficulty. In this setting, $Yw$ need not vanish,
formal skew-adjointness with respect to the weighted measure can fail,
and the Fourier localization
argument used here does not directly extend. Identifying
conditions on such weights that preserve the regularity theory
remains an open direction.

\medskip
\noindent
{\bf Declaration of use of AI.}
The author used GPT-5.6 Sol Ultra, developed by OpenAI, as an interactive tool
during the preparation and revision of this manuscript.  The tool was
used for language editing, organization, LaTeX formatting, and checks of
clarity and internal consistency.  The mathematical arguments, results,
and conclusions were developed and verified by the author, who also
checked the references and takes full responsibility for the content of
the manuscript.

\appendix

\section{Kinetic geometry and positivity propagation}
\label{app:positivity-figures}

We prove the geometric lemmas stated in Subsection
\ref{subsec:auxiliary-kinetic-geometry} and explain the accompanying
figures. Throughout, the notation is that of Section \ref{Section_7}.

\begin{proof}[Proof of Lemma~\ref{lem:kinetic-covering-geometry}]
The relevant engulfing statement is
\[
 Q_s(P')\cap Q_r(P)\ne\varnothing,\quad s\leq2r
 \quad\Longrightarrow\quad Q_s(P')\subset Q_r^*(P).
\]
Normalize $P=0$ and $r=1$.  An intersection point gives
$|x_{P'}|<1+s$, $-1<t_{P'}<s^2$, and
$|y_{P'}|<1+s^3+s^2(1+s)$.  Hence every point of $Q_s(P')$ satisfies
$|x|<5$, $-5<t<4$, and $|y|<41$.  These bounds place it in
$Q_5(\Phi_{12}(0))$.  Greedy selection in decreasing radius classes
now gives the covering assertion. The factor two allows radii that do not
attain a maximum.  The forward time shift in
\eqref{eq:covering-dilation} is essential for covering intersecting
cylinders with later terminal times.

The volume formula \eqref{eq:cylvolume} and doubling of $w$ prove
\eqref{eq:intrinsic-measure-doubling}. Shifting the center along $\Phi_s$
does not change the measure.  Centered intrinsic boxes are comparable to
balls of the homogeneous kinetic quasi-metric.  Their maximal operator
has weak type $(1,1)$ by the same selection argument and doubling, so the
usual approximation by continuous functions proves differentiation.
Backward cylinders are contained in comparable centered boxes, which
gives the stated differentiation property for them as well.  The proof
for comparable rectangles is identical.
\end{proof}

\begin{proof}[Proof of Lemma~\ref{lem:positivity-tree-geometry}]
We first construct the partition and root boxes used in both parts
of the lemma.  Intersect $Q_1^{\rm ear}$ with a fixed sufficiently
fine rectangular grid, discard cells of zero Lebesgue measure,
and denote the remaining cells by $E_\alpha$.  These cells form
a finite measurable partition up to null sets.  Choose product
boxes $S_\alpha^-$ containing them, with side lengths comparable
to the grid scale.  Since $Q_1^{\rm ear}$ lies below $t=-1/4$,
the grid and the surrounding boxes can be chosen so that every
$S_\alpha^-$ has terminal time at most $-\tau_*<0$, with a
fixed $\tau_*>0$.

We verify the weighted comparison between a cell and its root
box.  Enclosing the diffusive factor of $S_\alpha^-$ in a
comparable Euclidean ball, and using the product structure in
$(y,t)$, gives
\[
 \muW(S_\alpha^-)
 \int_{S_\alpha^-}w^{-1}\d x\d y\d t
 \leq C|S_\alpha^-|^2.
\]
Here $C$ depends only on the fixed side-length ratios, $n$,
and $M$.  Cauchy--Schwarz on $E_\alpha$, followed by
$E_\alpha\subset S_\alpha^-$, therefore yields
\[
 |E_\alpha|^2
 \leq
 \muW(E_\alpha)\int_{S_\alpha^-}w^{-1}\d x\d y\d t
 \leq
 C\,\frac{\muW(E_\alpha)}{\muW(S_\alpha^-)}
 |S_\alpha^-|^2.
\]
Each cell has positive Lebesgue measure, and there are only
finitely many cells.  Thus the minimum of
$|E_\alpha|/|S_\alpha^-|$ is positive.  This proves
\eqref{eq:positivity-root-comparison}, uniformly over the
stated $A_2$ class.

For (i)--(iv), a tree with a single vertex suffices for each
$\alpha$.  Choose a future product box $F^L$ containing the
closure of the forward stack $\overline Q_1^{\,L}(0)$ and
whose initial time is greater than $-\tau_*$.  For example,
we may take
\[
 F^L=
 B_2^v\times B_2^z\times B_{L+3}^y
 \times(-\tau_*/2,L+1).
\]
The time gap between $S_\alpha^-$ and $F^L$ is at least
$\tau_*/2$, and their transported coordinates lie in bounded
sets.  Hence all velocities joining their $(y,t)$ factors
are bounded in terms of $L$ and the fixed partition.
Choose an active work ball containing these velocities and
the active factors of both boxes, with a positive margin.
A sufficiently large product work region also contains all
joining characteristic segments with a positive margin.
Thus $(S_\alpha^-,F^L)$ is an admissible ordered pair.

For each $\alpha$, take this pair as the unique vertex of
$\mathscr T_\alpha$.  It is both the root and a terminal
vertex, and its future member contains the entire forward
stack.  Since the root family is finite, all boxes and
required outer enlargements lie in a bounded open set
$\mathcal U^L$, which we enlarge to contain the closures of $Q_1$
and the forward stack.  This proves (i)--(iv).  The construction is
illustrated in Figure~\ref{fig:positivity-tree-geometry}(a):
the selected cell
lies in its root box, and one application propagates
positivity to the common future box $F^L$.

We now turn to (v).  The enlargement step is based on the
fixed pair
\[
 \mathcal R:=
 B_1^v\times B_1^z\times B_1^y\times(-1,-1/2),
 \qquad
 \mathcal F:=
 B_3^v\times B_3^z\times B_9^y\times(11,15).
\]
For endpoints in their $(y,t)$ factors, the transported
displacement has norm less than $10$, while the time
difference is greater than $23/2$.  Every connecting
velocity therefore has norm less than $20/23<1$.
It follows that $(\mathcal R,\mathcal F)$ is admissible
in a fixed enlarged product region.  Moreover,
\begin{equation}
 T_{(0,0,0,16),2}(\mathcal R)\Subset\mathcal F.
 \label{eq:geometric-doubling-template}
\end{equation}
Indeed, the transformed past box has diffusive radii two,
transported radius eight, and time interval $(12,14)$,
all strictly inside the corresponding factors of
$\mathcal F$.

Let $d>0$ be the time gap between
$\overline{\mathcal Q_{\rm past}}$ and $\overline{Q^+}$.
Choose a structural scale $r_*\in(0,1)$ sufficiently small.
The estimates below will ensure that the initial and
enlargement steps, including their temporal margins,
advance by less than $d/2$.

Suppose first that $r<r_*$.  Set
\[
 \rho_j=2^jr,\qquad
 P_0'=\Phi_{4r^2}(P),\qquad
 P_{j+1}'=\Phi_{16\rho_j^2}(P_j'),\qquad
 N=\left\lceil\log_2\frac{r_*}{r}\right\rceil,
\]
and write
\[
 \mathcal R_j:=T_{P_j',\rho_j}\mathcal R,
 \qquad
 \mathcal F_j:=T_{P_j',\rho_j}\mathcal F.
\]
The translated and dilated cells $T_{P,r}(E_\alpha)$
partition $T_{P,r}(Q_1^{\rm ear})$, and their root boxes
are $T_{P,r}(S_\alpha^-)$.

The first step connects each root box to a future box
containing $\overline{\mathcal R_0}$ in its interior.
To see that this can be done uniformly, normalize by
$T_{P,r}$.  In these coordinates, the root boxes belong
to the fixed finite family, whereas $\mathcal R_0$
has diffusive and transported radii one and time
interval $(3,7/2)$.  A fixed slightly larger future
product box therefore gives admissible pairs by the
same bounded-velocity argument used above.  This
requires one measure-to-point application from each root.

Next use the ordered pairs
$(\mathcal R_j,\mathcal F_j)$ for $j=0,\ldots,N-1$.
The choice of centers and scales gives
\[
 \mathcal R_{j+1}
 =
 T_{P_j',\rho_j}
 \bigl(T_{(0,0,0,16),2}\mathcal R\bigr)
 \Subset
 T_{P_j',\rho_j}\mathcal F
 =
 \mathcal F_j,
\]
by \eqref{eq:geometric-doubling-template}.
Thus the past member of each new pair lies compactly
inside the future member of the preceding pair.
The containment margins are fixed after normalization.

The terminal scale satisfies $r_*\leq\rho_N<2r_*$.
Furthermore, the total advance of the centers from $P$
to $P_N'$ is
\[
 4r^2+16\sum_{j=0}^{N-1}\rho_j^2
 =
 4r^2+\frac{16}{3}(\rho_N^2-r^2)
 \leq C r_*^2.
\]
The fixed outer enlargements of these pairs have
temporal margins bounded by another structural
multiple of $r_*^2$.  Taking $r_*$ sufficiently small
therefore keeps the terminal past box and its required
neighbourhood strictly before $Q^+$, with a time gap
bounded below.  All centers follow the free
characteristic through $P$: their diffusive coordinates
remain fixed, and their transported coordinates change
by $v_P$ times the elapsed time.  Since $P$ ranges over
a bounded set, all intermediate boxes and their outer
enlargements remain in a fixed bounded region.

For the final step, pair $\mathcal R_N$ with a slightly
larger future box containing $\overline{Q^+}$.
This future box and its outer time margin can be chosen
strictly below $t=0$.  The connecting velocities are
uniformly bounded because the time gap is bounded below
and the spatial coordinates remain bounded.
Moreover, the terminal centers range over a bounded
set and $\rho_N\in[r_*,2r_*]$.  With the fixed positive
margins just chosen, the normalized terminal
configurations therefore belong to a compact family
to which the uniform version of Proposition
\ref{prop:weighted-intermediate-value} applies.

This construction gives one initial application,
$N$ enlargement applications, and one final application.
Consequently, each branch has $N+2$ vertices and
\[
 N(r)\leq N+2\leq C_1+C_2\log^+\frac1r.
\]
Figure~\ref{fig:positivity-tree-geometry}(b) shows one
such branch, including the containments
$\mathcal R_{j+1}\Subset\mathcal F_j$ and the terminal
future member containing $Q^+$.

If $r\geq r_*$, pair each root box directly with a
future box containing $\overline{Q^+}$.
These configurations also form a compact family with
uniform positive margins, since
$r_*\leq r\leq R_{\rm p}$ and the original
past--future time gap is fixed.

Finally, choose $R_*$ large enough to contain all work
regions required in (v).  Their terminal times remain
strictly below zero, and this choice is independent
of $P$ and $r$.  The root comparison is preserved by
intrinsic normalization, since translations and
dilations preserve the $A_2$ characteristic.
For (i)--(iv), the normalized configurations form a
finite family.  For (v), they consist of a finite
initial family, the single enlargement pair
$(\mathcal R,\mathcal F)$, and a compact terminal
family.  This gives the asserted uniform geometry
and completes the proof.
\end{proof}

Figure~\ref{fig:rooted-positivity-tree} illustrates a rooted tree
of admissible ordered pairs and the geometry of a parent--child
pair. Each vertex represents one measure-to-point application,
whereas an edge records the containment needed for the next
application. The highlighted branch has three vertices and
therefore length three, although it contains only two edges.

The diagrams in Figure~\ref{fig:positivity-tree-geometry}
show the time order and containments schematically.
In panel (b), the common vertical alignment corresponds
to the free-transport coordinate
$y-y_P-(t-t_P)v_P$.  The diffusive coordinates are
suppressed, and the widths illustrate enlargement
without representing the quantitative intrinsic
radii $\rho_j$ and $\rho_j^3$.

\begin{figure}[!ht]
\centering
\begin{tikzpicture}[x=1cm,y=1cm,font=\small]
\path[use as bounding box] (0,0) rectangle (16.4,8.8);
\node[kp/heading] at (.1,8.5) {(a) A rooted directed tree};
\node[kp/heading] at (8.85,8.5) {(b) Geometry of one edge};
\draw[black!15] (8.45,.35)--(8.45,8.1);

\node[kp/vertex] (root) at (4.1,7.0)
 {$\beta_0$ (root)\\[2pt]$(Q_{\beta_0}^-,Q_{\beta_0}^+)$};
\node[kp/vertex] (left) at (2.6,4.7)
 {$\beta_1$\\[2pt]$(Q_{\beta_1}^-,Q_{\beta_1}^+)$};
\node[kp/vertex] (right) at (6.8,4.7)
 {$\beta_2$\\[2pt]$(Q_{\beta_2}^-,Q_{\beta_2}^+)$};
\node[kp/terminal] (ll) at (1.15,2.35)
 {$\beta_{11}$\\[2pt]$(Q_{\beta_{11}}^-,Q_{\beta_{11}}^+)$};
\node[kp/terminal] (lr) at (4.05,2.35)
 {$\beta_{12}$\\[2pt]$(Q_{\beta_{12}}^-,Q_{\beta_{12}}^+)$};
\node[kp/terminal] (rr) at (6.8,2.35)
 {$\beta_{21}$\\[2pt]$(Q_{\beta_{21}}^-,Q_{\beta_{21}}^+)$};
\draw[kp/step] (root.south west)--(left.north);
\draw[kp/arrow] (root.south east)--(right.north);
\draw[kp/arrow] (left.south west)--(ll.north);
\draw[kp/step] (left.south east)--(lr.north);
\draw[kp/arrow] (right.south)--(rr.north);
\node[kp/small,text=KPBlue,fill=white,inner sep=2pt] at (1.1,6.25)
 {highlighted\\branch};
\draw[kp/leader,draw=KPBlue] (1.75,6.23)--(2.85,6.05);
\draw[decorate,decoration={brace,mirror,amplitude=4pt},black!60]
 (.1,1.57)--(7.9,1.57);
\node[kp/small] at (4.0,1.18) {terminal vertices (no children)};
\node[align=center,font=\small,text=KPBlue] at (4.05,.45)
 {Highlighted branch: $\beta_0\longrightarrow\beta_1\longrightarrow\beta_{12}$\\
  $3$ vertices $=3$ measure-to-point applications};

\draw[kp/axis] (9.1,.85)--(9.1,7.95) node[above] {$t$};
\node[kp/small,rotate=90] at (8.78,4.4) {forward time};
\draw[kp/past] (10.05,1.15) rectangle (13.25,2.05);
\node at (11.65,1.60) {$Q_\beta^-$};
\draw[kp/future] (9.95,3.2) rectangle (15.95,5.5);
\node[anchor=south east] at (15.7,3.45) {$Q_\beta^+$};
\draw[kp/past] (10.7,4.15) rectangle (13.2,5.05);
\node at (11.95,4.60) {$Q_\gamma^-$};
\draw[kp/future] (10.45,6.6) rectangle (13.55,7.5);
\node at (12.0,7.05) {$Q_\gamma^+$};
\draw[kp/step] (11.65,2.1)--(11.65,3.12);
\node[kp/small,anchor=west] at (11.92,2.64)
 {application at $\beta$};
\draw[kp/step] (11.95,5.1)--(11.95,6.53);
\node[kp/small,anchor=west] at (12.2,5.95)
 {application at $\gamma$};
\node[kp/small,align=left,anchor=west] at (13.6,4.72)
 {positive\\containment\\margin};
\draw[kp/leader] (13.65,4.18)--(13.2,4.18);
\node[align=center] at (12.75,.43)
 {$\beta\longrightarrow\gamma$ means $Q_\gamma^-\Subset Q_\beta^+$};
\end{tikzpicture}
\caption{A finite rooted directed tree of admissible ordered pairs.
In (a), every vertex carries one pair, every edge points from a parent
to a child, and double outlines mark terminal vertices. The highlighted
branch has three vertices, so positivity is propagated along it by three
measure-to-point applications. In (b), the past member of the child lies
compactly inside the future member of its parent. A lower bound on
$Q_\beta^+$ is therefore available on $Q_\gamma^-$ and can be propagated
to $Q_\gamma^+$. The boxes in (b) are schematic projections. Their
containment and time order, rather than their metric sizes, are shown.}
\label{fig:rooted-positivity-tree}
\end{figure}
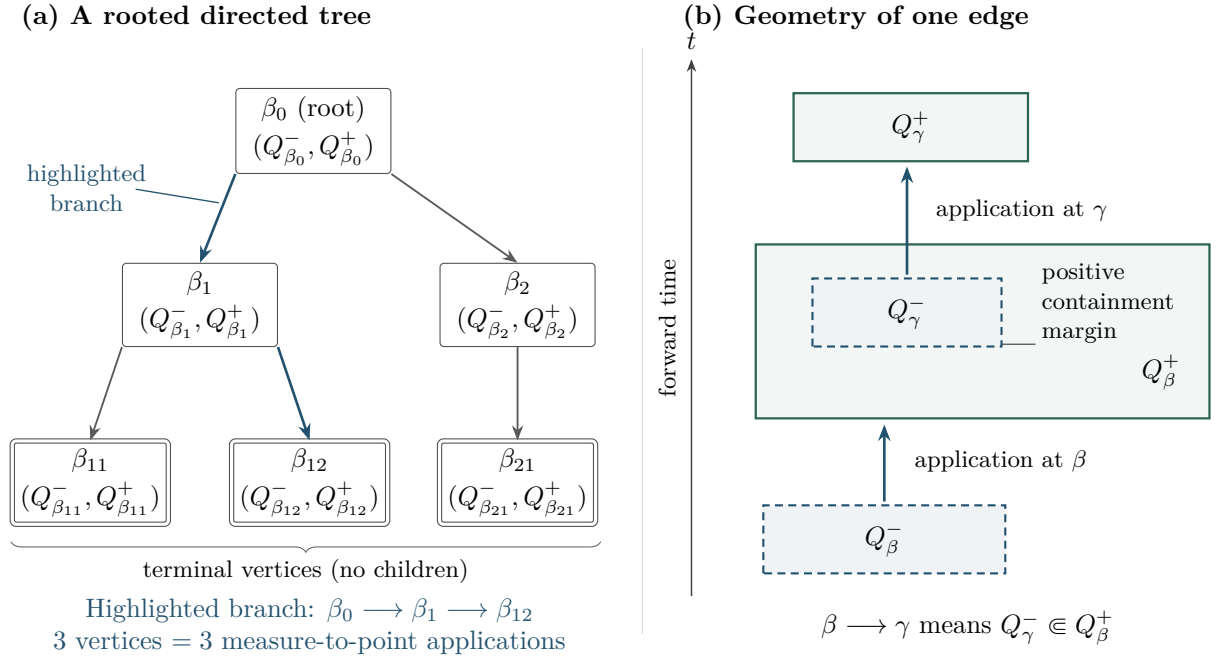
\clearpage
\begin{figure}[!ht]
\centering
\begin{tikzpicture}[x=1cm,y=1cm,font=\small]
\path[use as bounding box] (0,0) rectangle (16.4,17.25);

\begin{scope}[shift={(0,9.1)}]
\node[kp/heading] at (.05,7.82)
 {(a) From each root cell to the forward stack: (i)--(iv)};
\draw[kp/outer] (1.85,.36) rectangle (15.72,7.30);
\node[fill=white,inner sep=2pt,anchor=east] at (15.55,7.28)
 {$\mathcal U^L$};
\draw[kp/axis] (.92,.55)--(.92,7.24) node[above] {$t$};

\draw[kp/future] (2.25,2.69) rectangle (15.12,6.90);
\node[align=center] at (8.45,6.17)
 {$F^L=Q_{\beta_\alpha^0}^+$\\[2pt]terminal future member};
\draw[draw=KPGreen,line width=1pt,fill=KPGreen!15]
 (3.50,2.90) rectangle (12.50,4.90);
\node[align=center] at (8.0,3.90)
 {forward stack\\[3pt]$\overline Q_1^{\,L}(0)$};

\fill[black!3] (6.5,.9) rectangle (9.5,2.9);
\fill[KPBlue!8] (6.5,.9) rectangle (9.5,2.4);
\foreach \xx in {7.25,8.0,8.75}{
 \draw[black!40,line width=.3pt] (\xx,.9)--(\xx,2.4);
}
\foreach \yy in {1.4,1.9}{
 \draw[black!40,line width=.3pt] (6.5,\yy)--(9.5,\yy);
}
\fill[KPGold!35] (7.25,1.40) rectangle (8.0,1.90);
\draw[black!70,line width=.65pt] (6.5,.9) rectangle (9.5,2.9);
\draw[black!55,densely dotted] (6.5,2.4)--(9.5,2.4);
\draw[kp/past,fill=none] (7.10,1.24) rectangle (8.15,2.06);
\node[anchor=west] at (9.72,2.59) {$Q_1$};
\node[kp/small,anchor=east] at (6.20,2.00) {$Q_1^{\rm ear}$};
\node[kp/small,anchor=east] at (5.85,1.10) {$E_\alpha$};
\draw[kp/leader] (5.97,1.1)--(6.4,1.1)--(7.52,1.58);
\node[kp/small,anchor=west] at (9.85,1.07) {$S_\alpha^-$};
\draw[kp/leader] (9.77,1.09)--(8.6,1.09)--(8.15,1.28);

\foreach \yy/\lab in {.9/{-1},2.4/{-\frac14},2.9/{0},4.9/{L}}{
 \draw[black!60] (.83,\yy)--(1.02,\yy);
 \node[anchor=east,font=\footnotesize] at (.70,\yy) {$\lab$};
}
\draw[black!22,densely dotted] (1.03,2.9)--(3.5,2.9);
\draw[black!22,densely dotted] (1.03,4.9)--(3.5,4.9);
\node[kp/small] at (12.45,1.36)
 {$\sup\{t:(x,y,t)\in S_\alpha^-\}\leq-\tau_*<0$};
\draw[kp/step] (8.2,2.06) .. controls (12.9,2.08) and (14.15,3.15) .. (14.15,5.75);
\node[kp/small,rotate=90,fill=KPGreen!6,inner sep=2pt] at (14.42,4.18)
 {one application};
\node[kp/small] at (8.3,.03)
 {$E_\alpha\subset S_\alpha^-=Q_{\beta_\alpha^0}^-\,,\qquad
   \mu_w(S_\alpha^-)\leq C_L\mu_w(E_\alpha)$};
\end{scope}

\begin{scope}
\node[kp/heading] at (.05,8.54)
 {(b) From a small cylinder to the fixed future box: (v)};
\draw[kp/outer] (1.85,.64) rectangle (15.72,7.96);
\node[anchor=east,fill=white,inner sep=2pt] at (15.55,7.94) {$Q_{R_*}$};
\draw[kp/axis] (.92,.80)--(.92,7.72) node[above] {$t$};
\draw[black!35,densely dotted] (1.03,7.43)--(15.32,7.43);
\node[anchor=east,font=\footnotesize] at (.70,7.43) {$0$};

\draw[draw=black!55,fill=black!3,line width=.65pt]
 (4.55,.98) rectangle (10.65,1.95);
\node[kp/small,anchor=west] at (10.88,1.22) {$\mathcal Q_{\rm past}$};
\draw[draw=black!80,fill=white,line width=.7pt] (7.0,1.10) rectangle (8.2,1.78);
\fill[KPBlue!8] (7.0,1.10) rectangle (8.2,1.61);
\draw[black!35,line width=.3pt] (7.40,1.10)--(7.40,1.61);
\draw[black!35,line width=.3pt] (7.80,1.10)--(7.80,1.61);
\fill[KPGold!35] (7.40,1.10) rectangle (7.80,1.40);
\draw[kp/past,fill=none] (7.32,1.06) rectangle (7.88,1.48);
\node[kp/small,anchor=east] at (6.74,1.48) {$Q_r(P)$};
\fill (7.60,1.78) circle (1.2pt);
\node[kp/small,anchor=west] at (8.37,1.81) {$P$};

\draw[kp/future] (6.84,2.14) rectangle (8.36,2.76);
\draw[kp/past] (7.22,2.30) rectangle (7.98,2.56);
\node[kp/small,anchor=east] at (6.55,2.43) {$\mathcal R_0$};
\draw[kp/leader] (6.64,2.43)--(7.19,2.43);
\node[kp/small,anchor=west] at (8.62,2.45) {root future member};
\draw[kp/step] (7.6,1.83)--(7.6,2.10);

\draw[kp/future] (5.90,3.15) rectangle (9.30,4.19);
\node[kp/small,anchor=west] at (6.03,3.37) {$\mathcal F_0$};
\draw[kp/past] (6.88,3.65) rectangle (8.32,3.97);
\node[kp/small] at (7.60,3.81) {$\mathcal R_1$};
\draw[kp/step] (7.60,2.59)--(7.60,3.10);

\node[font=\large] at (7.60,4.53) {$\vdots$};
\node[kp/small,anchor=west] at (10.03,3.91)
 {$\rho_{j+1}=2\rho_j$};

\draw[kp/future] (4.95,4.89) rectangle (10.25,5.99);
\node[kp/small,anchor=west] at (5.08,5.15) {$\mathcal F_{N-1}$};
\draw[kp/past] (6.18,5.40) rectangle (9.02,5.79);
\node[kp/small] at (7.60,5.59) {$\mathcal R_N$};
\node[kp/small,anchor=west] at (10.63,5.59)
 {$r_*\leq\rho_N<2r_*$};

\draw[kp/future] (4.60,6.51) rectangle (10.60,7.20);
\draw[draw=KPGreen,line width=.9pt,fill=KPGreen!17]
 (5.0,6.65) rectangle (10.2,7.06);
\node at (7.60,6.85) {$Q^+$};
\draw[kp/step] (7.60,5.82)--(7.60,6.46);
\node[kp/small,anchor=west] at (10.87,6.84)
 {terminal future member};
\node[kp/small,anchor=east,align=right] at (4.22,5.58)
 {one branch\\shown};
\node[kp/small] at (8.20,.24)
 {$\mathcal R_{j+1}\Subset\mathcal F_j\,,\qquad
   N(r)\leq C_1+C_2\log^+(1/r)$};
\end{scope}
\end{tikzpicture}
\caption{The constructions in Lemma \ref{lem:positivity-tree-geometry}.
Panel (a) illustrates (i)--(iv): the early region is partitioned into
cells $E_\alpha$, with $E_\alpha\subset S_\alpha^-$ and a uniform negative
terminal-time margin. In the proof, each root may be paired directly with
one future box $F^L$ containing the entire forward stack, so its tree has
one vertex. Panel (b) illustrates one branch in (v), for $r<r_*$.
Here $\mathcal R_j:=T_{P_j',\rho_j}\mathcal R$ and
$\mathcal F_j:=T_{P_j',\rho_j}\mathcal F$, with $\rho_j=2^jr$.
Successive past boxes are compactly contained in preceding future boxes.
The terminal future member contains $Q^+$. The dashed outer envelopes
also contain the auxiliary enlargements, which are not drawn.
Both panels suppress the diffusive variables and are schematic in space
and time. Dashed inner outlines denote past members, solid inner outlines
denote future members, and the shaded gold cell is the selected root cell.}
\label{fig:positivity-tree-geometry}
\end{figure}


\begin{thebibliography}{99}

\bibitem{AtaeiNystrom2025}
A.~Ataei and K.~Nystr\"om,
\newblock \href{https://doi.org/10.1007/s11118-024-10143-7}{On fundamental solutions and Gaussian bounds for degenerate
parabolic equations with time-dependent coefficients},
\newblock \emph{Potential Anal.} \textbf{62} (2025), no.~3, 465--483.

\bibitem{AmbrosioFuscoPallara2000}
L.~Ambrosio, N.~Fusco, and D.~Pallara,
\newblock \emph{Functions of Bounded Variation and Free Discontinuity
Problems},
\newblock Oxford Mathematical Monographs, The Clarendon Press,
Oxford University Press, New York, 2000.

\bibitem{AnceschiDietertGuerandLoherMouhotRebucci2024}
F.~Anceschi, H.~Dietert, J.~Gu\'erand, A.~Loher, C.~Mouhot, and
A.~Rebucci,
\newblock \href{https://doi.org/10.5802/jep.350}{Poincar\'e inequality and quantitative De Giorgi method for
hypoelliptic operators},
\newblock \emph{J. \'Ec. polytech. Math.} \textbf{13} (2026),
1393--1417.

\bibitem{AuscherImbertNiebel2025}
P.~Auscher, C.~Imbert, and L.~Niebel,
\newblock Weak solutions to Kolmogorov-Fokker-Planck equations:
regularity, existence and uniqueness,
\newblock preprint, \href{https://arxiv.org/abs/2403.17464v3}{arXiv:\allowbreak 2403.17464v3}, 9 December 2025.

\bibitem{Baadi2026}
K.~Baadi,
\newblock \href{https://doi.org/10.1007/s00028-026-01201-1}{Degenerate parabolic equations in divergence form:
fundamental solution and Gaussian bounds},
\newblock \emph{J. Evol. Equ.} \textbf{26} (2026), Paper No.~52.

\bibitem{Bouchut2002}
F.~Bouchut,
\newblock Hypoelliptic regularity in kinetic equations,
\newblock \emph{J. Math. Pures Appl. (9)} \textbf{81} (2002), no.~11,
1135--1159.

\bibitem{BrigatiMouhot2025}
G.~Brigati and C.~Mouhot,
\newblock Introduction to quantitative De Giorgi methods,
\newblock lecture notes, \href{https://arxiv.org/abs/2510.11481}{arXiv:\allowbreak 2510.11481}, 2025.

\bibitem{CaffarelliSilvestre2007}
L.~Caffarelli and L.~Silvestre,
\newblock An extension problem related to the fractional Laplacian,
\newblock \emph{Comm. Partial Differential Equations} \textbf{32} (2007),
nos.~7--9, 1245--1260.

\bibitem{ChiarenzaSerapioni1984a}
F.~Chiarenza and R.~Serapioni,
\newblock Degenerate parabolic equations and Harnack inequality,
\newblock \emph{Ann. Mat. Pura Appl. (4)} \textbf{137} (1984), 139--162.

\bibitem{ChiarenzaSerapioni1984b}
F.~Chiarenza and R.~Serapioni,
\newblock A Harnack inequality for degenerate parabolic equations,
\newblock \emph{Comm. Partial Differential Equations} \textbf{9} (1984),
no.~8, 719--749.

\bibitem{DietertMouhotNiebelZacher2025}
H.~Dietert, C.~Mouhot, L.~Niebel, and R.~Zacher,
\newblock Critical trajectories in kinetic geometry,
\newblock preprint, \href{https://arxiv.org/abs/2508.14868}{arXiv:\allowbreak 2508.14868}, 2025.

\bibitem{DiNezzaPalatucciValdinoci2012}
E.~Di Nezza, G.~Palatucci, and E.~Valdinoci,
\newblock Hitchhiker's guide to the fractional Sobolev spaces,
\newblock \emph{Bull. Sci. Math.} \textbf{136} (2012), no.~5,
521--573.

\bibitem{DiPernaLionsMeyer1991}
R.~J. DiPerna, P.-L. Lions, and Y.~Meyer,
\newblock $L^p$ regularity of velocity averages,
\newblock \emph{Ann. Inst. H. Poincar\'e Anal. Non Lin\'eaire}
\textbf{8} (1991), no.~3--4, 271--287.

\bibitem{DongYastrzhembskiy2022}
H.~Dong and T.~Yastrzhembskiy,
\newblock \href{https://doi.org/10.1137/22M1512120}{Global $L_p$ estimates for kinetic Kolmogorov-Fokker-Planck
equations in divergence form},
\newblock \emph{SIAM J. Math. Anal.} \textbf{56} (2024), no.~1,
1223--1263.

\bibitem{FabesKenigSerapioni1982}
E.~B. Fabes, C.~E. Kenig, and R.~P. Serapioni,
\newblock The local regularity of solutions of degenerate elliptic equations,
\newblock \emph{Comm. Partial Differential Equations} \textbf{7} (1982),
no.~1, 77--116.

\bibitem{GarciaCuervaRubio1985}
J.~Garc\'ia-Cuerva and J.~L. Rubio de Francia,
\newblock \emph{Weighted Norm Inequalities and Related Topics},
\newblock North-Holland Mathematics Studies, vol.~116,
Notas de Matem\'atica, vol.~104,
North-Holland, Amsterdam, 1985.

\bibitem{GarofaloTralli2021}
N.~Garofalo and G.~Tralli,
\newblock A class of nonlocal hypoelliptic operators and their extensions,
\newblock \emph{Indiana Univ. Math. J.} \textbf{70} (2021), no.~5,
1717--1744.

\bibitem{GarofaloTralli2020}
N.~Garofalo and G.~Tralli,
\newblock Functional inequalities for a class of nonlocal hypoelliptic
equations of H\"ormander type,
\newblock \emph{Nonlinear Anal.} \textbf{193} (2020),
Paper No.~111567, 23 pp.

\bibitem{GolseImbertMouhotVasseur2019}
F.~Golse, C.~Imbert, C.~Mouhot, and A.~F. Vasseur,
\newblock Harnack inequality for kinetic Fokker-Planck equations with rough
coefficients and application to the Landau equation,
\newblock \emph{Ann. Sc. Norm. Super. Pisa Cl. Sci. (5)}
\textbf{19} (2019), 253--295.

\bibitem{Guerand2023}
J.~Gu\'erand and C.~Imbert,
\newblock Log-transform and the weak Harnack inequality for kinetic
Fokker-Planck equations,
\newblock \emph{J. Inst. Math. Jussieu} \textbf{22} (2023), no.~6,
2749--2774.

\bibitem{GuerandMouhot2022}
J.~Gu\'erand and C.~Mouhot,
\newblock Quantitative De Giorgi methods in kinetic theory,
\newblock \emph{J. \'Ec. polytech. Math.} \textbf{9} (2022),
1159--1181.

\bibitem{Hormander1967}
L.~H\"ormander,
\newblock Hypoelliptic second order differential equations,
\newblock \emph{Acta Math.} \textbf{119} (1967), 147--171.

\bibitem{Hormander1983}
L.~H\"ormander,
\newblock \emph{The Analysis of Linear Partial Differential Operators I:
Distribution Theory and Fourier Analysis},
\newblock Grundlehren der Mathematischen Wissenschaften, vol.~256,
Springer-Verlag, Berlin, 1983.

\bibitem{Imbert2026}
C.~Imbert,
\newblock De Giorgi's regularity theory for elliptic, parabolic and kinetic
equations,
\newblock lecture notes, \href{https://arxiv.org/abs/2601.15238}{arXiv:\allowbreak 2601.15238}, 2026.

\bibitem{Moser1964}
J.~Moser,
\newblock A Harnack inequality for parabolic differential equations,
\newblock \emph{Comm. Pure Appl. Math.} \textbf{17} (1964), 101--134.

\bibitem{Niebel2026}
L.~Niebel,
\newblock Transfer of regularity by kinetic mollification along critical
trajectories,
\newblock preprint, \href{https://arxiv.org/abs/2605.13582v2}{arXiv:\allowbreak 2605.13582v2}, 1 June 2026.

\bibitem{NiebelZacher2025}
L.~Niebel and R.~Zacher,
\newblock On a kinetic Poincar\'e inequality and beyond,
\newblock \emph{J. Funct. Anal.} \textbf{289} (2025), no.~1,
Paper No.~110899, 18 pp.

\bibitem{PascucciPolidoro2004}
A.~Pascucci and S.~Polidoro,
\newblock The Moser's iterative method for a class of ultraparabolic
equations,
\newblock \emph{Commun. Contemp. Math.} \textbf{6} (2004), no.~3,
395--417.

\bibitem{PolidoroRagusa2001}
S.~Polidoro and M.~A. Ragusa,
\newblock H\"older regularity for solutions of ultraparabolic equations in
divergence form,
\newblock \emph{Potential Anal.} \textbf{14} (2001), no.~4,
341--350.

\end{thebibliography}
\end{document}